\documentclass[11pt, reqno]{amsart}

\usepackage{amsmath}

\usepackage{tabu}
\usepackage{tikz}
\usepackage[margin=1.25in]{geometry}
\usepackage[matrix,arrow,curve,frame]{xy}
 \usepackage{hyperref}
\usepackage{amsthm}
\usepackage{amsfonts}
\usepackage{amssymb}
\usepackage{wasysym}
\usepackage{mathrsfs}
\usepackage{mathtools}
\usepackage{comment}

\definecolor{my-linkcolor}{rgb}{0.75,0,0}
\definecolor{my-citecolor}{rgb}{0.1,0.57,0}
\definecolor{my-urlcolor}{rgb}{0,0,0.75}
\hypersetup{
	colorlinks,
	linkcolor={my-linkcolor},
	citecolor={my-citecolor},
	urlcolor={my-urlcolor}
}

\title[Big module categories for $W$-(super)algebras]{Rigid tensor structure on big module categories\\ for some $W$-(super)algebras in type $A$}
 \author{Thomas Creutzig, Robert McRae and Jinwei Yang}
\date{}

\address{(T. C.) Department of Mathematical and Statistical Sciences, University of Alberta, Edmonton, Alberta T6G 2G1, Canada}

 \email{creutzig@ualberta.ca}

 \address{(R. M.) Yau Mathematical Sciences Center, Tsinghua University, Beijing 100084, China}
  \email{rhmcrae@tsinghua.edu.cn}
  
  \address{(J. Y.) School of Mathematical Sciences, Shanghai Jiaotong University, Shanghai 200240, China}
  \email{jinwei2@sjtu.edu.cn}


\newtheorem{thm}{Theorem}[section]
\newtheorem{cor}[thm]{Corollary}
\newtheorem{lem}[thm]{Lemma}
\newtheorem{prop}[thm]{Proposition}
\newtheorem{conj}[thm]{Conjecture}
\theoremstyle{definition}\newtheorem{defi}[thm]{Definition}
\theoremstyle{definition}\newtheorem{rem}[thm]{Remark}
\theoremstyle{definition}
\theoremstyle{definition}

\newcommand{\cB}{\mathcal{B}}
\newcommand{\cD}{\mathcal{D}}
\newcommand{\cE}{\mathcal{E}}
\newcommand{\cH}{\mathcal{H}}
\newcommand{\cY}{\mathcal{Y}}

\newcommand{\cV}{\mathcal{V}}
\newcommand{\cA}{\mathcal{A}}
\newcommand{\cR}{\mathcal{R}}
\newcommand{\cM}{\mathcal{M}}
\newcommand{\cS}{\mathcal{S}}
\newcommand{\cX}{\mathcal{X}}
\newcommand{\cF}{\mathcal{F}}
\newcommand{\cI}{\mathcal{I}}
\newcommand{\cJ}{\mathcal{J}}
\newcommand{\cL}{\mathcal{L}}
\newcommand{\cO}{\mathcal{O}}
\newcommand{\cP}{\mathcal{P}}
\newcommand{\cC}{\mathcal{C}}
\newcommand{\cG}{\mathcal{G}}
\newcommand{\cW}{\mathcal{W}}
\newcommand{\cZ}{\mathcal{Z}}
\newcommand{\cQ}{\mathcal{Q}}
\newcommand{\til}{\widetilde}

\newcommand{\CC}{\mathbb{C}}
\newcommand{\ZZ}{\mathbb{Z}}
\newcommand{\NN}{\mathbb{N}}

\newcommand{\QQ}{\mathbb{Q}}

\newcommand{\Id}{\mathrm{Id}}

\newcommand{\tens}{\boxtimes}
\newcommand{\vac}{\mathbf{1}}
\newcommand{\ind}{\mathrm{Ind}}
\newcommand{\even}{\overline{0}}
\newcommand{\odd}{\overline{1}}

\newcommand{\Oloc}{\cO^{\mathrm{loc}}}
\newcommand{\Otw}{\cO^{\mathrm{tw}}}
 \DeclareMathOperator{\im}{Im}

 \DeclareMathOperator{\rep}{Rep}
 \let\ker\relax
 \let\hom\relax
 \DeclareMathOperator{\ker}{Ker}
 \DeclareMathOperator{\hom}{Hom}

\newcommand{\repA}{\rep A}

\begin{document}
\bibliographystyle{alpha}

\numberwithin{equation}{section}

 \begin{abstract}
We establish rigid tensor category structure on finitely-generated weight modules for the subregular $W$-algebras of 
 $\mathfrak{sl}_n$ at levels $ - n + \frac{n}{n+1}$ (the $\mathcal B_{n+1}$-algebras of  Creutzig-Ridout-Wood) and at levels $- n + \frac{n+1}{n}$ (the finite cyclic orbifolds of the $\beta\gamma$-vertex algebra), as well as for their Feigin-Semikhatov dual principal $W$-superalgebras of $\mathfrak{sl}_{n|1}$. 
 These categories are neither finite nor semisimple, and in the $W$-algebra case they contain modules with infinite-dimensional conformal weight spaces and no lower bound on conformal weights. 
 We give complete lists of indecomposable projective modules in these tensor categories and fusion rules for simple modules. All these vertex operator (super)algebras are simple current extensions of singlet algebras tensored with a rank-one Heisenberg algebra, so we more generally study
 simple current extensions in direct limit completions of vertex algebraic tensor categories. Then our results for $W$-(super)algebras follow from the known ribbon category structure on modules for the singlet algebras. Our results include and generalize those of Allen-Wood on the $\beta\gamma$-vertex algebra, as well as our own on the affine vertex superalgebra of $\mathfrak{gl}_{1|1}$. Our results also include the first examples of ribbon category structure on all finitely-generated weight modules for an affine vertex algebra at a non-integral admissible level, namely for affine $\mathfrak{sl}_2$ at levels $-\frac{4}{3}$ and $-\frac{1}{2}$.
\end{abstract}

\maketitle

\tableofcontents

\section{Introduction}

Vertex operator algebras (VOAs) first appeared in the 1980s as a mathematically rigorous formulation of the symmetry algebra of a two-dimensional rational conformal field theory. 
Moore and Seiberg axiomatized 
such rational  theories \cite{MS}, and formalizing their work led to the notion of modular tensor category \cite{Tu}. 
The corresponding VOA tensor category theory was subsequently developed by Huang and Lepowsky, leading to Huang's proof that the representation category of a rational and $C_2$-cofinite vertex operator algebra is indeed a  semisimple modular tensor category \cite{Hu-rig-mod}. 

Although most VOAs are neither rational nor $C_2$-cofinite, it is still expected that nice tensor categories of modules exist for many non-rational VOAs. Such categories are of increasing interest in higher-dimensional quantum field theory. For example, 
they appear in physics in the context of new invariants of $3$-manifolds  \cite{Cheng} and as categories that underlie non-semisimple topological field theories \cite{CGP, CDGG}. However, establishing in general that suitable module categories for a VOA admit braided tensor category structure is an as yet largely unresolved problem. Further, even when such tensor categories exist, it is still difficult to show that they have nice properties, especially rigidity (existence of duals in a strong sense).

\subsection{General results}

The singlet algebras $\cM(p)$, $p \in \mathbb Z_{\geq 2}$, are good examples of VOAs with module categories that are neither finite nor semisimple \cite{Ad}. We recently established ribbon category structure and determined the fusion rules for these algebras \cite{CMY-singlet, CMY-singlet-typical}, and in the present work we use these results to construct tensor categories of modules for several series of subregular $W$-algebras  and principal $W$-superalgebras in type $A$. These (super)algebras are all simple current extensions of a singlet algebra tensored with a rank-one Heisenberg VOA $\cH$. (A simple current is an invertible object in a tensor category, and a simple current extension of a VOA $V$ is a simple vertex operator (super)algebra $A$ that is a possibly infinite direct sum of simple current $V$-modules that includes $V$ itself.)

In this work, we construct and describe a large rigid tensor (super)category of modules, with uncountably many simple objects, for any simple current extension $A$ of $\cM(p) \otimes \cH$. We first consider the category $\Oloc_A$ of finitely-generated strongly $\cH$-weight-graded $A$-modules: these modules are doubly-graded by Heisenberg and conformal weights such that all doubly-homogeneous subspaces are finite dimensional, and there is a lower bound on the conformal weights of any fixed Heisenberg weight space. We also consider the category $\Otw_A$ of finitely-generated strongly $\cH$-weight-graded twisted $A$-modules associated to an involution $\theta$; in cases where $A$ is a vertex operator superalgebra or a $\frac{1}{2}\ZZ$-graded vertex operator algebra, these are the modules of the Ramond sector in physics terminology.

 Setting $\cO_A = \Oloc_A \oplus \Otw_A$, our first main result, Theorem \ref{thm:main_thm}, states that $\Oloc_A$ is a rigid braided tensor (super)category and that $\cO_A$ is a rigid braided $\ZZ/2\ZZ$-crossed tensor (super)category. 
If $A$ is $\ZZ$-graded by conformal weights, then $e^{2\pi i L(0)}$ defines a ribbon twist on $\Oloc_A$, so that $\Oloc_A$ is a ribbon tensor (super)category. We then determine concrete properties of the categories $\Oloc_A$ and $\cO_A$; specifically, we obtain:
 \begin{itemize}
\item The classification of simple modules (Theorem \ref{thm:simple_A-module_classification}),

\item The existence of projective covers for all simple modules, together with their Loewy diagrams (Theorem \ref{thm:gen_proj_covers}), and

\item Formulas for the tensor product of any pair of simple modules (Theorem \ref{thm:general_fusion_rules}).

\end{itemize}
In cases where $\cO_A$ contains modules with infinite-dimensional conformal weight spaces, we also prove finiteness and semisimplicity of the categories of highest-weight modules in $\cO_A$, of grading-restricted generalized $A$-modules, and of $C_1$-cofinite $A$-modules (Theorems \ref{thm:hw_cat_ss}, \ref{thm:grad_rest_ss}, and \ref{thm:C_1-cofin}). In cases where all objects of $\cO_A$ have finite-dimensional weight spaces, we show that the category of $C_1$-cofinite $A$-modules is a rigid braided tensor category that contains $\Oloc_A$.

To prove these results, we combine the theory of vertex operator (super)algebra extensions developed in \cite{HKL, CKL, CKM-exts, CMY-completions} with the detailed structure of the ribbon category of $C_1$-cofinite $\cM(p)$-modules derived in \cite{CMY-singlet, CMY-singlet-typical}. One subtle point is that we need to show that $\cO_A$ is the same as the category of finite-length strongly $\cH$-weight-graded (twisted) $A$-modules; here the proof uses some results from \cite{CKLR} on the representation theory of VOAs having a Heisenberg subalgebra. For our semisimplicity results, we need to prove the absence of non-trivial self-extensions of certain simple $\cM(p)$-modules (see the proof of Theorem \ref{thm:grad_rest_ss} for details), and the classification of $C_1$-cofinite $A$-modules uses tensor product formulas to show that most simple $A$-modules are not $C_1$-cofinite in cases where $\cO_A$ contains modules with infinite-dimensional conformal weight spaces.

Examples of VOAs covered by our theorems include two series of subregular $W$-algebras associated to $\mathfrak{sl}_n$ at certain admissible levels, namely the $\cB_p$-algebras introduced in \cite{CRW} and the finite cyclic orbifolds of the $\beta\gamma$ (or symplectic boson) VOA. These two series of VOAs include the $\beta\gamma$ VOA itself, for which the existence of rigid tensor category structure has already been recently proved by Allen and Wood \cite{AW}, as well as the simple affine VOAs associated to $\mathfrak{sl}_2$ at levels $-\frac{4}{3}$ and $-\frac{1}{2}$, and Bershadsky-Polyakov algebras at levels $-\frac{9}{4}$ and $-\frac{5}{3}$; see Section \ref{subsec:special_cases} for a detailed comparison of our results with previous results and conjectures on these algebras. Our results also cover two series of principal $W$-superalgebras associated to $\mathfrak{sl}_{n\vert 1}$ at certain admissible levels; see Section \ref{sec:super}.

\subsection{Big module categories for \texorpdfstring{$W$}{W}-(super)algebras}

Let us put our results into context. Vertex operator algebras associated to affine Lie algebras, as well as their quantum Hamiltonian reductions, called $W$-algebras, are some of the richest families of examples of VOAs. There are some $W$-algebras that have modular tensor categories of representations: especially, for any simple Lie algebra $\mathfrak{g}$ and level $k\in\QQ$ which is admissible in the sense of Kac and Wakimoto \cite{KW}, there is an ``exceptional'' simple $W$-algebra associated to $\mathfrak{g}$ at level $k$ which is rational and $C_2$-cofinite \cite{Ar-C2, Ar-rat1, Ar-rat2, McR-rat}. These exceptional $W$-algebras include the simple affine VOAs themselves at positive integer levels. In general, we expect non-exceptional $W$-algebras and affine VOAs at admissible levels to also admit nice tensor categories of modules. Indeed, the category of ordinary (that is, grading-restricted) modules for an affine VOA at admissible level is a semisimple braided tensor category \cite{CHY}, and in many cases this category is also known to be rigid \cite{Cr2}. 

The problem is that for non-integral admissible levels, very few modules of the simple affine VOA  are ordinary, and one thus wants to understand the ``big'' category that includes all finitely-generated weight modules. This category admits modules that lack finiteness conditions such as finite-dimensional conformal weight spaces and lower bounds on conformal weights. As a result, standard techniques of VOA theory do not apply well to such categories. Developing new ideas and techniques for understanding categories of big modules for VOAs is a major challenge, illustrated well by the $\beta\gamma$ VOA. Though this VOA has been known and studied for decades, the complete description in \cite{AW} of the rigid tensor structure on its category of weight modules is quite recent. Our results in this paper show that VOA extension theory is one viable approach, at least as long as the affine VOA of interest contains a well-enough-understood subalgebra.

Recently, there has been good progress in understanding the big category of weight modules for affine VOAs at admissible levels \cite{Ad_sl2_-4/3, Ad-sl2_osp1|2, ACG, CR1, CR2, CRR, KR1, KR2, KRW, Ri1, Ri2, RW}, but complete understanding of its abelian category structure (such as classification of projective modules) remains open, as well as existence of tensor category structure and properties like fusion rules and rigidity. Our work here is the first to successfully address these points in the examples of the affine VOA of $\mathfrak{sl}_2$ at the admissible levels $-\frac{1}{2}$ and $-\frac{4}{3}$. (We note that a classification of indecomposable modules for affine VOAs of $\mathfrak{sl}_2$ at all admissible levels, as well as the existence of braided tensor category structure, is work in progress that will appear soon.)

Among $W$-algebras, the best understood class is principal $W$-algebras, but these VOAs lack a Heisenberg subalgebra and thus do not have interesting weight modules with infinite-dimensional conformal weight spaces. The next most studied series of $W$-algebras are subregular $W$-algebras in type $A$. The subregular $W$-algebra of $\mathfrak{sl}_2$ is just the affine VOA, and the subregular $W$-algebra of $\mathfrak{sl}_3$ is the Bershadsky-Polyakov algebra \cite{Be, Po}. At admissible levels with denominator $n-1$, the simple subregular $W$-algebra of $\mathfrak{sl}_n$ is rational, and these examples are fairly well understood \cite{Ar-BP, ACL, Ar-rat2, CL-sub, CL-trialities}. At general admissible levels the abelian structure of the ``big'' category of weight modules has been studied \cite{AK1, AK2, AKR, Fe, FR, FKR}, but no general classification of indecomposable or projective modules is yet achieved, let alone any results on tensor category structure.  Thus our work here provides the first examples of subregular $W$-algebras for which these problems are solved.

We specifically construct rigid tensor categories of weight modules for the simple subregular $W$-algebra of $\mathfrak{sl}_n$ at levels $-n+\frac{n}{n+1}$ (the $\cB_{p}$-algebra of \cite{CRW, ACKR, ACGY} at $p=n+1$) and $-n+\frac{n+1}{n}$ (the $\ZZ/n\ZZ$-orbifold of the $\beta\gamma$ VOA).  We in particular classify the simple strongly $\cH$-weight-graded (twisted) modules for these subregular $W$-algebras (Theorems \ref{thm:Bp_mod_class} and \ref{thm:B2m_simple_objects}), describe their projective covers (Theorems \ref{thm:Bp_proj_modules} and \ref{thm:B2m_proj_modules}), and compute their tensor products (Theorems \ref{thm:Bp_tensor_products} and \ref{thm:B2m_tens_prods}). We also prove semisimplicity of all highest-weight (twisted) modules and all grading-restricted generalized (twisted) modules for these subregular $W$-algebras (this can be compared to Arakawa's semisimplicity of category $\cO$ for affine VOAs at admissible levels \cite{Ar-cat-O}, and the semisimplicity of highest-weight modules for Bershadsky-Polyakov algebras at admissible levels \cite[Theorem 4.10]{FKR}). The rigid tensor structure on the categories of weight modules for the $\cB_{p}$-algebras had previously been predicted in \cite{ACKR} using a conjectural correspondence between the singlet VOA and the unrolled restricted quantum group of $\mathfrak{sl}_2$; thus now that we have a comprehensive understanding of the singlet algebras \cite{CMY-singlet, CMY-singlet-typical}, the tensor category results in \cite{ACKR} are now rigorous theorems about $\cB_p$-module categories.

Compared with affine VOAs and their $W$-algebras, even less is known about the representation theory of affine vertex operator superalgebras and their $W$-superalgebras. Some of the best-understood examples are the affine superalgebras associated to $\mathfrak{osp}_{1\vert 2n}$, since $\mathfrak{osp}_{1|2n}$ behaves like a simple Lie algebra in many ways. At many rational levels, the category of ordinary modules of the affine $\mathfrak{osp}_{1\vert 2n}$ VOA is a semisimple rigid braided tensor category \cite{CGL}. The case of affine $\mathfrak{osp}_{1|2}$ is quite similar to affine $\mathfrak{sl}_2$, and a fair amount is known about its representation theory at admissible levels \cite{KR1, CKLR2, RSW, Ad-sl2_osp1|2}.

 Recently, we established ribbon supercategory structure on $C_1$-cofinite modules for the affine vertex operator superalgebra of $\mathfrak{gl}_{1|1}$ \cite{CMY3}. This affine vertex operator superalgebra is closely related to the $\beta\gamma$ VOA; specifically, it is a $\mathbb C^\times$-orbifold of the $\beta\gamma$ VOA tensored with a pair of free fermions. That is, this tensor product vertex operator superalgebra is a certain simple current extension of the affine $\mathfrak{gl}_{1\vert 1}$ vertex operator superalgebra. This relation is the first instance of Feigin-Semikhatov duality \cite{FS, CGN}, which itself is one of the simplest cases of triality of $W$-algebras  \cite{CL-trialities, CL-ortho} (triality of vertex algebras at the corner in physics \cite{GR}). Feigin-Semikhatov duality in type $A$ says that a certain Heisenberg coset of the subregular $W$-algebra of $\mathfrak{sl}_n$ at level $k$ tensored with a pair of free fermions is the principal $W$-superalgebra of $\mathfrak{sl}_{n|1}$ at level $\ell$, where the levels satisfy $(k+n)(\ell+n-1) =1$. By a similar construction, the subregular $W$-algebra is a Heisenberg coset of the $W$-superalgebra tensored with the lattice vertex operator superalgebra $V_{\sqrt{-1}\ZZ}$. 
 
 The representation categories of the two dual (super)algebras are closely related \cite{CGNS}, and in Section \ref{sec:super} of the present work, we discuss the Feigin-Semikhatov duals of the $\cB_p$-algebras and the cyclic orbifolds of the $\beta\gamma$ VOA. These are the principal $W$-superalgebras of $\mathfrak{sl}_{n|1}$ at levels $- (n -1) + \frac{n+1}{n}$ and at levels $- (n -1) + \frac{n}{n+1}$. The tensor categories for these principal $W$-superalgebras look very similar (though not quite identical) to those for the subregular $W$-algebras, but the conformal weight gradings of the modules are completely different. In particular, we show that every finitely-generated weight module for these $W$-superalgebras is $C_1$-cofinite, and in particular has finite-dimensional conformal weight spaces and a lower bound on conformal weights. Moreover, the category of $C_1$-cofinite modules for any of these $W$-superalgebras is a rigid braided tensor supercategory which contains weight modules as a tensor subcategory; this is similar to what we found for affine $\mathfrak{gl}_{1\vert 1}$ in \cite{CMY3}.

\subsection{Quasi-lisse vertex operator algebras}

The subregular $W$-algebras of $\mathfrak{sl}_n$ studied in this paper are examples of quasi-lisse vertex algebras in the sense of Arakawa and Kawasetsu \cite{AK}. Quasi-lisse vertex algebras are vertex algebras whose associated varieties have only finitely many symplectic leaves; they generalize the class of $C_2$-cofinite (also called lisse \cite{Ar-assoc-var}) vertex algebras, whose associated varieties are just points. Like $C_2$-cofinite vertex algebras, a quasi-lisse vertex algebra has only finitely many simple ordinary modules (where ordinary means grading-restricted in the sense of having finite-dimensional conformal weight spaces and a lower bound on conformal weights). Beyond $C_2$-cofinite vertex algebras, quasi-lisse vertex algebras include affine vertex algebras at admissible levels and $W$-algebras at non-degenerate admissible levels \cite{AK}.

One might hope that the category of grading-restricted generalized modules for a quasi-lisse VOA would admit the braided tensor category structure of \cite{HLZ1}-\cite{HLZ8}, since this is true for $C_2$-cofinite VOAs \cite{Hu-C2} and affine VOAs at admissible levels \cite{CHY}.  This was formulated as Conjecture 1 in \cite{ACF}. However, in this paper we show that this conjecture is not true in general: Theorem \ref{thm:Bp_grad-rest} enumerates the simple ordinary modules of the $\mathcal B_p$-algebra, and Theorem \ref{thm:Bp_tensor_products} then shows that they do not close under tensor products.

Let us use this opportunity to revise \cite[Conjecture 1]{ACF}. Since a quasi-lisse vertex algebra $V$ has only finitely many simple ordinary modules, one can easily show that all grading-restricted generalized $V$-modules have finite length. In particular, all $C_1$-cofinite $V$-modules have finite length. Miyamoto proved that $C_1$-cofinite modules do close under tensor products \cite{Miy}, and then it was shown in \cite{CJORY, CY} (see also \cite[Theorem 2.3]{McR-cosets}) that the category of $C_1$-cofinite $V$-modules is a tensor category if it is closed under contragredient duals.

\begin{conj}
If $V$ is a simple quasi-lisse vertex operator algebra of CFT type that is its own contragredient dual, then the category of $C_1$-cofinite $V$-modules is closed under contragredient duals and thus admits the vertex algebraic braided tensor category structure of \cite{HLZ1}-\cite{HLZ8}.
\end{conj}

In retrospect, it is perhaps not natural to expect too much of the category of grading-restricted generalized $V$-modules in general, since a VOA may admit many different conformal vectors (as is the case for subregular $W$-algebras like $\cB_p$), in which case the category of grading-restricted generalized modules may depend heavily on the choice of conformal vector. In contrast, if $V$ has CFT type, then a $V$-module is $C_1$-cofinite if and only if it is finitely strongly generated over $V$ \cite[Lemma 3.1.6]{Ar-assoc-var}, and this latter condition does not depend on the choice of conformal vector. Thus it is much more natural to consider the $C_1$-cofinite category; however, the $C_1$-cofinite category may be much smaller than the category of grading-restricted generalized modules. For example, Theorem \ref{thm:Bp_grad-rest} below shows that, with respect to the conformal vector we have chosen, $\cB_p$ has on the order of $\frac{1}{2} p^2$ simple grading-restricted modules but just one simple $C_1$-cofinite module (namely $\cB_p$ itself).

Another interesting feature of the quasi-lisse $W$-algebras studied in this paper is the semisimplicity of their categories of grading-restricted generalized modules. This result parallels Arakawa's semisimplicity of category $\cO$ for (quasi-lisse) simple affine VOAs at admissible levels \cite{Ar-cat-O} and the rationality of ``exceptional'' $C_2$-cofinite $W$-algebras at admissible levels \cite{Ar-rat1, Ar-rat2, McR-rat}. Thus one might be tempted to conjecture that the category of grading-restricted generalized modules is semisimple for any quasi-lisse simple affine $W$-algebra at an admissible level, though proving such a general conjecture would surely be a deep problem. There are also examples of quasi-lisse affine VOAs and $W$-algebras at non-admissible levels, some of which are even $C_2$-cofinite \cite{Ka, AM1, AM2}, but it seems difficult to predict when such quasi-lisse algebras will have semisimple grading-restricted representation theory. An interesting case is the simple affine VOA of $\mathfrak{so}_8$ at level $-2$, for which the grading-restricted module category is semisimple with one simple object (namely, the VOA itself), but the larger highest-weight category $\cO$ is not semisimple \cite{KR2}.
\vspace{3mm}
 
 \noindent{\bf Acknowledgments.} 
 We thank Dra\v{z}en Adamovi\'{c}, Tomoyuki Arakawa, Shashank Kanade, Kazuya Kawasetsu, and Shigenori Nakatsuka for comments and discussions on related ideas. 
 TC is supported by  NSERC Grant Number RES0048511. RM is supported by a startup grant from Tsinghua University. JY is supported by a startup grant from Shanghai Jiao Tong University.

\section{Preliminaries}

\subsection{Vertex operator superalgebras and tensor categories}\label{sec:prelim}

We use the definition of vertex operator superalgebra discussed in \cite[Section 3.1]{CKM-exts}. In particular, a vertex operator superalgebra $V$ has a parity $\ZZ/2\ZZ$-grading $V=V^{\even}\oplus V^{\odd}$ and a conformal weight $\frac{1}{2}\ZZ$-grading $V=\bigoplus_{n\in\frac{1}{2}\ZZ} V_{(n)}$ such that $V=\bigoplus_{i\in\ZZ/2\ZZ, n\in\frac{1}{2}\ZZ} V^i\cap V_{(n)}$. We do not require that $V^{\overline{i}}=\bigoplus_{n\in\frac{i}{2}+\ZZ} V_{(n)}$ for $i=0,1$ in general. Thus a vertex operator superalgebra may be $\ZZ$-graded by conformal weights, and a vertex operator algebra (which is a vertex operator superalgebra such that $V^{\odd}=0$) may be $\frac{1}{2}\ZZ$-graded by conformal weights. We use
\begin{align*}
Y: V\otimes V & \rightarrow V((x))\nonumber\\
 u\otimes v & \mapsto Y(u,x)v=\sum_{n\in\ZZ} u_n v\,x^{-n-1}
\end{align*}
to denote the vertex operator of $V$ (an even linear map), and we use $\vac\in V^{\even}\cap V_{(0)}$ to denote the vacuum vector and $\omega\in V^{\even}\cap V_{(2)}$ to denote the conformal vector. For a parity-homogeneous vector $v\in V$, we use $\vert v\vert\in\ZZ/2\ZZ$ to denote its parity.

In this paper, we will consider vertex operator superalgebras which are \textit{strongly $\CC$-graded} in the sense of \cite[Definition 2.23]{HLZ1}. This means there is an additional $\CC$-grading $V=\bigoplus_{\lambda\in\CC} V^{(\lambda)}$ satisfying the following conditions:
\begin{enumerate}
\item All gradings of $V$ are compatible: $V=\bigoplus_{i\in\ZZ/2\ZZ, n\in\frac{1}{2}\ZZ,\lambda\in\CC} V^i\cap V_{(n)}\cap V^{(\lambda)}$.

\item For all $n\in\frac{1}{2}\ZZ$ and $\lambda\in\CC$, $V^{(\lambda)}_{(n)}:=V^{(\lambda)}\cap V_{(n)}$ is finite dimensional, and for all $\lambda\in\CC$, $V^{(\lambda)}_{(n)}=0$ for all sufficiently negative $n\in\frac{1}{2}\ZZ$.

\item $\vac,\omega\in V^{(0)}$.

\item For all $\lambda,\mu\in\CC$, if $u\in V^{(\lambda)}$ and $v\in V^{(\mu)}$, then  $Y(u,x)v\subseteq V^{(\lambda+\mu)}((x))$.
\end{enumerate}
In our examples, the $\CC$-grading  will be the eigenvalue grading for the operator $h(0)=\mathrm{Res}_x\,Y(h,x)$ for some vector $h\in V^{\even}\cap V_{(1)}$; such a grading satisfies the fourth condition above in particular due to the commutator formula
\begin{equation*}
[h(0), Y(v,x)] =Y(h(0)v,x)
\end{equation*}
for $v\in V$.

Given a strongly $\CC$-graded vertex operator superalgebra $V$, we define the notion of strongly $\CC$-graded $V$-module following \cite[Definition 2.25]{HLZ1}:
\begin{defi}\label{def:strongly_graded}
A \textit{strongly $\CC$-graded $V$-module} is a $\ZZ/2\ZZ\times\CC\times\CC$-graded vector space $W=W^{\even}\oplus W^{\odd} =\bigoplus_{h\in\CC} W_{[h]}=\bigoplus_{\lambda\in\CC} W^{(\lambda)}$ equipped with a vertex operator map
\begin{align*}
Y_W: V\otimes W & \rightarrow W((x))\nonumber\\
 v\otimes w & \mapsto Y_W(v,x)w=\sum_{n\in\ZZ} v_n w\,x^{-n-1}
\end{align*}
satisfying the following properties:
\begin{enumerate}

\item The \textit{grading restriction conditions}: For any $h,\lambda\in\CC$, $W^{(\lambda)}_{[h]}:= W^{(\lambda)}\cap W_{[h]}$ is finite dimensional, and for any $h,\lambda\in\CC$, $W^{(\lambda)}_{[h-n]}=0$ for all $n\in\NN$ sufficiently positive.

\item Homogeneity of the vertex operator: For $v\in V^{\overline{i}}\cap V^{(\lambda)}$ and $w\in W^{\overline{j}}\cap W^{(\mu)}$ where $i,j\in\lbrace 0,1\rbrace$ and $\lambda,\mu\in\CC$, 
\begin{equation*}
Y_W(v,x)w\in (W^{\overline{i+j}}\cap W^{(\lambda+\mu)})((x)).
\end{equation*}

\item The \textit{vacuum property}: $Y_W(\vac,x) =\Id_{W}$.

\item The \textit{Jacobi identity}: For parity-homogeneous $u,v\in V$,
\begin{align*}
x_0^{-1}\delta\left(\frac{x_1-x_2}{x_0}\right) Y_W(u,x_1)Y_W(v,x_2) & - (-1)^{\vert u\vert\vert v\vert} x_0^{-1}\delta\left(\frac{x_2-x_1}{-x_0}\right) Y_W(v,x_2)Y_W(u,x_1)\nonumber\\
& = x_2^{-1}\delta\left(\frac{x_1-x_0}{x_2}\right) Y_W(Y(u,x_0)v,x_2).
\end{align*}

\item Virasoro algebra properties: Setting $Y_W(\omega,x)=\sum_{n\in\ZZ} L(n)\,x^{-n-2}$, 
\begin{equation*}
[L(m),L(n)] =(m-n)L(m+n)+\frac{m^3-m}{12}\delta_{m+n,0}c
\end{equation*}
for $m,n\in\ZZ$, where $c$ is the central charge of $V$, and for any $h\in\CC$, $W_{[h]}$ is the generalized $L(0)$-eigenspace with eigenvalue $h$.

\item The \textit{$L(-1)$-derivative property}: $\frac{d}{dx}Y_W(v,x)=Y_W(L(-1)v,x)$ for any $v\in V$.

\end{enumerate}
\end{defi}

Again in our examples, the second $\CC$-grading $W=\bigoplus_{\lambda\in\CC} W^{(\lambda)}$ of a strongly $\CC$-graded $V$-module $W$ will be given by eigenvalues of $h(0)=\mathrm{Res}_x\,Y_W(h,x)$ for some $h\in V^{\overline{0}}\cap V_{(1)}$. Again, the vertex operator $Y_W$ is homogeneous with respect to such a grading due to the commutator formula
\begin{equation*}
[h(0), Y_W(v,x)] = Y_W(h(0)v,x),
\end{equation*}
which follows by taking $u=h$ in the Jacobi identity and extracting the coefficient of $x_0^{-1} x_1^{-1}$.

\begin{rem}
A \textit{grading-restricted generalized $V$-module} is a strongly $\CC$-graded $V$-module for which the second $\CC$-grading is trivial, that is, $W=W^{(0)}$. If we drop the second $\CC$-grading $W=\bigoplus_{\lambda\in\CC} W^{(\lambda)}$ and the grading restriction conditions from the definition of strongly $\CC$-graded $V$-module, then we get the notion of \textit{generalized $V$-module}, which is ``generalized'' in the sense that it is the direct sum of generalized $L(0)$-eigenspaces.
\end{rem}

As in \cite[Definition 2.32]{HLZ1} the graded dual of a strongly $\CC$-graded $V$-module $W$ is
\begin{equation*}
W'=\bigoplus_{h,\lambda\in\CC} (W')_{[h]}^{(\lambda)},\qquad\text{where}\quad (W')_{[h]}^{(\lambda)}=(W^{(-\lambda)}_{[h]})^*.
\end{equation*}
Because $V$ is $\frac{1}{2}\ZZ$-graded by conformal weights in general, there are two natural strongly $\CC$-graded $V$-module structures on $W'$ (see \cite[Section 3.1]{CKM-exts}):
\begin{equation}\label{eqn:contra_action}
\left\langle Y^{\pm}_{W'}(v,x)w',w\right\rangle =(-1)^{\vert v\vert\vert w'\vert}\left\langle w', Y_W(e^{xL(1)}e^{\pm\pi i L(0)} x^{-2L(0)} v, x^{-1})w\right\rangle
\end{equation}
for $w\in W$ and parity-homogeneous $v\in V$, $w'\in W'$. Since the $L(0)$-conjugation formula
\begin{equation*}
e^{2\pi i L(0)} Y_W(v,x)w=Y_W(e^{2\pi i L(0)} v,x)e^{2\pi i L(0)} w
\end{equation*}
implies that 
\begin{equation*}
e^{2\pi i L(0)}: (W', Y_{W'}^+) \longrightarrow (W', Y_{W'}^-)
\end{equation*}
is a $V$-module isomorphism, we call either $V$-module structure $W'_+ := (W',Y_{W'}^+)$ or $W'_- := (W',Y_{W'}^-)$ the \textit{contragredient} of $W$. Moreover, as in \cite[Proposition 5.3.1]{FHL} and \cite[Theorem 2.34]{HLZ1}, the natural linear isomorphism $W\rightarrow W''$ is an isomorphism of strongly $\CC$-graded $V$-modules; more specifically, it gives isomorphisms $W\rightarrow (W'_+)'_-$ and $W\rightarrow (W'_-)'_+$.

Now we recall the definition of intertwining operator among strongly $\CC$-graded $V$-modules from \cite[Definition 3.10 and 3.14]{HLZ2} and \cite[Definition 3.7]{CKM-exts}:
\begin{defi}
Given strongly $\CC$-graded $V$-modules $W_1$, $W_2$, $W_3$, a \textit{parity-homogeneous intertwining operator} of type $\binom{W_3}{W_1\,W_2}$ is a parity-homogeneous linear map
\begin{align*}
\cY: W_1\otimes W_2 & \rightarrow W_3[\log x]\lbrace x\rbrace\nonumber\\
w_1\otimes w_2 & \mapsto \cY(w_1,x)w_2=\sum_{h\in\CC}\sum_{k\in\NN} (w_1)_{h;k} w_2\,x^{-h-1}(\log x)^k
\end{align*}
satisfying the following properties:
\begin{enumerate}
\item \textit{Lower truncation}: For $w_1\in W_1$, $w_2\in W_2$, and $h\in\CC$, $(w_1)_{h-n,k} w_2 = 0$
for all $n\in\NN$ sufficiently large, independently of $k$.

\item The \textit{Jacobi identity}: For parity-homogeneous $v\in V$ and $w_1\in W_1$,
\begin{align*}
(-1)^{\vert\cY\vert\vert v\vert} & x_0^{-1}  \delta\left(\frac{x_1-x_2}{x_0}\right) Y_{W_3}(v,x_1)\cY(w_1,x_2)\nonumber\\
&\hspace{3em} - (-1)^{\vert v\vert\vert w_1\vert} x_0^{-1}\delta\left(\frac{x_2-x_1}{-x_0}\right) \cY(w_1,x_2)Y_{W_2}(v,x_1)\nonumber\\
& = x_2^{-1}\delta\left(\frac{x_1-x_0}{x_2}\right) \cY(Y_{W_1}(v,x_0)w_1,x_2).
\end{align*}

\item The \textit{$L(-1)$-derivative property}: $\frac{d}{dx}\cY(w_1,x)=\cY(L(-1)w,x)$ for any $w_1\in W_1$.
\end{enumerate}
An \textit{intertwining operator} of type $\binom{W_3}{W_1\,W_2}$ is any sum of an even and an odd intertwining operator of type $\binom{W_3}{W_1\,W_2}$. An intertwining operator of type $\binom{W_3}{W_1\,W_2}$ is \textit{grading-compatible} if $(w_1)_{h,k}w_2\in W_3^{(\lambda+\mu)}$ for any $w_1\in W_1^{(\lambda)}$, $w_2\in W_2^{(\mu)}$, $h\in\CC$, and $k\in\NN$.
\end{defi}

\begin{rem}
When the second $\CC$-gradings on $W_1$, $W_2$, and $W_3$ are given by eigenvalues of the zero-mode $h(0)$ for some $h\in V^{\even}\cap V_{(1)}$, all intertwining operators of type $\binom{W_3}{W_1\,W_2}$ are grading-compatible thanks to the commutator formula
\begin{equation}\label{eqn:Y_compat_with_h(0)}
[h(0),\cY(w_1,x)] =\cY(h(0)w_1,x)
\end{equation}
for all $w_1\in W_1$, which follows easily from the Jacobi identity and the evenness of $h$.
\end{rem}

For strongly $\CC$-graded $V$-modules $W_1$, $W_2$, and $W_3$, we use $\cI\binom{W_3}{W_1\,W_2}$ to denote the vector superspace of grading-compatible intertwining operators of type $\binom{W_3}{W_1\,W_2}$. As in \cite[Proposition 3.44]{HLZ2}, there is for any $r\in\ZZ$ an even linear isomorphism
\begin{equation*}
\Omega_r: \cI\binom{W_3}{W_1\,W_2}\longrightarrow\cI\binom{W_3}{W_2\,W_1}
\end{equation*}
defined by
\begin{equation*}
\Omega_r(\cY)(w_2,x)w_1= (-1)^{\vert w_1\vert\vert w_2\vert}e^{xL(-1)}\cY(w_1,e^{(2r+1)\pi i}x)w_2
\end{equation*}
for parity-homogeneous $w_1\in W_1$, $w_2\in W_2$. Also, adapting the proof of \cite[Proposition 3.46]{HLZ2} to the case that $V$ is $\frac{1}{2}\ZZ$-graded by conformal weights, there is for any $r\in \ZZ$ an even linear map
\begin{equation*}
A_r:\cI\binom{W_3}{W_1\,W_2} \longrightarrow\cI\binom{(W_2)'_\pm}{W_1\,(W_3)'_\pm},
\end{equation*}
where the signs are both $+$ for $r$ even and both $-$ for $r$ odd. The map $A_r$ is defined by
\begin{equation*}
\langle A_r(\cY)(w_1,x)w_3', w_2\rangle =(-1)^{(\vert\cY\vert+\vert w_1\vert)\vert w_3'\vert}\big\langle w_3',\cY(e^{xL(1)}(e^{(2r+1)\pi i}x^{-2})^{L(0)} w_1,x^{-1})w_2\big\rangle
\end{equation*}
for parity-homogeneous $\cY$, $w_1\in W_1$, $w_3'\in W_3'$, and $w_2\in W_2$. Each $A_r$ is an isomorphism because $A_{-r-1}(A_r(\cY))$ is identified with $\cY$ under the $V$-module isomorphisms $((W_2)'_{\pm})'_{\mp}\cong W_2$ and $((W_3)'_\pm)'_\mp\cong W_3$.

We now recall elements of the vertex algebraic tensor category theory of \cite{HLZ1}-\cite{HLZ8}; see also the discussion of the superalgebra generality in \cite[Section 3]{CKM-exts}. Let $\cC$ be a category of strongly $\CC$-graded $V$-modules. We assume that all morphisms in the category preserve both $\CC$-gradings of $V$-modules; this is automatic for the conformal weight grading  because $V$-module homomorphisms commute with $L(0)$, and  this will also be automatic for the second weight grading if it is given by $h(0)$-eigenvalues for some $h\in V^{\even}\cap V_{(1)}$. Actually, $\cC$ is a supercategory in the sense that morphisms form vector superspaces, since every homomorphism between $V$-modules is uniquely the sum of an even and an odd homomorphism.

\begin{defi}
Given two strongly-graded $V$-modules $W_1$ and $W_2$ which are objects of $\cC$, a \textit{tensor product} of $W_1$ and $W_2$ in $\cC$ is an object $W_1\tens W_2$ of $\cC$ equipped with an even grading-compatible intertwining operator $\cY_\tens$ of type $\binom{W_1\tens W_2}{W_1\,W_2}$ such that the following universal property holds: for any object $W_3$ of $\cC$ equipped with a grading-compatible intertwining operator $\cY$ of type $\binom{W_3}{W_1\,W_2}$, there is a unique $V$-module homomorphism $f_\cY: W_1\tens W_2\rightarrow W_3$ such that $f_\cY\circ\cY_\tens =\cY$.
\end{defi}

If a tensor product of any pair of objects in $\cC$ exists, and if certain further conditions (which we will not need to specify here) hold, then $\cC$ has the structure of a braided tensor category with unit object $V$. See \cite{HLZ8} or the exposition in \cite[Section 3.3]{CKM-exts} for descriptions of the unit isomorphisms $l_W: V\tens W\rightarrow W$ and $r_W: W\tens V\rightarrow W$, the associativity isomorphisms $\cA_{W_1,W_2,W_3}: W_1\tens(W_2\tens W_3)\rightarrow (W_1\tens W_2)\tens W_3$, and the braiding isomorphisms $\cR_{W_1,W_2}: W_1\tens W_2\rightarrow W_2\tens W_1$. In particular, the right unit and braiding isomorphisms come from the intertwining operator isomorphisms $\Omega_r$, but we will not need these details of the vertex algebraic braided tensor category structure in this paper.

Assuming that our (super)category $\cC$ of strongly $\CC$-graded $V$-modules is closed under tensor products and contragredient modules, and that $V$ is isomorphic to its contragredient, the intertwining operator isomorphisms $A_r$ combined with the universal property of tensor products yield a sequence of even natural linear isomorphisms
\begin{align*}
\hom_V(X\tens W,V) & \cong\cI\binom{V}{X\,W}\cong\binom{W'}{X\,V'}\nonumber\\
&\cong\cI\binom{W'}{X\,V}\cong\hom_V(X\tens V,W')\cong\hom_V(X,W').
\end{align*}
From this, it is easy to see that contragredients enjoy the following universal property: there is an even evaluation morphism $e_W: W'\tens W\rightarrow V$ in $\cC$ such that for any morphism $f: W\tens X\rightarrow V$ in $\cC$, there is a unique morphism $\varphi: X\rightarrow W'$ such that the diagram
\begin{equation*}
\xymatrixcolsep{4pc}
\xymatrix{
X\tens W \ar[d]_{\varphi\tens\Id_W} \ar[rd]^{f} & \\
W'\tens W \ar[r]_(.55){e_W} & V \\
}
\end{equation*}
commutes. 

Now suppose that $W$ is a rigid object of $\cC$, so that there is a dual object $W^*$ together with an even evaluation $e_W: W^*\tens W\rightarrow V$ and an even coevaluation $i_W: V\rightarrow W\tens W^*$ such that both compositions
\begin{equation*}
W \xrightarrow{\cong} V\tens W\xrightarrow{i_W\tens\Id_W} (W\tens W^*)\tens W\xrightarrow{\cong} W\tens (W^*\tens W)\xrightarrow{\Id_W\tens e_W} W\tens V\xrightarrow{\cong} W
\end{equation*}
and
\begin{equation*}
W^* \xrightarrow{\cong} W^*\tens V\xrightarrow{\Id_{W^*}\tens i_W} W^*\tens (W\tens W^*)\xrightarrow{\cong} (W^*\tens W)\tens W^*\xrightarrow{e_W\tens\Id_{W^*}} V\tens W^*\xrightarrow{\cong} W^*
\end{equation*}
are identities. Then it is straightforward to show that $(W^*, e_W)$ also satisfies the universal property of contragredient modules in $\cC$, and thus $W^*\cong W'$. In other words, if $V$ is self-contragredient and if $\cC$ is a rigid tensor supercategory of strongly $\CC$-graded $V$-modules that is closed under contragredients, then duals in $\cC$ are given by contragredients.

\subsection{The Heisenberg and singlet vertex operator algebras}

Let $\cH$ be the rank-one Heisenberg vertex operator algebra associated to the  abelian Lie algebra $\CC h$ with symmetric bilinear form such that $\langle h,h\rangle =1$. Writing $Y(h,x)=\sum_{n\in\ZZ} h(n)\,x^{-n-1}$, the standard conformal vector of $\cH$ is $\omega_{Std}=\frac{1}{2} h(-1)^2\vac$, which gives $\cH$ central charge $1$. The irreducible $\cH$-modules are the Heisenberg Fock modules $\cF_\lambda$ for $\lambda\in\CC$; the one-dimensional lowest conformal weight space $\CC v_\lambda$ of $\cF_\lambda$ satisfies
\begin{equation*}
h(n)v_\lambda =\delta_{n,0}\lambda v_\lambda.
\end{equation*}
for $n\geq 0$. Moreover, the lowest conformal weight of $\cF_\lambda$ with respect to the standard conformal vector is $\frac{1}{2}\lambda^2$, and the $\cH$-module contragredient of $\cF_\lambda$ is $\cF_{-\lambda}$.

For any $p\in\ZZ_{\geq 2}$, we can give $\cH$ a new conformal vector
\begin{equation*}
\omega =\frac{1}{2}h(-1)^2\vac+\frac{p-1}{\sqrt{2p}}h(-2)\vac =\omega_{Std}+\frac{\alpha_0}{2}h(-2)\vac,
\end{equation*}
where
\begin{equation*}
\alpha_0=\alpha_++\alpha_-,\qquad\alpha_+=\sqrt{2p},\qquad\alpha_-=-\sqrt{2/p}.
\end{equation*}
With respect to the new conformal vector $\omega$, $\cH$ has central charge $c_{p,1}=13-6p-6p^{-1}$, the minimal conformal weight of $\cF_\lambda$ is 
\begin{equation*}
h_\lambda =\frac{1}{2}\lambda(\lambda-\alpha_0),
\end{equation*}
and the new contragredient of $\cF_\lambda$ is $\cF_{\alpha_0-\lambda}$. The Heisenberg algebra $\cH$ is not semisimple as a module for the Virasoro algebra of central charge $c_{p,1}$: it is shown in \cite{Ad} that the Virasoro socle of $\cH$ is a simple vertex operator subalgebra (containing $\omega$) called the singlet algebra $\cM(p)$.

Most Heisenberg Fock modules remain simple as $\cM(p)$-modules. To list those which do not, we introduce the rank-one lattice $L=\mathbb{Z}\alpha_+$ and its dual $L^\circ=\ZZ\frac{\alpha_-}{2}$. Any $\lambda\in L^\circ$ equals
\begin{equation*}
\alpha_{r,s}:=\frac{1-r}{2}\alpha_++\frac{1-s}{2}\alpha_-
\end{equation*}
for a unique $r\in\ZZ$ and $s\in\lbrace 1,2,\ldots, p\rbrace$. For every $r\in\ZZ$ and $1\leq s\leq p$, set $\cM_{r,s}$ to be the $\cM(p)$-socle of $\cF_{\alpha_{r,s}}$. Then we get the following classification of irreducible $\cM(p)$-modules from \cite{Ad}, using the parameterization of \cite[Section 2]{CRW}:
\begin{itemize}
\item Every irreducible $\cM(p)$-module is isomorphic to either $\cF_\lambda$ for some $\lambda\in\CC\setminus L^\circ$ or $\cM_{r,s}$ for some $r\in\ZZ$ and $1\leq s\leq p$.

\item For $r\in\ZZ$ and $1\leq s\leq p$, $\cM_{r,s}=\cF_{\alpha_{r,s}}$ if and only if $s=p$.
\end{itemize}
Moreover, for $r\in\ZZ$ and $1\leq s\leq p$, the lowest conformal weight of $\cM_{r,s}$ is
\begin{equation*}
h_{r,s}:=\frac{r^2-1}{4}p-\frac{rs-1}{2}+\frac{s^2-1}{4p}
\end{equation*}
if $r\geq 1$, and $h_{2-r,s}$ if $r\leq 1$. For $\lambda\in\CC$, the $\cM(p)$-module contragredient of $\cF_\lambda$ is still $\cF_{\alpha_0-\lambda}$, and for $r\in\ZZ$ and $1\leq s\leq p$, $\cM_{r,s}'\cong\cM_{2-r,s}$.

Let $\cO_{\cM(p)}$ be the category of finite-length grading-restricted generalized $\cM(p)$-modules. We showed in \cite{CMY-singlet-typical} that $\cO_{\cM(p)}$ equals the category of $C_1$-cofinite grading-restricted generalized $V$-modules, and that $\cO_{\cM(p)}$ admits the vertex algebraic braided tensor category structure of \cite{HLZ1}-\cite{HLZ8}. Moreover, $\cO_{\cM(p)}$ is rigid and ribbon with duals given by contragredient modules and natural ribbon twist $e^{2\pi i L(0)}$. Note that $\cO_{\cM(p)}$ is much smaller than the category of all grading-restricted generalized (or ``ordinary'') $\cM(p)$-modules, since infinite direct sums of simple ordinary $\cM(p)$-modules can be grading restricted (take for example the triplet vertex algebra $\cW(p)\cong\bigoplus_{n\in\ZZ} \cM_{2n+1,1}$ considered as an $\cM(p)$-module). However, one corollary of our results in \cite{CMY-singlet-typical} is that all finitely-generated grading-restricted generalized $\cM(p)$-modules have finite length and are thus objects of $\cO_{\cM(p)}$.

As an abelian category, $\cO_{\cM(p)}$ does not have non-zero projective objects. To fix this problem, in \cite{CMY-singlet-typical} we defined $\cO_{\cM(p)}^T$ to be the full subcategory of $\cO_{\cM(p)}$ consisting of objects $M$ such that the monodromy, or double braiding, isomorphism
\begin{equation*}
\cR^2_{\cM_{2,1},M}:= \cR_{M,\cM_{2,1}}\circ\cR_{\cM_{2,1},M}
\end{equation*}
is diagonalizable. 
The $T$ in the notation $\cO_{\cM(p)}^T$ is the algebraic torus $\CC/2L^\circ$, which acts by automorphisms on the doublet abelian intertwining algebra $\cA(p)=\bigoplus_{r\in\ZZ}\cM_{r,1}$ of \cite{AM-doub}. We expect that there is a notion of twisted module for abelian intertwining algebras and an induction functor from $\cO_{\cM(p)}$ to a suitable category of some kind of representations for $\cA(p)$ such that  $\cO_{\cM(p)}^T$ consists of all $\cM(p)$-modules in $\cO_{\cM(p)}$ which induce to (direct sums of) twisted $\cA(p)$-modules associated to automorphisms in $T$. However, since the theory of twisted modules for abelian intertwining algebras is not well developed, it is more straightforward to define $\cO_{\cM(p)}^T$ using the monodromy isomorphisms in $\cO_{\cM(p)}$. In any case, we showed in \cite{CMY-singlet, CMY-singlet-typical} that $\cO_{\cM(p)}^T$ has enough projectives:
\begin{thm}
The following are the indecomposable projective objects in $\cO_{\cM(p)}^T$:
\begin{enumerate} 
\item  For $\lambda\in(\CC\setminus L^\circ)\cup\lbrace\alpha_{r,p}\,\vert\,r\in\ZZ\rbrace$, the simple $\cM(p)$-module $\cF_\lambda$ is its own projective cover in $\cO_{\cM(p)}^T$. 

\item For $r\in\ZZ$ and $1\leq s\leq p-1$, the simple $\cM(p)$-module $\cM_{r,s}$ has a length-$4$ projective cover $\cP_{r,s}$ in $\cO_{\cM(p)}^T$ with Loewy diagram
 \begin{equation}\label{eqn:Prs_Loewy_diag}
 \begin{matrix}
  \begin{tikzpicture}[->,>=latex,scale=1.5]
\node (b1) at (1,0) {$\cM_{r, s}$};
\node (c1) at (-1, 1){$\cP_{r, s}$:};
   \node (a1) at (0,1) {$\cM_{r-1, p-s}$};
   \node (b2) at (2,1) {$\cM_{r+1, p-s}$};
    \node (a2) at (1,2) {$\cM_{r,s}$};
\draw[] (b1) -- node[left] {} (a1);
   \draw[] (b1) -- node[left] {} (b2);
    \draw[] (a1) -- node[left] {} (a2);
    \draw[] (b2) -- node[left] {} (a2);
\end{tikzpicture}
\end{matrix} .
 \end{equation}
 \end{enumerate}
\end{thm}

In \cite{CMY-singlet, CMY-singlet-typical}, we showed that $\cO_{\cM(p)}^T$ is a ribbon tensor subcategory of $\cO_{\cM(p)}$; in particular, $\cO_{\cM(p)}^T$ is closed under tensor products and contragredients. We also determined tensor products of simple modules in $\cO_{\cM(p)}$, equivalently in $\cO_{\cM(p)}^T$:

\begin{thm}\label{thm:singlet_fus_rules}
 The following tensor product formulas hold in $\cO_{\cM(p)}$:
 \begin{enumerate}

\item For $r,r'\in\ZZ$ and $1\leq s, s'\leq p$,
  \begin{equation*}
   \cM_{r,s}\boxtimes \cM_{r',s'}   \cong\bigg(\bigoplus_{\substack{\ell = |s-s'|+1 \\ \ell+s+s' \equiv 1\; (\mathrm{mod}\; 2)}}^{{\rm min}\{s+s'-1, 2p-1-s-s'\}}\cM_{r+r'-1, \ell}\bigg) \oplus \bigg(\bigoplus_{\substack{\ell = 2p+1-s-s'\\ \ell+s+s' \equiv 1\; (\mathrm{mod}\; 2)}}^{p}\cP_{r+r'-1, \ell}\bigg)
  \end{equation*}
with sums taken to be empty if the lower bound exceeds the upper bound. 
 
  \item 
  For $r\in\ZZ$, $1\leq s\leq p$, and $\lambda\in\CC\setminus L^\circ$, 
 \begin{equation*}
  \cM_{r,s}\tens\cF_\lambda\cong\bigoplus_{\ell=0}^{s-1} \cF_{\lambda+\alpha_{r,s}+\ell\alpha_-}.
 \end{equation*}

  \item
  For $\lambda,\mu\in\CC\setminus L^\circ$ such that $\lambda+\mu =\alpha_++\alpha_-+\alpha_{r,s}\in L^\circ$ for some $r\in\ZZ$, $1\leq s\leq p$,
  \begin{equation*}
  \cF_\lambda\tens\cF_{\mu}\cong\bigoplus_{\substack{s'= s\\ s'\equiv s\,\,(\mathrm{mod}\,2)\\}}^p \cP_{r,s'}\oplus\bigoplus_{\substack{s'=p+2-s\\s'\equiv p-s\,\,(\mathrm{mod}\,2)\\}}^p \cP_{r-1,s'}.
 \end{equation*}

  \item
  For $\lambda,\mu\in\CC\setminus L^\circ$ such that $\lambda+\mu\notin L^\circ$,
\begin{equation*}
\cF_{\lambda}\tens\cF_\mu \cong \bigoplus_{\ell=0}^{p-1} \cF_{\lambda + \mu + \ell \alpha_-}.
\end{equation*}
 \end{enumerate}
 \end{thm}

\section{Results for general simple current extensions}\label{sec:main_results}

In this section, we derive tensor-categorical results for general simple current extensions of $\cH\otimes\cM(p)$.

\subsection{Simple current extensions in direct limit completions}

Let $V$ be a simple vertex operator algebra and $\cC$ a category of grading-restricted generalized $V$-modules that is closed under submodules and quotients and such that every module in $\cC$ has finite length. We assume that $V$ itself is an object of $\cC$ and that $\cC$ admits the vertex algebraic braided tensor category structure of \cite{HLZ1}-\cite{HLZ8}. Suppose $J$ is a simple current $V$-module in $\cC$, that is, there is a simple $V$-module $J^{-1}$ in $\cC$ such that $J\tens J^{-1}\cong V$. For any $n\in\ZZ$, we use $J^n$ to denote the $n$th tensor power of $J$ in $\cC$ (if $n>0$), the $\vert n\vert$th tensor power of $J^{-1}$ (if $n<0$), or $V$ (if $n=0$). In our examples, we will have $J^m\cong J^n$ if and only if $m,n\in\ZZ$ are equal.

By \cite[Theorem 3.12]{CKL}, under suitable assumptions on the conformal weight gradings of the simple currents $J^n$, the $V$-module
\begin{equation*}
A=\bigoplus_{n\in\ZZ} J^n
\end{equation*}
has a natural (in general $\frac{1}{2}\ZZ$-graded) simple vertex operator (super)algebra structure. If $A$ is a superalgebra, with $A^{\odd}\neq 0$, then
\begin{equation}\label{eqn:A_parity}
A^{\overline{i}} =\bigoplus_{n\in i+2\ZZ} J^n
\end{equation}
for $i =0,1$. This is because $A^{\even}$ contains $V$ (since $V$ is a vertex operator algebra), and therefore $A^{\even}$ and $A^{\odd}$ are semisimple $V$-modules. If $J\subseteq A^{\even}$, then $A=A^{\even}$ since $J$ generates $A$ as a vertex operator (super)algebra. Thus $J\subseteq A^{\odd}$ if $A^{\odd}\neq 0$, and it then follows from simple current fusion rules that $J^n\subseteq A^{\overline{i}}$ when $n-i\in2\ZZ$. Similarly, if $A$ has strict half-integral conformal weights, then
\begin{equation}\label{eqn:A_conf_wts}
\bigoplus_{n\in\frac{i}{2}+\ZZ} A_{(n)} = \bigoplus_{n\in i+2\ZZ} J^n
\end{equation}
for $i=0,1$. We would like to use the tensor category $\cC$ to understand tensor structure on categories of $A$-modules, but note that $A$ is not an object of the locally-finite category $\cC$. Thus we need the direct limit completion of $\cC$ as studied in \cite{CMY-completions}.

 The direct limit completion, or $\mathrm{Ind}$-category, of $\cC$ is the category $\ind(\cC)$ of generalized $V$-modules (typically with infinite-dimensional conformal weight spaces) whose objects are the unions of their $\cC$-submodules. Equivalently, a generalized $V$-module is an object of $\ind(\cC)$ if and only if all of its singly-generated $V$-submodules are objects of $\cC$. The main theorem of \cite{CMY-completions} states that $\ind(\cC)$ admits the vertex algebraic braided tensor category structure of \cite{HLZ1}-\cite{HLZ8} under the following condition: For any intertwining operator $\cY$ of type $\binom{X}{W_1\,W_2}$ where $W_1$, $W_2$ are objects of $\cC$ and $X$ is an object of $\ind(\cC)$, the image $\im\cY\subseteq X$ is an object of $\cC$. (Recall that the image of an intertwining operator $\cY$ of type $\binom{X}{W_1\,W_2}$ is the submodule of $X$ spanned by coefficients of powers of $x$ and $\log x$ in $\cY(w_1,x)w_2$ for $w_1\in W_1$, $w_2\in W_2$.) This condition is satisfied if $\cC$ is the category of $C_1$-cofinite grading-restricted generalized $V$-modules (see \cite[Theorem 7.1]{CMY-completions}).
 
Now, the simple current extension $A=\bigoplus_{n\in\ZZ} J^n$ is an object of $\ind(\cC)$. Thus when $A$ is a vertex operator algebra, it is a commutative algebra in the braided tensor category $\ind(\cC)$, with multiplication morphism $\mu_A: A\tens A\rightarrow A$ induced by the vertex operator $Y_A$ \cite{HKL, CMY-completions}. When $A$ is a superalgebra, it is a commutative algebra with even multiplication in an auxiliary tensor supercategory $\cS\ind(\cC)$ \cite{CKL, CKM-exts}, which is just the Deligne product of $\ind(\cC)$ with the supercategory of finite-dimensional vector superspaces. Equivalently, $\cS\ind(\cC)$ is the tensor supercategory of $V$-modules in $\ind(\cC)$ where $V$ is treated as a vertex operator superalgebra with trivial odd component.

We then have the tensor (super)category $\repA$ of (possibly non-local) $A$-modules in $(\cS)\ind(\cC)$ \cite{KO, CKM-exts}: Objects of $\repA$ are pairs $(X,\mu_X)$ where $X$ is an object of $(\cS)\ind(\cC)$ and $\mu_X: A\tens X\rightarrow X$ is an (even) unital associative action of the algebra $A$ on $X$. Equivalently by \cite[Proposition 3.46]{CKM-exts}, an object of $\repA$ is a generalized $V$-module $X$ in $\ind(\cC)$ (with a $\ZZ/2\ZZ$-grading if $A$ is superalgebra) equipped with an (even) intertwining operator
\begin{align*}
Y_X: A\otimes X & \rightarrow X[\log x]\lbrace x\rbrace\nonumber\\
a\otimes b & \mapsto Y_X(a,x)b=\sum_{h\in\CC}\sum_{k\in\NN} a_{h;k}b\,x^{-h-1}(\log x)^k
\end{align*}
satisfying the vacuum property $Y_X(\vac,x)=\Id_X$ and the convergence and associativity relation
\begin{equation}\label{eqn:RepA_assoc}
\langle b', Y_X(a_1,e^{\ln r_1})Y_X(a_2,e^{\ln r_2})b\rangle =\langle b',Y_X(Y(a_1,e^{\ln(r_1-r_2)})a_2,e^{\ln r_2})b\rangle
\end{equation}
for all real numbers $r_1>r_2>r_1-r_2$, all $a_1,a_2\in A$, all $b\in X$, and all $b'\in X'=\bigoplus_{h\in\CC} X_{[h]}^*$. Since the grading restriction conditions need not hold for generalized modules $X$ in $\ind(\cC)$, the graded dual vector space $X'$ may not be a $V$-module in $\ind(\cC)$. Nevertheless, products and iterates of intertwining operators as in \eqref{eqn:RepA_assoc} still converge absolutely (see the proof of \cite[Theorem 6.3(3)]{CMY-completions}).

By \cite[Theorem~3.65]{CKM-exts} and \cite[Theorem~7.7]{CMY-completions}, we also have the vertex algebraic braided tensor (super)category $\rep^0A$ of local $A$-modules in $(\cS)\ind(\cC)$. Its objects are pairs $(X,\mu_X)$ of $\repA$ such that
\begin{equation*}
\mu_X\circ\cR_{A,X}^2=\mu_X.
\end{equation*}
Here $\cR_{A,X}^2:=\cR_{X,A}\circ\cR_{A,X}$ is the monodromy (or double braiding) isomorphism in the braided tensor (super)category $(\cS)\ind(\cC)$. Equivalently, an object $X$ in $\repA$ is an object of $\rep^0A$ if and only if
\begin{equation*}
Y_X(a,x)b\in X((x))
\end{equation*}
for all $a\in A$ and $b\in X$. By \cite[Theorem 3.4]{HKL}, $\rep^0A$ consists precisely of the generalized $A$-modules (with $A$ considered as a vertex operator (super)algebra) that restrict to $V$-modules in $\ind(\cC)$.

The vertex operator (super)algebra $A$ has an involution $\theta\in\mathrm{Aut}(A)$ defined by 
\begin{equation*}
\theta\vert_{J^n} = (-1)^n\Id_{J^n}
\end{equation*}
for $n\in\ZZ$; by \eqref{eqn:A_parity} and \eqref{eqn:A_conf_wts}, $\theta$ is the parity automorphism of $A$ when $A^{\odd}\neq 0$, and $\theta$ is the ribbon twist $e^{2\pi i L(0)}$ when $\bigoplus_{n\in\frac{1}{2}+\ZZ} A_{(n)}\neq 0$. Then we can define the full subcategory $\rep^1A$ of $\repA$ to consist of those objects $(X,\mu_X)$ such that
\begin{equation*}
\mu_X\circ\cR_{A,X}^2\circ(\theta\tens\Id_X)=\mu_X,
\end{equation*}
or equivalently, those objects $(X,Y_X)$ such that
\begin{equation*}
Y_X(\theta(a),e^{2\pi i} x)=Y_X(a,x)
\end{equation*}
for all $a\in A$. By \cite[Theorem 4.14]{McR-orb2}, $\rep^1A$ consists precisely of the generalized $\theta$-twisted $A$-modules that restrict to $V$-modules in $\ind(\cC)$. Moreover, by \cite[Theorem 4.15(2)]{McR-orb2}, the direct sum category $\rep^0A\oplus\rep^1A$ (whose objects are direct sums of local and $\theta$-twisted $A$-modules in $\repA$) is a braided $\ZZ/2\ZZ$-crossed tensor (super)category.

We will need the following monodromy lemma (which is a minor generalization of \cite[Proposition 6.4]{CMY-singlet-typical}) in the next subsection:
\begin{lem}\label{lem:monodromy}
 If $X$ is a generalized $A$-module in $\rep^kA$ for $k=0$ or $k=1$ and $W\subseteq X$ is a $V$-submodule, then monodromies satisfy $\cR^2_{J,W}=(-1)^k\Id_{J\tens W}$.
\end{lem}
\begin{proof}
 Let $i: W\hookrightarrow X$ and $j: J\hookrightarrow A$ denote the respective inclusions. By naturality of the monodromy isomorphisms in $(\cS)\ind(\cC)$, the diagram
 \begin{equation*}
 \xymatrixcolsep{4pc}
  \xymatrix{
  J\tens W \ar[d]^{\cR_{J,W}^2} \ar[r]^{\Id_J\tens i} & J\tens X \ar[d]^{\cR_{J,X}^2} \ar[r]^{j\tens\Id_X} & A\tens X \ar[d]^{\cR^2_{A,X}}\\
  J\tens W \ar[r]^{\Id_J\tens i} \ar[rd]_{\mu_X\vert_{J\tens W}} & J\tens X \ar[d]^(.4){\mu_X\vert_{J\tens X}} \ar[r]^{j\tens\Id_X} & A\tens X \ar[ld]^{\mu_X}\\
  & X & \\
  }
 \end{equation*}
commutes. Here $\Id_J\tens i$ is injective because $J\tens\bullet$ is exact (since the simple current $J$ is rigid with dual $J^{-1}$), and $j\tens\Id_X$ injective because $J$ is a direct summand of $A$. Now since $\mu_X\circ\cR_{A,X}^2=\mu_X\circ(\theta^{-k}\tens\Id_X)$ by definition of $\rep^kA$ and since $\theta^{-k}\circ j=(-1)^k j$,
\begin{equation*}
 \mu_X\vert_{J\tens X}\circ(\Id_J\tens i)\circ\cR_{J,W}^2=(-1)^k\mu_X\circ(j\tens i)=(-1)^k\mu_X\vert_{J\tens X}\circ(\Id_J\tens i).
\end{equation*}
Because $\Id_J\tens i$ is injective, it is enough to show $\mu_X\vert_{J\tens X}$ is injective as well. In fact, one can check that $\mu_X\vert_{J\tens X}$ is an isomorphism with inverse
\begin{align*}
 X\xrightarrow{l_X^{-1}} & V\tens X\xrightarrow{(\mu_A\vert_{J\tens J^{-1}})^{-1}\tens\Id_X} (J\tens J^{-1})\tens X\nonumber\\
 &\xrightarrow{\cA_{J,J^{-1},X}^{-1}} J\tens(J^{-1}\tens X)\xrightarrow{\Id_J\tens\mu_X\vert_{J^{-1}\tens X}} J\tens X,
\end{align*}
using associativity of the multiplication $\mu_X$ and the fact that $J$ is a simple current.
\end{proof} 

There is a tensor functor of induction
\begin{equation*}
 \cF: \cC\longrightarrow\repA
\end{equation*}
defined by $\cF(W)=A\tens W$ on objects and $\cF(f)=\Id_A\tens f$ on morphisms (in the case that $A$ is a superalgebra, we set $\cF(W)^{\overline{i}}=A^{\overline{i}}\tens W$ for $i=0,1$). Since the simple current $J$ is a rigid object of $\cC$ with dual $J^{-1}$, the induction functor $\cF$ is exact, as in the proof of \cite[Proposition 3.2.4]{CMY-singlet} (see also \cite[Remark 3.21]{CKL}). Induction also satisfies Frobenius reciprocity in the sense that there is a natural isomorphism
\begin{equation*}
\mathrm{Hom}_{\repA}(\cF(W),X)\cong\mathrm{Hom}_{\ind(\cC)}(W,X)
\end{equation*}
for any $V$-module $W$ in $\cC$ and any object $X$ in $\repA$. Moreover, if $W$ is a simple $V$-module in $\cC$ such that $J^n\tens W\cong W$ only for $n=0$, then $\cF(W)$ is a simple object of $\repA$ by \cite[Proposition 4.4]{CKM-exts}. Similar to \cite[Proposition 5.0.1]{CMY3}, we can use induction to classify simple modules in $\rep^0A$ and $\rep^1A$. For a simple $V$-module $W$ in $\cC$, let $h_W\in\CC$ be its minimal conformal weight (or \textit{conformal dimension}).
\begin{prop}\label{prop:general_simple_module_classification}
Assume that for all simple $V$-modules $W$ in $\cC$, $J^n\tens W\cong W$ only if $n=0$. Then an object of $\rep^kA$ for $k=0$ or $k=1$ is simple if and only if it is isomorphic to $\cF(W)$ for some simple $V$-module $W$ in $\cC$ such that
\begin{equation*}
h_{J\tens W}-h_J-h_W\in\frac{k}{2}+\ZZ.
\end{equation*}
Moreover, if $W$ and $\til{W}$ are simple $V$-modules in $\cC$, then $\cF(W)\cong\cF(\til{W})$ if and only if $\til{W}\cong J^n\tens W$ for some $n\in\ZZ$.
\end{prop}
\begin{proof}
If $W$ is a simple $V$-module in $\cC$, then $\cF(W)$ is a simple object of $\repA$ by \cite[Proposition 4.4]{CKM-exts}. Moreover, \cite[Proposition 2.65]{CKM-exts} (and its easy generalization to $\theta$-twisted modules) shows that $\cF(W)$ is an object of $\rep^kA$ if and only if $\cR_{A,W}^2=\theta^{-k}\tens\Id_W$, equivalently $\cR_{J^n,W}^2=(-1)^{kn}\Id_{J^n\tens W}$ for all $n\in\ZZ$. Since $J$ generates the group of simple currents $\lbrace J^n\rbrace_{n\in\ZZ}$, it then follows from \cite[Theorem 2.11(2)]{CKL} that $\cF(W)$ is an object of $\rep^kA$ if and only if $\cR^2_{J,W}=(-1)^k\Id_{J\tens W}$. Then since
\begin{align}\label{eqn:lifting_criterion}
\cR_{J,W}^2=e^{2\pi i L_{J\tens W}(0)}\circ(e^{-2\pi i L_J(0)}\tens e^{-2\pi i L_W(0)}) =e^{2\pi i(h_{J\tens W}-h_J-h_W)}\Id_{J\tens W}.
\end{align}
by the balancing equation for monodromy, $\cF(W)$ is a simple object of $\rep^kA$ if $W$ is a simple $V$-module in $\cC$ such that $h_{J\tens W}-h_J-h_W\in\frac{k}{2}+\ZZ$.

Conversely, suppose $X$ is a simple object of $\rep^kA$. Because the parity-homogeneous parts of $X$ are objects of $\ind(\cC)$, they are the unions of their $\cC$-submodules. Then because every module in $\cC$ has finite length, one of the parity-homogeneous parts of $X$ contains a simple $V$-submodule $W$ which is an object of $\cC$. The $\ind(\cC)$-inclusion $W\hookrightarrow X$ induces by Frobenius reciprocity a non-zero parity-homogeneous $\repA$-morphism $\cF(W)\rightarrow X$. Since $\cF(W)$ and $X$ are both simple, it follows that $\cF(W)\cong X$ (possibly by an odd isomorphism), and then since $X$ is an object of $\rep^kA$, \eqref{eqn:lifting_criterion} implies that $h_{J\tens W}-h_J-h_W\in\frac{k}{2}+\ZZ$.

Now suppose $W$ and $\til{W}$ are simple $V$-modules in $\cC$. Since $\cF(W)$ and $\cF(\til{W})$ are both simple objects of $\repA$, $\cF(W)\cong\cF(\til{W})$ if and only if $\mathrm{Hom}_{\repA}(\cF(\til{W}),\cF(W))\neq 0$. By Frobenius reciprocity, this holds if and only if $\mathrm{Hom}_{\ind(\cC)}(\til{W},\cF(W))\neq 0$. As a $V$-module, $\cF(W)\cong\bigoplus_{n\in\ZZ} J^n\tens W$ is a direct sum of simple $V$-modules, so there is a non-zero $V$-module homomorphism $\til{W}\rightarrow X$ if and only if $\til{W}\cong J^n\tens W$ for some $n\in\ZZ$.
\end{proof}

\subsection{General simple current extensions of Heisenberg times singlet}

We now apply the results and discussion of the previous subsection to the case $V=\cH\otimes\cM(p)$ for $p\in\ZZ_{\geq 2}$.  Since we can consider Heisenberg Fock modules as either $\cH$-modules or $\cM(p)$-modules, we will from now on use $\cF^\cH_\lambda$ for $\lambda\in\CC$ to denote a Fock module considered as an $\cH$-module, while $\cF_\lambda$ will denote the same Fock module considered as an $\cM(p)$-module.

 By \cite[Proposition 6.1]{CMY-singlet-typical}, both direct limit completions $\ind(\cO_{\cM(p)})$ and $\ind(\cO_{\cM(p)}^T)$ satisfy the assumptions of \cite[Theorem 1.1]{CMY-completions}, and thus both admit the vertex algebraic braided tensor category structure of \cite{HLZ8}. We also have the following result, which is \cite[Proposition 6.2]{CMY-singlet-typical}:
\begin{prop}\label{prop:gen_Mp_mod_in_DLC}
 Any grading-restricted generalized $\cM(p)$-module is an object of the direct limit completion $\ind(\cO_{\cM(p)})$.
\end{prop}

\begin{rem}
Although not every grading-restricted generalized (or ``ordinary'') $\cM(p)$-module has finite length (recall the triplet vertex algebra $\cW(p)\cong\bigoplus_{n\in\ZZ} \cM_{2n+1,1}$ considered as an $\cM(p)$-module), the above proposition shows that every grading-restricted generalized $\cM(p)$-module is the union of its finite-length submodules.
\end{rem}

 Let $\mathcal{O}_{\mathcal{H}}$ be the category of finite-length grading-restricted generalized $\cH$-modules on which $h(0)$ acts semisimply. From \cite[Theorem 1.7.3]{FLM} (see also \cite[Lemma 2.1]{M}), $\cO_{\cH}$ is semisimple with simple objects the Fock modules $\cF_{\lambda}^{\cH}$ for $\lambda \in \mathbb{C}$. Now let $\cC$ be the category of $\cH\otimes\cM(p)$-modules whose objects are (isomorphic to) finite direct sums of modules $\cF^\cH_\lambda\otimes M$ with $M$ an object of $\cO^T_{\cM(p)}$. By \cite[Theorems 5.2 and 5.5]{CKM2}, $\cC$ admits the vertex algebraic braided tensor category structure of \cite{HLZ8}, and moreover $\cC$ is braided tensor equivalent to the Deligne product $\cO_{\cH} \boxtimes \cO_{\cM(p)}^T$. So in particular $\cC$ is rigid, since both $\cO_\cH$ and $\cO_{\cM(p)}^T$ are rigid. Note that $\cC$ is a relatively small subcategory of the category of all grading-restricted generalized $\cH\otimes\cM(p)$-modules, since for example it does not contain any module which has a non-split self-extension of a Heisenberg Fock module (on which $h(0)$ acts non-semisimply) as an $\cH$-submodule, nor does it contain any grading-restricted infinite direct sums of simple $\cH\otimes\cM(p)$-modules.

The direct limit completion $\ind(\mathcal{C})$ consists of all generalized $\cH \otimes \cM(p)$-modules which are isomorphic to arbitrary direct sums
\begin{equation*}
 X=\bigoplus_{\lambda\in\CC} \cF_\lambda^\cH\otimes M_\lambda
\end{equation*}
such that $M_\lambda$ is a generalized $\cM(p)$-module in $\ind(\cO_{\cM(p)}^T)$ (see the proof of \cite[Proposition 3.6]{McR-cosets}). Moreover, since $\cO_{\cM(p)}^T$ satisfies the assumptions in \cite[Theorem 1.1]{CMY-completions}, so does $\cC$ by \cite[Proposition 3.6]{McR-cosets}. Thus $\ind(\cC)$ admits the vertex algebraic braided tensor category structure of \cite{HLZ8}. We also have:
\begin{lem}
For $\lambda\in\CC$ and $P$ a projective object of $\cO_{\cM(p)}^T$, $\cF^\cH_\lambda\otimes P$ is projective in $\ind(\cC)$.
\end{lem}
\begin{proof}
Consider the following diagram in $\ind(\cC)$ where $p$ is surjective:
\begin{equation*}
\xymatrixcolsep{4pc}
\xymatrix{
& \cF^\cH_\lambda\otimes P \ar[d]^{q} \\
\bigoplus_{\mu\in\CC} \cF^\cH_\mu\otimes M_\mu \ar[r]^{p} & \bigoplus_{\mu\in\CC} \cF^\cH_\mu\otimes N_\mu \\
}
\end{equation*}
Since $p$ and $q$ preserve $h(0)$-eigenspaces, $p=\bigoplus_{\mu\in\CC} p_\mu$ where each $p_\mu: \cF^\cH_\mu\otimes M_\mu\rightarrow \cF^\cH_\mu\otimes N_\mu$ is surjective, and $\im q\subseteq\cF^\cH_\lambda\otimes N_\lambda$. Moreover, $q$ and each $p_\mu$ preserve conformal weight spaces for the Heisenberg conformal vector $\omega_{Std}\otimes\vac_{\cM(p)}$, so the above diagram restricts to
\begin{equation*}
\xymatrixcolsep{4pc}
\xymatrix{
&v_\lambda\otimes P \ar[d]^{q\vert_{v_\lambda\otimes P}}\\
v_\lambda\otimes M_\lambda \ar[r]^{p_\lambda\vert_{v_\lambda\otimes M_\lambda}} & v_\lambda\otimes N_\lambda\\
}
\end{equation*}
where $v_\lambda\in\cF_\lambda^\cH$ is a non-zero vector of minimal conformal weight and $p_\lambda\vert_{v_\lambda\otimes M_\lambda}$ is surjective. This is now a diagram in $\ind(\cO_{\cM(p)}^T)$. Thus because $P$ is a (finitely-generated) projective object of $\cO_{\cM(p)}^T$, it follows as in the proof of \cite[Lemma 5.0.3]{CMY3} that there is an $\cM(p)$-module homomorphism $f: P\rightarrow M_\lambda$ such that 
\[
p_\lambda\vert_{v_\lambda\otimes M_\lambda}\circ f= q\vert_{v_\lambda\otimes P}.
\]
We then get the $\ind(\cC)$-morphism 
\begin{equation*}
F: \cF^\cH_\lambda\otimes P \xrightarrow{\Id_{\cF^\cH_\lambda}\otimes f} \cF^\cH_\lambda\otimes M_\lambda\hookrightarrow\bigoplus_{\mu\in\CC} \cF^\cH_\mu\otimes M_\mu,
\end{equation*}
which satisfies $p\circ F=q$ because $v_\lambda\otimes P$ generates $\cF^\cH_\lambda\otimes P$ as an $\cH$-module. Thus $\cF^\cH_\lambda\otimes P$ is projective in $\ind(\cC)$.
\end{proof}

\begin{cor}\label{cor:proj_cover_in_C}
For $\lambda\in\CC$ and $\mu\in\CC\setminus L^\circ\cup\lbrace\alpha_{r,p}\,\vert\,r,\in\ZZ\rbrace$, the simple $\cH\otimes\cM(p)$-module $\cF^\cH_\lambda\otimes\cF_\mu$ is its own projective cover in $\cC$, while for $\lambda\in\CC$, $r\in\ZZ$, and $1\leq s\leq p-1$, $\cF^\cH_\lambda\otimes\cP_{r,s}$ is a projective cover of $\cF^\cH_\lambda\otimes\cM_{r,s}$ in $\cC$.
\end{cor}

We now consider a simple current extension $A=\bigoplus_{n\in\ZZ} J^n$ of $\cH\otimes\cM(p)$ as in the previous subsection, where $J=\cF^\cH_{\lambda_J}\otimes\cM_{r_J+1,1}$ for some $\lambda_J\in\CC$ and $r_J\in\ZZ$. Replacing $J$ with $J^{-1}$ if necessary, we assume $r_J\geq 0$. We also assume $\lambda_J\neq 0$ to guarantee that for any simple $\cH\otimes\cM(p)$-module $W$ in $\cC$, $J^n\tens W\cong W$ only if $n=0$.  From fusion rules for $\cH$- and $\cM(p)$-modules, we have
\begin{equation*}
J^n\cong\cF^\cH_{n\lambda_J}\otimes\cM_{n r_J+1,1}
\end{equation*}
for $n\in\ZZ$, so $J^n$ is the $n\lambda_J$-eigenspace for the Heisenberg zero-mode $h(0)$ acting on $A$. Moreover, the lowest conformal weight of $J^n$ is
\begin{equation}\label{eqn:lowest_wt_of_Jn}
\frac{1}{2}(n\lambda_J)^2+h_{\vert n\vert r_J+1,1} =\frac{n^2}{2}\left(\lambda_J^2+r_J^2\frac{p}{2}\right) +\frac{\vert n\vert}{2}r_J(p-1).
\end{equation}
Thus to ensure that $J$ is $\ZZ$- or $(\frac{1}{2}+\ZZ)$-graded by conformal weights, and that $A$ is $\frac{1}{2}\ZZ_{\geq 0}$-graded with finite-dimensional weight spaces, we further assume that $\lambda_J^2+r_J^2\frac{p}{2}\in\ZZ_{\geq 0}$.

Considering the vertex operator (super)algebra $A$ as a commutative algebra in $(\cS)\ind(\cC)$, we have the tensor (super)category $\repA$ and its subcategories $\rep^0A$ and $\rep^1A$ of local and $\theta$-twisted generalized $A$-modules in $\ind(\cC)$. We also have the exact induction functor $\cF: \cC\rightarrow\repA$. We now generalize Proposition \ref{prop:general_simple_module_classification} beyond simple modules:
 \begin{prop}\label{prop:fin_gen_Rep_B_p}
The following three full subcategories of $\repA$ are equal:
\begin{enumerate}
\item The image of the induction functor $\cF:\cC\rightarrow\repA$.
\item The category of finite-length modules in $\repA$.
\item The category of finitely-generated modules in $\repA$.
\end{enumerate} 
\end{prop}
\begin{proof}
First suppose an object $X$ of $\repA$ is isomorphic to $\cF(W)$ for some $\cH\otimes\cM(p)$-module $W$ in $\cC$. Then because $W$ has finite length, and because induction is exact and sends simple $\cH\otimes\cM(p)$-modules to simple non-local $A$-modules by \cite[Proposition 4.4]{CKM-exts}, $X$ has finite length. Thus the image of $\cF$ is contained in the category of finite-length modules in $\repA$. It is clear that the category of finite-length modules in $\repA$ is contained in the category of finitely-generated modules. So it remains to show that any finitely-generated module in $\repA$ is isomorphic to the induction of a module in $\cC$.

Any object $X$ of $\repA$ has a decomposition $X=\bigoplus_{\lambda\in\CC} \cF^\cH_\lambda\otimes M_\lambda$ as an $\cH\otimes\cM(p)$-module, where each $M_\lambda$ is a generalized $\cM(p)$-module in $\ind(\cO^T_{\cM(p)})$. If $X$ is finitely generated, its finitely many generators are contained in the direct sum of finitely many $\cF^\cH_{\lambda_k}\otimes M_{\lambda_k}$ for $k=1,\ldots,K$. So because the vertex operator action of $A$ on $X$ changes $h(0)$-eigenvalues by elements of $\ZZ\lambda_J$, $X=\sum_{k=1}^K X_k$ where
\begin{equation*}
X_k=\bigoplus_{n\in\ZZ} \cF^\cH_{\lambda_k+n\lambda_J}\otimes M_{\lambda_k+n\lambda_J}.
\end{equation*}
Since induction is an additive functor, it is enough to show that each $X_k$ is isomorphic to an induced module. Moreover, each $X_k$ is generated by the projections of the generators of $X$ to $X_k$. That is, we are reduced to the case that $X=\bigoplus_{n\in\ZZ} \cF^\cH_{\lambda+n\lambda_J}\otimes M_{\lambda+n\lambda_J}$ for some fixed $\lambda\in\CC$.

We may assume that $\cF^\cH_\lambda\otimes M_\lambda\neq 0$. This subspace generates $X$ by the same argument as in the beginning of the proof of \cite[Proposition 3.7]{CKLR}  (except that the reference there to \cite[Corollary 4.5.15]{LL}, which applies only when $X$ is an object of $\rep^0A$, should be replaced by its non-local generalization \cite[Lemma 3.73]{CKM-exts}). We can now show that $\cF^\cH_\lambda\otimes M_\lambda$ is an object of $\cC$; it is enough to show that $\cF^\cH_\lambda\otimes M_\lambda$ is a finitely-generated $\cH\otimes\cM(p)$-module. Since $X$ is a finitely-generated non-local $A$-module which is also generated by $\cF^\cH_\lambda\otimes M_\lambda$ as an $A$-module, each of the finitely many generators for $X$ is contained in the $A$-submodule generated by finitely many vectors in $\cF^\cH_\lambda\otimes M_\lambda$. Thus as an $A$-module, $X$ has a finite generating set $\lbrace b_i\rbrace_{i=1}^I$ contained in $\cF^\cH_\lambda\otimes M_\lambda$, so that
\begin{equation*}
X =\mathrm{span}\lbrace a_{h;m} b_i\,\,\vert\,\,a\in A, h\in\CC, m\in\NN, 1\leq i\leq I\rbrace
\end{equation*}
by \cite[Lemma 3.74]{CKM-exts}. Since $a_{h;m} b_i\in\cF^\cH_{\lambda+n\lambda_J}\otimes M_{\lambda+n\lambda_J}$ for $a\in J^n\subseteq A$, it follows that $\lbrace b_i\rbrace_{i=1}^I$ is a finite generating set for $\cF^\cH_\lambda\otimes M_\lambda$ as an $\cH\otimes\cM(p)$-module, as desired.

We have now shown that $\cF^\cH_\lambda\otimes M_\lambda$ is a (grading-restricted generalized) $\cH\otimes\cM(p)$-module in $\cC$. By Frobenius reciprocity, the $\cH\otimes\cM(p)$-module inclusion $\cF^\cH_\lambda\otimes M_\lambda\hookrightarrow X$ induces a $\repA$-morphism
\begin{equation*}
F: \cF(\cF^\cH_\lambda\otimes M_\lambda)\longrightarrow X
\end{equation*}
which contains $\cF^\cH_\lambda\otimes M_\lambda$ in its image and thus is surjective since $\cF^\cH_\lambda\otimes M_\lambda$ generates $X$. (Note in the case that $A$ is a superalgebra that $F$ need not be parity-homogeneous.) We will show that $F$ is an isomorphism, so that $X$ is isomorphic to an induced object.

 As in the proof of \cite[Proposition 3.7]{CKLR} (with the reference to \cite[Proposition 4.5.6]{LL} replaced by \cite[Lemma 3.74]{CKM-exts}), the vertex operator $Y_X$ induces an $\cH\otimes\cM(p)$-module surjection
 \begin{equation}\label{eqn:like_simple_current}
 J^n\tens(\cF^\cH_\lambda\otimes M_\lambda)\longrightarrow\cF^\cH_{\lambda+n\lambda_J}\otimes M_{\lambda+n\lambda_J}
 \end{equation}
 for each $n\in\ZZ$. This shows that each $\cF^\cH_{\lambda+n\lambda_J}\otimes M_{\lambda+n\lambda_J}$ is a grading-restricted generalized $\cH\otimes\cM(p)$-module, and then the concluding argument in the proof of \cite[Proposition 3.7]{CKLR} shows that each surjection \eqref{eqn:like_simple_current} is actually an isomorphism. Thus we have the following isomorphisms in $\ind(\cC)$:
\begin{align*}
X \cong\bigoplus_{n\in\ZZ} \cF^\cH_{\lambda+n\lambda_J}\otimes M_{\lambda+n\lambda_J}\cong\bigoplus_{n\in\ZZ} J^n\tens(\cF^\cH_\lambda\otimes M_\lambda)\cong\cF(\cF^\cH_\lambda\otimes M_\lambda).
\end{align*}
This shows that the simultaneous (generalized) $L(0)$- and $h(0)$-eigenspaces of the domain and codomain of $F$ have the same (finite) dimension. Thus $F$ is injective as well as surjective, and therefore a $\repA$-isomorphism.
\end{proof}

Proposition \ref{prop:general_simple_module_classification} shows that all simple objects in $\repA$ are inductions of simple objects in $\cC$. The same holds for indecomposable projective objects:  
\begin{prop}\label{prop:general_proj_covers}
Let $P$ be a projective object in $\cC$. Then:
\begin{enumerate}
\item $\cF(P)$ is projective in $\repA$.
\item If $P$ is the projective cover of a simple object $W$ such that $\cF(W)$ is an object in $\rep^k A$ for $k=0$ or $k=1$, then $\cF(P)$ is a projective cover of $\cF(W)$ in $\rep^k A$.
\item If $P$ is a projective cover of a simple object $W$, then the socle series and the corresponding Loewy diagram of $\cF(P)$ are obtained by replacing each composition factor of $P$ by its image under the induction functor. 
\end{enumerate}
\end{prop}
\begin{proof}
The proof of $(1)$ is exactly the same as that of \cite[Lemma~5.0.3]{CMY3}. For $(2)$, it is enough, as in the proof of \cite[Proposition 5.0.4]{CMY3}, to show $\cF(P)$ is indecomposable and an object of $\rep^k A$. Since $W=\cF_{\lambda}^{\cH} \otimes M$ for some $\lambda \in \CC$ and an irreducible $\cM(p)$-module $M$, Corollary \ref{cor:proj_cover_in_C} shows that $P = \cF_{\lambda}^{\cH} \otimes P_M$ where $P_M$ is a projective cover of $M$ in $\cO_{\cM(p)}^T$. If $M$ is a Fock module, then $P_M = M$ and therefore $\cF(P) = \cF(W)$ is a projective cover of $\cF(W)$ in $\rep^k A$. 

It remains to show that $\cF(P)$ is an indecomposable object of $\rep^k A$ when $M=\cM_{r,s}$ for $r\in\ZZ$, $1\leq s\leq p-1$. For the indecomposability, it is enough to prove that the socle of $\cF(P)$ is $\cF(W)$. In fact, for any simple object $S$ in $\cC$, Frobenius reciprocity and the fusion rules in \cite[Theorem~5.2.1]{CMY-singlet} show that $\cF(S)$ appears as a submodule of $\cF(P)$, with multiplicity one, if and only if $S \cong J^n \tens W$ for some $n \in \ZZ$. Moreover, $\cF(J^n\tens W)\cong\cF(W)$ by Proposition \ref{prop:general_simple_module_classification}. Consequently, $\mathrm{Soc}(\cF(P)) \cong \cF(W)$, showing $\cF(P)$ is indecomposable. To show that $\cF(P)$ is an object of $\rep^k A$, note that \cite[Theorem 1.4(1)]{CKL} in the case $k=0$, and its easy generalization to the case $k=1$, imply that the same conformal weight criterion in Proposition \ref{prop:general_simple_module_classification} for $W$ to induce to an object of $\rep^k A$ also applies to its projective cover $P$, provided that $P$ is a subquotient of a tensor product of simple objects of $\cC$. But this follows easily from the fusion rules in \cite[Theorems 3.2.8 and 5.1.4]{CMY-singlet}.

The conclusion of $(3)$ is obvious if $M$ is a Fock module, so we assume $M \cong \cM_{r,s}$ for some $r \in \ZZ$ and $1 \leq s \leq p-1$. From \eqref{eqn:Prs_Loewy_diag}, the Loewy diagram of $P$ is
\begin{equation*}
 \begin{matrix}
  \begin{tikzpicture}[->,>=latex,scale=1.5]
\node (b1) at (1,0) {$\cF_{\lambda}^{\cH} \otimes \cM_{r,s}$};
\node (c1) at (-1, 1){};
   \node (a1) at (0,1) {$\cF_{\lambda}^{\cH} \otimes \cM_{r-1,p-s}$};
   \node (b2) at (2,1) {$\cF_{\lambda}^{\cH} \otimes \cM_{r+1,p-s}$};
    \node (a2) at (1,2) {$\cF_{\lambda}^{\cH} \otimes \cM_{r,s}$};
\draw[] (b1) -- node[left] {} (a1);
   \draw[] (b1) -- node[left] {} (b2);
    \draw[] (a1) -- node[left] {} (a2);
    \draw[] (b2) -- node[left] {} (a2);
\end{tikzpicture}
\end{matrix} .
 \end{equation*}
 In the proof of $(2)$, we have shown that $\mathrm{Soc}(\cF(P)) = \cF(W)$. Next, we apply the induction functor $\cF$ to the exact sequence
 \[
 0 \longrightarrow (\cF_{\lambda}^{\cH} \otimes \cM_{r-1,p-s}) \oplus (\cF_{\lambda}^{\cH} \otimes \cM_{r+1,p-s} ) \longrightarrow P/W \longrightarrow W \longrightarrow 0
 \]
 and get the exact sequence
 \[
 0 \longrightarrow \cF(\cF_{\lambda}^{\cH} \otimes \cM_{r-1,p-s}) \oplus \cF(\cF_{\lambda}^{\cH} \otimes \cM_{r+1,p-s} ) \longrightarrow \cF(P)/\cF(W) \longrightarrow \cF(W) \longrightarrow 0.
 \]
 This sequence does not split because by exactness of induction and Frobenius reciprocity,
 \begin{align*}
 &\hom_{\repA}(\cF(P)/\cF(W), \cF(\cF_{\lambda}^{\cH} \otimes \cM_{r-1,p-s}) \oplus \cF(\cF_{\lambda}^{\cH} \otimes \cM_{r+1,p-s} ) ) \\
 &\qquad \cong \hom_{\cC}(P/W, \bigoplus_{n \in \ZZ}((\cF_{\lambda + n\lambda_J}^{\cH} \otimes \cM_{r+nr_J-1,p-s}) \oplus (\cF_{\lambda + n\lambda_J}^{\cH} \otimes \cM_{r+nr_J+1,p-s}))) = 0,
 \end{align*}
so $\mathrm{Soc}(\cF(P)/\cF(W)) = \cF(\cF_{\lambda}^{\cH} \otimes \cM_{r-1,p-s}) \oplus \cF(\cF_{\lambda}^{\cH} \otimes \cM_{r+1,p-s} )$. The socle series and Loewy diagram of $\cF(P)$ then follow easily.
\end{proof}

\begin{rem}
Replacing $\cO_{\cM(p)}^T$ with the larger category $\cO_{\cM(p)}$, we can also consider $A$ as a commutative algebra in the braided tensor category $\ind(\cO_\cH\tens\cO_{\cM(p)})$ of generalized $\cH\otimes\cM(p)$-modules. All objects of this direct limit completion have the form $\bigoplus_{\lambda\in\CC} \cF^\cH_\lambda\otimes M_\lambda$ where each $M_\lambda$ is a generalized $\cM(p)$-module in $\ind(\cO_{\cM(p)})$. However, this larger direct limit completion will not contain non-zero projective objects.
\end{rem}

We now define the main categories of $A$-modules that are of interest to us:
\begin{defi}
 A \textit{strongly $\cH$-weight-graded ($\theta$-twisted) $A$-module} is a generalized ($\theta$-twisted) $A$-module $X$ such that:
 \begin{enumerate}
  \item The Heisenberg zero-mode $h(0)$ acts semisimply on $X$, that is, $X=\bigoplus_{\lambda,h\in\CC} X^{(\lambda)}_{[h]}$ where $X^{(\lambda)}=\bigoplus_{h\in\CC} X^{(\lambda)}_{[h]}$ is the $h(0)$-eigenspace with eigenvalue $\lambda$ and $X_{[h]}=\bigoplus_{{\lambda}\in\CC} X^{(\lambda)}_{[h]}$ the generalized $L(0)$-eigenspace with eigenvalue $h$.
  
  \item For any $\lambda,h\in\CC$, $\dim X^{(\lambda)}_{[h]}<\infty$.
  
  \item For any $\lambda, h\in\CC$, $X^{(\lambda)}_{[h-n]}=0$ for all $n\in\ZZ$ sufficiently positive.
 \end{enumerate}
 We define $\Oloc_A$ to be the (super)category of finitely-generated strongly $\cH$-weight-graded $A$-modules, $\Otw_A$ to be the (super)category of finitely-generated strongly $\cH$-weight-graded $\theta$-twisted $A$-modules, and $\cO_A$ to be the direct sum (super)category $\Oloc_A\oplus\Otw_A$.
\end{defi}

\begin{rem}\label{rem:strongly_H_or_C_graded}
 The vertex operator (super)algebra $A$ itself is a strongly $\cH$-weight-graded $A$-module, and any strongly $\cH$-weight graded $A$-module is a strongly $\CC$-graded $A$-module in the sense of Definition \ref{def:strongly_graded}. Moreover, if $X$ is a strongly $\cH$-weight-graded $A$-module with strongly $\CC$-graded contragredient $X'=\bigoplus_{\lambda,h\in\CC} (X_{[h]}^{(-\lambda)})^*$, then $X'$ is also strongly $\cH$-weight graded because \eqref{eqn:contra_action} and the fact that $h$ has conformal weight $1$ imply
 \begin{equation*}
 \langle h(0)b',b\rangle =-\langle b',h(0)b\rangle
 \end{equation*}
 for any $b\in X$, $b'\in X'$.
\end{rem}

If $X$ is a strongly $\cH$-weight-graded ($\theta$-twisted) $A$-module, then the grading restrictions and \cite[Theorem 1.7.3]{FLM} imply that the $h(0)$-eigenspace $X^{(\lambda)}$ is isomorphic to $\cF_\lambda^\cH\otimes M_\lambda$ as a weak $\cH$-module, where
\begin{equation*}
 M_\lambda=\lbrace b\in X\,\,\vert\,\,h(n)b=\delta_{n,0}\lambda b\,\,\text{for all}\,\,n\in\ZZ_{\geq 0}\rbrace.
\end{equation*}
By the grading restrictions, $M_\lambda$ is a grading-restricted generalized $\cM(p)$-module (see for example \cite[Theorem 2.6]{McR-cosets}). We now present some basic properties of strongly $\cH$-weight-graded $A$-modules:

\begin{lem}\label{lem:Bp_mod_in_ind_C}
 Every strongly $\cH$-weight-graded $A$-module (respectively, $\theta$-twisted $A$-module) restricts to a generalized $\cH\otimes\cM(p)$-module in $\ind(\cC)$ and thus is an object of $\rep^0A$ (respectively, of $\rep^1A$).
\end{lem}
\begin{proof}
 Let $X=\bigoplus_{\lambda\in\CC} \cF^\cH_\lambda\otimes M_\lambda$ be a strongly $\cH$-weight-graded ($\theta$-twisted) $A$-module, where each $M_\lambda$ is a grading-restricted generalized $\cM(p)$-module. By Proposition \ref{prop:gen_Mp_mod_in_DLC}, each $M_\lambda$ is an object of $\ind(\cO_{\cM(p)})$, so that $X$ restricts to an $\cH\otimes\cM(p)$-module in the braided tensor category $\ind(\cO_\cH\tens\cO_{\cM(p)})$. Thus $X$ is a local (or $\theta$-twisted) module for $A$ considered as a commutative algebra in $\ind(\cO_\cH\tens\cO_{\cM(p)})$.

 We need each $M_\lambda$ to be an object of $\ind(\cO_{\cM(p)}^T)$. For any finite-length submodule $M\subseteq M_\lambda$, $\cF^\cH_\lambda\otimes M$ is a finite-length $\cH\otimes\cM(p)$-submodule of $X$. Thus by Lemma \ref{lem:monodromy},
 \begin{align*}
  \pm\Id_{J\tens(\cF^\cH_\lambda\otimes M)} &=\cR^2_{J,\cF^\cH_\lambda\otimes M}=\cR^2_{\cF^\cH_{\lambda_J}, \cF^\cH_\lambda}\otimes\cR^2_{\cM_{r_J+1,1},M} = e^{2\pi i\lambda_J\lambda}(\Id_{\cF^\cH_{\lambda_J+\lambda}}\otimes\cR^2_{\cM_{r_J+1,1},M}).
 \end{align*}
It follows that
\begin{equation*}
 \cR^2_{\cM_{r_J+1,1},M}=\pm e^{-2\pi i\lambda_J\lambda}\Id_{\cM_{r_J+1,1}\tens M}.
\end{equation*}
Now exactly as in the conclusion of the proof of \cite[Proposition 6.4]{CMY-singlet-typical}, consideration of open Hopf link endomorphisms of $M$ shows that, if $M$ is indecomposable (as we may assume), then $\cR_{\cM_{2,1},M}^2$ is an $r_J$th root of $\pm e^{-2\pi i\lambda_J\lambda}$. Since $\cO_{\cM(p)}^T$ consists precisely of the finite-length grading-restricted generalized $\cM(p)$-modules $M$ such that $\cR_{\cM_{2,1},M}^2$ is diagonalizable, this shows that every finite-length submodule of $M_\lambda$ is an object of $\cO_{\cM(p)}^T$. Since $M_\lambda$ is an object of $\cO_{\cM(p)}$, it is the union of its finite-length submodules, and thus it is the union of its $\cO_{\cM(p)}^T$-submodules. This means $M_\lambda$ is an object of $\ind(\cO_{\cM(p)}^T)$ for all $\lambda\in\CC$, and thus $X$ (considered as an $\cH\otimes\cM(p)$-module) is an object of $\ind(\cC)$.
\end{proof}

\begin{prop}\label{prop:OlocA_OtwA_characterization}
 The category $\Oloc_A$, respectively $\Otw_A$, of finitely-generated strongly $\cH$-weight-graded $A$-modules, respectively $\theta$-twisted $A$-modules, is equal to the subcategory of finitely-generated, equivalently finite-length, modules in $\rep^0A$, respectively $\rep^1A$.
\end{prop}

\begin{proof}
Any object of $\Oloc_A$, respectively $\Otw_A$, is a finitely-generated, equivalently finite-length, module in $\rep^0A$, respectively $\rep^1A$, by Lemma \ref{lem:Bp_mod_in_ind_C} and Proposition \ref{prop:fin_gen_Rep_B_p}.
Conversely, if $X$ is a finitely-generated module in $\rep^0A$ or $\rep^1A$, then as an $\cH\otimes\cM(p)$-module, $X=\bigoplus_{\lambda\in\CC}\cF_\lambda^\cH\otimes M_\lambda$ where each $M_\lambda$ is a generalized $\cM(p)$-module in $\ind(\cO_{\cM(p)}^T)$. Thus it is sufficient to show that each $M_\lambda$ is a grading-restricted generalized $\cM(p)$-module. This follows from the second and third paragraphs of the proof of Proposition \ref{prop:fin_gen_Rep_B_p}, which shows that each $\cF^\cH_\lambda\otimes M_\lambda$ is finitely generated and thus an object of $\cC$.
\end{proof}

We can now establish the main result of this section:
\begin{thm}\label{thm:main_thm}
The category $\Oloc_A$ of finitely-generated strongly $\cH$-weight-graded $A$-modules is a rigid braided tensor (super)category, with the vertex algebraic braided tensor category structure of \cite{HLZ8} and duals given by contragredient modules as defined in Section \ref{sec:prelim}. Moreover, $\cO_A=\Oloc_A\oplus\Otw_A$ is a rigid braided $\ZZ/2\ZZ$-crossed tensor (super)category. If $A$ is $\ZZ$-graded by conformal weights, then $e^{2\pi i L(0)}$ defines a ribbon twist on $\Oloc_A$, so that $\Oloc_A$ is a ribbon tensor (super)category.
\end{thm}
\begin{proof}
Recall that $\rep^0A$ is a braided tensor subcategory of $\repA$ and $\rep^0 A\oplus\rep^1A$ is a braided $\ZZ/2\ZZ$-crossed tensor subcategory of $\repA$. Moreover, Propositions \ref{prop:fin_gen_Rep_B_p} and \ref{prop:OlocA_OtwA_characterization} show that  $\Oloc_A$, respectively $\cO_A$, is the intersection of $\rep^0A$, respectively $\rep^0A\oplus\rep^1A$, with the image of the induction functor $\cF:\cC\rightarrow\repA$. Thus because induction is monoidal, $\Oloc_A$, respectively $\cO_A$, is closed under the tensor product on $\repA$. This shows $\Oloc_A$ is a braided tensor subcategory of $\rep^0A$, and $\cO_A$ is a braided $\ZZ/2\ZZ$-crossed subcategory of $\rep^0A\oplus\rep^1A$.
Moreover, the braided tensor category structure on $\Oloc_A$ is the vertex algebraic one of \cite{HLZ8} because this is the braided tensor structure on $\rep^0A$ by \cite[Theorem 3.65]{CKM-exts} and \cite[Theorem 7.7]{CMY-completions}. 

The tensor categories $\Oloc_A$ and $\cO_A$ are rigid because induction is monoidal and every object in these categories is isomorphic to the induction of a (rigid) object of $\cC$ by Propositions \ref{prop:fin_gen_Rep_B_p} and \ref{prop:OlocA_OtwA_characterization}. By the discussion concluding Section \ref{sec:prelim}, duals in $\Oloc_A$ will be given by contragredients provided that $A$ is self-contragredient and that $\Oloc_A$ is closed under contragredients. The second of these conditions is addressed in Remark \ref{rem:strongly_H_or_C_graded}, and $A$ is self-contragredient by the same proof as in Theorem 3.1 and Corollary 3.2 of \cite{Li} for the $\ZZ$-graded vertex operator algebra case. That is, the linear functional $f\in (A')^{(0)}_{(0)} =(\CC\vac)^*$ defined by $f(\vac)=1$ is a vacuum-like vector satisfying $L(-1)f=0$ since $L(1)(\cH\otimes\cM(p))_{(1)}=0$. Then $a\mapsto a_{-1} f$ is a non-zero $A$-module homomorphism from $A$ to $A'$ which must be an isomorphism since $A$ is simple.

Finally, $e^{2\pi i L(0)}$ defines a ribbon twist on $\Oloc_A$ in the case that $A$ is $\ZZ$-graded exactly as in the proof of \cite[Theorem 4.1]{Hu-rig-mod}.
\end{proof}

\begin{rem}
When $A$ is $\frac{1}{2}\ZZ$-graded, $e^{2\pi i L(0)}$ fails to define a twist on $\Oloc_A$ because it does not act as the identity on $A$, and in fact it does not define $A$-module endomorphisms of $A$-modules in general. When $A$ is a $\frac{1}{2}\ZZ$-graded superalgebra, we can define a new twist by $P e^{2\pi i L(0)}$, where $P$ is the parity involution of any $A$-module. This isomorphism acts as the identity on $A$ and satisfies the properties of a ribbon twist, except that it is not quite a natural transformation in the usual sense, since
\begin{equation*}
P e^{2\pi i L(0)}\circ f =(-1)^{\vert f\vert}f\circ P e^{2\pi i L(0)}
\end{equation*}
for parity-homogeneous morphisms $f$ in $\Oloc_A$. Thus strictly speaking, $Pe^{2\pi i L(0)}$ only gives the underlying category $\underline{\Oloc_A}$, which has the same objects as $\Oloc_A$ but retains only the even morphisms, the structure of a ribbon tensor category (not supercategory!).
\end{rem}

\section{Detailed structure of \texorpdfstring{$\cO_A$}{OA}}\label{sec:detailed_structure}

We continue to fix a simple vertex operator (super)algebra extension $A=\bigoplus_{n\in\ZZ} J^n$ of $\cH\otimes\cM(p)$, where $J=\cF^\cH_{\lambda_J}\otimes\cM_{r_J+1,1}$ for some $\lambda_J\in\CC^\times$ and $r_J\geq 0$ such that $\lambda_J^2+\frac{p}{2}r_J^2\in\ZZ_{\geq 0}$. Thus $J^n\cong \cF^\cH_{n\lambda_J}\otimes\cM_{nr_J+1,1}$ for any $n\in\ZZ$. In this section, we describe the braided tensor (super)category $\Oloc_A$ and its $\ZZ/2\ZZ$-graded extension $\cO_A=\Oloc_A\oplus\Otw_A$  in detail, beginning with the classification of their simple objects.

\subsection{Simple modules, projective modules, and tensor product formulas}

 In this subsection, we classify the simple objects of $\cO_A$ using Proposition \ref{prop:general_simple_module_classification} and the projective objects using Proposition \ref{prop:general_proj_covers}. We also compute the tensor products of simple objects in $\cO_A$. For the following theorem, recall the notation $\alpha_+=\sqrt{2p}$, $\alpha_-=-\sqrt{2/p}$, $\alpha_0=\alpha_++\alpha_-$, $L=\ZZ\alpha_+$, and $L^\circ=\ZZ\frac{\alpha_-}{2}$:
\begin{thm}\label{thm:simple_A-module_classification}
Simple objects in $\cO_A$ and $\Oloc_A$ are as follows:
\begin{enumerate}
\item Every simple object of $\cO_A$ is isomorphic to at least one of the following induced modules:
\begin{equation*}
\cW^{(\ell)}_{r,s} := \cF\left(\cF^\cH_{(r_J(1-r)p/2+\ell)/\lambda_J}\otimes\cM_{r,s}\right)
\end{equation*}
for $\ell\in\frac{1}{2}\ZZ$, $r\in\ZZ$, and $1\leq s\leq p$; or
\begin{equation*}
\cE^{(\ell)}_\lambda := \cF\left(\cF^\cH_{(r_J(\alpha_+\lambda-p+1)/2+\ell)/\lambda_J}\otimes\cF_\lambda\right)
\end{equation*}
for $\ell\in\frac{1}{2}\ZZ$ and $\lambda\in\CC\setminus L^\circ$.

\item The following are all isomorphisms between simple modules from part (1):
\begin{align*}
\cW_{r,s}^{(\ell)} & \cong \cW_{r+nr_J,s}^{(\ell+n(\lambda_J^2+r_J^2p/2))},\qquad\quad r\in\ZZ, \,\,\, 1\leq s\leq p,\,\,\, \ell\in\frac{1}{2}\ZZ,\,\,\, n\in\ZZ,\\
\cE^{(\ell)}_\lambda & \cong \cE_{\lambda-nr_J\alpha_+/2}^{(\ell+n(\lambda_J^2+r_J^2 p/2))}, \qquad\quad \lambda\in\CC\setminus L^\circ,\,\,\, \ell\in\frac{1}{2}\ZZ,\,\,\, n\in\ZZ.
\end{align*}
In particular, if $r_J>0$, then the modules $\cW_{r,s}^{(\ell)}$ for $1\leq r\leq r_J$ and $\cE^{(\ell)}_\lambda$ for $0\leq\mathrm{Re}(\lambda)<\frac{r_J}{2}\alpha_+$ give a complete list of inequivalent simple objects in $\cO_A$.

\item The module $\cW_{r,s}^{(\ell)}$ is an object of $\Oloc_A$ if and only if $\ell\in \frac{r_J}{2}(s-1)+\ZZ$, and $\cE_\lambda^{(\ell)}$ is an object of $\Oloc_A$ if and only if $\ell\in \frac{r_J}{2}(p-1)+\ZZ$.

\end{enumerate}
\end{thm}
\begin{proof}
To simplify notation, let $\cM_\lambda$ for $\lambda\in\CC$ denote the $\cM(p)$-socle of $\cF_\lambda$, so $\cM_\lambda=\cM_{r,s}$ if $\lambda=\alpha_{r,s}$ for some $r\in\ZZ$ and $1\leq s\leq p$, and $\cM_\lambda=\cF_\lambda$ otherwise. Then by Proposition \ref{prop:general_simple_module_classification}, every simple object of $\cO_A$ is isomorphic to $\cF(\cF^\cH_\gamma\otimes\cM_\lambda)$ for some $\gamma,\lambda\in\CC$ such that 
\begin{equation*}
\frac{1}{2}((\lambda_J+\gamma)^2-\lambda_J^2-\gamma^2)+h_{\alpha_{r_J+1,1}+\lambda}-h_{\alpha_{r_J+1,1}}-h_\lambda\in\frac{1}{2}\ZZ,
\end{equation*}
where $h_\lambda=\frac{1}{2}\lambda(\lambda-\alpha_0)$ for $\lambda\in\CC$. Recalling that $\alpha_{r_J+1,1}=-\frac{r_J}{2}\alpha_+$, this simplifies to 
\begin{equation*}
\lambda_J\gamma-\frac{r_J}{2}\alpha_+\lambda\in\frac{1}{2}\ZZ,
\end{equation*}
or
\begin{equation*}
\gamma = \frac{1}{\lambda_J}\left(\frac{r_J}{2}\alpha_+\lambda +k\right)
\end{equation*}
for $k\in\frac{1}{2}\ZZ$. Moreover, $\cF(\cF^\cH_\gamma\otimes\cM_\lambda)$ for such $\gamma$ is an object of $\Oloc_A$ if and only if $k\in \ZZ$. In case $\lambda=\alpha_{r,s}$ for some $r\in\ZZ$, $1\leq s\leq p$, we have
\begin{equation*}
\gamma=\frac{1}{\lambda_J}\left[\frac{r_J}{2}((1-r) p+s-1)+k\right],
\end{equation*}
recalling that $\alpha_{r,s}=\frac{1-r}{2}\alpha_++\frac{1-s}{2}\alpha_-$. For later use, we adjust the labeling by setting $\ell=k+\frac{r_J}{2}(s-1)$ if $\lambda=\alpha_{r,s}$ for some $r\in\ZZ$, $1\leq s\leq p$, and $\ell=k+\frac{r_J}{2}(p-1)$ if $\lambda\in\CC\setminus L^\circ$. Thus every simple object of $\cO_A$ is isomorphic to at least one of the modules $\cW^{(\ell)}_{r,s}$ or $\cE^{(\ell)}_\lambda$ defined in the statement of the  theorem. Moreover, $\cW_{r,s}^{(\ell)}$ is an object of $\Oloc_A$ if and only if $\ell\in \frac{r_J}{2}(s-1)+\ZZ$, and $\cE_\lambda^{(\ell)}$ is an object of $\Oloc_A$ if and only if $\ell\in \frac{r_J}{2}(p-1)+\ZZ$. It remains to determine which of these simple modules are isomorphic to each other.

Since $J^n\cong\cF^\cH_{n\lambda_J}\otimes\cM_{nr_J+1,1}$, Proposition \ref{prop:general_simple_module_classification} says that $\cF(\cF^\cH_\gamma\otimes\cM_\lambda)\cong\cF(\cF^\cH_{\til{\gamma}}\otimes\cM_{\til{\lambda}})$ if and only if $\til\gamma=\gamma+n\lambda_J$ and $\til{\lambda}=\lambda+\alpha_{nr_J+1,1}$ for some $n\in\ZZ$. Since 
\begin{equation}\label{eqn:Jn_times_Mrs}
\cM_{nr_J+1,1}\tens\cM_{r,s} \cong \cM_{r+nr_J,s},
\end{equation}
we get
\begin{align*}
\cW^{(\ell)}_{r,s} & \cong \cF\left(\cF^{\cH}_{(r_J(1-r)p/2+\ell)/\lambda_J+n\lambda_J}\otimes\cM_{r+nr_J,s}\right)\nonumber\\
& =\cF\left(\cF^\cH_{(r_J(1-(r+nr_J))p/2+\ell+nr_J^2p/2)/\lambda_J+n\lambda_J}\otimes\cM_{r+nr_J,s}\right) =\cW_{r+nr_J,s}^{(\ell+n(\lambda_J^2+r_J^2 p/2))}
\end{align*}
for $r\in\ZZ$, $1\leq s\leq p$, $\ell\in\frac{1}{2}\ZZ$, and $n\in\ZZ$. Also since
\begin{equation}\label{eqn:alpha_Jn}
\alpha_{nr_J+1,1}=-n\frac{r_J}{2}\alpha_+,
\end{equation}
we get
\begin{align*}
\cE^{(\ell)}_\lambda & \cong \cF\left(\cF^\cH_{(r_J(\alpha_+\lambda-p+1)/2+\ell)/\lambda_J+n\lambda_J}\otimes\cF_{\lambda-nr_J\alpha_+/2}\right)\nonumber\\
& = \cF\left(\cF^\cH_{(r_J(\alpha_+(\lambda-n r_J\alpha_+/2)-p+1)/2+\ell+n r_J^2 p/2]/\lambda_J+n\lambda_J}\otimes\cF_{\lambda-nr_J\alpha_+/2}\right) = \cE_{\lambda-nr_J\alpha_+/2}^{(\ell+n(\lambda_J^2+r_J^2 p/2))}
\end{align*}
for $\lambda\in\CC\setminus L^\circ$, $\ell\in\frac{1}{2}\ZZ$, and $n\in\ZZ$.
\end{proof}

\begin{rem}\label{rem:spectral_flow}
The half-integer label $\ell$ in the notation for the simple modules in $\cO_A$ is a spectral flow index: changing $\ell$ by $1$ corresponds to tensoring with $\cW^{(1)}_{1,1}=\cF(\cF^\cH_{1/\lambda_J}\otimes\cM_{1,1})$, which is a simple current $A$-module since $\cF^\cH_{1/\lambda_J}\otimes\cM_{1,1}$ is a simple current $\cH\otimes\cM(p)$-module. Note that shifting $\ell$ by an integer preserves both local and twisted $A$-modules in $\cO_A$, while shifting $\ell$ by a strict half-integer exchanges local and twisted modules.
\end{rem}

\begin{rem}\label{rem:mod_class_for_Bp_and_orbs}
The isomorphisms in Theorem \ref{thm:simple_A-module_classification}(2) show that if $\lambda_J^2+\frac{p}{2}r_J^2=0$, then $(\CC/\frac{r_J}{2} L)\times\frac{1}{2}\ZZ$ naturally parametrizes the simple objects of $\cO_A$: For a coset $\overline{\lambda}\in(\CC\setminus L^\circ)/\frac{r_J}{2} L$ and $\ell\in\frac{1}{2}\ZZ$, we define $\cE_{\overline{\lambda}}^{(\ell)}=\cE_{\lambda}^{(\ell)}$ for any coset representative $\lambda\in\overline{\lambda}$; by Theorem \ref{thm:simple_A-module_classification}(2), this module is independent of the choice of $\lambda\in\overline{\lambda}$ exactly when $\lambda_J^2+\frac{p}{2}r_J^2=0$. Similarly, for a coset $\overline{r}\in\ZZ/r_J\ZZ$, we can define $\cW^{(\ell)}_{\overline{r},s}=\cW_{r,s}^{(\ell)}$ for any coset representative $r\in\overline{r}$.
\end{rem}

The shift from the label $k$ to the label $\ell$ in the proof of Theorem \ref{thm:simple_A-module_classification} leads to nice formulas for the duals of simple objects in $\cO_A$:
\begin{prop}
Duals of simple objects in $\cO_A$ are as follows: $(\cW_{r,s}^{(\ell)})'\cong \cW_{2-r,s}^{(-\ell)}$ for $r\in\ZZ$, $1\leq s\leq p$, $\ell\in\frac{1}{2}\ZZ$; and $(\cE_\lambda^{(\ell)})'\cong\cE_{\alpha_0-\lambda}^{(-\ell)}$ for $\lambda\in\CC\setminus L^\circ$, $\ell\in\frac{1}{2}\ZZ$.
\end{prop}
\begin{proof}
Since duals of induced modules are the inductions of duals, we have
\begin{align*}
(\cW_{r,s}^{(\ell)})' & \cong\cF\left(\cF^\cH_{-(r_J(1-r)p/2+\ell)/\lambda_J}\otimes\cM_{r,s}'\right) \cong\cF\left(\cF^\cH_{(r_J(1-(2-r))p/2-\ell)/\lambda_J}\otimes\cM_{2-r,s}\right) =\cW^{(-\ell)}_{2-r,s}
\end{align*}
and
\begin{align*}
(\cE_\lambda^{(\ell)})' & \cong \cF\left(\cF^\cH_{-(r_J(\alpha_+\lambda-p+1)/2+\ell)/\lambda_J}\otimes\cF_\lambda'\right)\nonumber\\ &\cong\cF\left(\cF^\cH_{(r_J(\alpha_+(\alpha_0-\lambda)-p+1)/2-\ell)/\lambda_J}\otimes\cF_{\alpha_0-\lambda}\right) =\cE^{(-\ell)}_{\alpha_0-\lambda}
\end{align*}
since $\alpha_+\alpha_0=2p-2$.
\end{proof}

Now that we have described the simple objects of $\cO_A$, it is straightforward to use Proposition \ref{prop:general_proj_covers} to describe the indecomposable projective objects of $\cO_A$:
\begin{thm}\label{thm:gen_proj_covers}
The simple ($\theta$-twisted) $A$-modules $\cW_{r,p}^{(\ell)}$ and $\cE_\lambda^{(\ell)}$ for $r\in\ZZ$, $\lambda\in\CC\setminus L^\circ$, and $\ell\in\frac{1}{2}\ZZ$ are projective in $\cO_A$. For $r\in\ZZ$, $1\leq s\leq p-1$, and $\ell\in\frac{1}{2}\ZZ$, the simple ($\theta$-twisted) $A$-module $\cW_{r,s}^{(\ell)}$ has a projective cover
\begin{equation*}
\cQ_{r,s}^{(\ell)} := \cF\left(\cF^\cH_{(r_J(1-r)p/2+\ell)/\lambda_J}\otimes\cP_{r,s}\right)
\end{equation*}
in $\cO_A$ which has Loewy diagram
\begin{equation}\label{eqn:Q_rsl_Loewy_diag}
 \begin{matrix}
  \begin{tikzpicture}[->,>=latex,scale=1.5]
\node (b1) at (1,0) {$\cW_{r,s}^{(\ell)}$};
\node (c1) at (-1, 1){$\mathcal{Q}_{r,s}^{(\ell)}:$};
   \node (a1) at (0,1) {$\cW_{r-1,p-s}^{(\ell-r_J p/2)}$};
   \node (b2) at (2,1) {$\cW_{r+1,p-s}^{(\ell+r_J p/2)}$};
    \node (a2) at (1,2) {$\cW_{r,s}^{(\ell)}$};
\draw[] (b1) -- node[left] {} (a1);
   \draw[] (b1) -- node[left] {} (b2);
    \draw[] (a1) -- node[left] {} (a2);
    \draw[] (b2) -- node[left] {} (a2);
\end{tikzpicture}
\end{matrix} .
 \end{equation}
\end{thm}
\begin{proof}
Everything follows from Proposition \ref{prop:general_proj_covers} and the description of the projective objects in $\cO_{\cM(p)}^T$, especially the Loewy diagram \eqref{eqn:Prs_Loewy_diag} of the $\cM(p)$-module $\cP_{r,s}$. In particular, the middle two composition factors of the induced module $\cQ_{r,s}^{(\ell)}$ are
\begin{align*}
&\cF\left(\cF^\cH_{(r_J(1-r)p/2+\ell)/\lambda_J}\otimes\cM_{r\pm 1, p-s}\right)\nonumber\\
 & \hspace{5em} = \cF\left(\cF^\cH_{(r_J(1-(r\pm 1))p/2+\ell\pm r_J p/2)/\lambda_J}\otimes\cM_{r\pm 1,p-s}\right) =\cW_{r\pm 1, p-s}^{(\ell\pm r_J p/2)}
\end{align*}
for $r\in\ZZ$, $1\leq s\leq p-1$, and $\ell\in\frac{1}{2}\ZZ$.
\end{proof}

It is also straightforward to use the $\cM(p)$-module fusion rules of Theorem \ref{thm:singlet_fus_rules} together with the fact that induction is a tensor functor to compute all tensor products of simple objects in $\cO_A$. For simpler formulas, we sometimes use the alternate notation $\cQ_{r,p}^{(\ell)}$ for $\cW_{r,p}^{(\ell)}$:
\begin{thm}\label{thm:general_fusion_rules}
The following tensor product formulas hold in $\cO_A$:
\begin{enumerate}
\item For $r,r'\in\ZZ$, $1\leq s,s'\leq p$, and $\ell,\ell'\in\frac{1}{2}\ZZ$,
 \begin{equation*}
   \cW_{r,s}^{(\ell)}\boxtimes \cW_{r',s'}^{(\ell')}   \cong \bigg(\bigoplus_{\substack{k = |s-s'|+1 \\ k+s+s' \equiv 1\; (\mathrm{mod}\; 2)}}^{{\rm min}\{s+s'-1, 2p-1-s-s'\}}\cW_{r+r'-1, k}^{(\ell+\ell')}\bigg) \oplus \bigg(\bigoplus_{\substack{k = 2p+1-s-s'\\ k+s+s' \equiv 1\; (\mathrm{mod}\; 2)}}^{p}\cQ_{r+r'-1, k}^{(\ell+\ell')}\bigg)
  \end{equation*}
with sums taken to be empty if the lower bound exceeds the upper bound. 
 
  \item 
  For $r\in\ZZ$, $1\leq s\leq p$, $\lambda\in\CC\setminus L^\circ$, and $\ell,\ell'\in\frac{1}{2}\ZZ$,
 \begin{equation*}
  \cW_{r,s}^{(\ell)}\tens\cE_\lambda^{(\ell')}\cong\bigoplus_{k=0}^{s-1} \cE_{\lambda+\alpha_{r,s}+k\alpha_-}^{(\ell+\ell'+r_J(k-(s-1)/2))}.
 \end{equation*}

  \item
  For $\ell,\ell'\in\frac{1}{2}\ZZ$ and $\lambda,\mu\in\CC\setminus L^\circ$ such that $\lambda+\mu =\alpha_0+\alpha_{r,s}\in L^\circ$ for some $r\in\ZZ$, $1\leq s\leq p$,
  \begin{equation*}
  \cE_\lambda^{(\ell)}\tens\cE_{\mu}^{(\ell')}\cong\bigoplus_{\substack{t= s\\ t\equiv s\,\,(\mathrm{mod}\,2)\\}}^p \cQ_{r,t}^{(\ell+\ell'+r_J(s-1)/2)}\oplus\bigoplus_{\substack{t=p+2-s\\t\equiv p-s\,\,(\mathrm{mod}\,2)\\}}^p \cQ_{r-1,t}^{(\ell+\ell'-r_J(p-s+1)/2)}.
 \end{equation*}

  \item
  For $\ell,\ell'\in\frac{1}{2}\ZZ$ and $\lambda,\mu\in\CC\setminus L^\circ$ such that $\lambda+\mu\notin L^\circ$,
\begin{equation*}
\cE_{\lambda}^{(\ell)}\tens\cE_\mu^{(\ell')} \cong \bigoplus_{k=0}^{p-1} \cE_{\lambda + \mu + k \alpha_-}^{(\ell+\ell'+r_J(k-(p-1)/2))}.
\end{equation*}
\end{enumerate}
\end{thm}

\begin{rem}
In some cases it is possible to rewrite some of the above tensor product formulas using the isomorphisms of Theorem \ref{thm:simple_A-module_classification}(2) or using the notation of Remark \ref{rem:mod_class_for_Bp_and_orbs}.
\end{rem}

\begin{rem}\label{rem:projective_fusion}
One can similarly obtain formulas for the tensor products of projective objects in $\cO_A$ using the $\cM(p)$-module fusion rules in \cite[Theorem 5.2.1]{CMY-singlet} and \cite[Theorem 4.3]{CMY-singlet-typical}. We leave the computation to the interested reader.
\end{rem}

\subsection{Lower-bounded, highest-weight, and grading-restricted modules}\label{subsec:lb_hw_gr}

We now determine which simple objects of $\cO_A$ satisfy certain grading conditions. We recall that a generalized ($\theta$-twisted) $A$-module is \textit{lower bounded} if its conformal weight spaces satisfy $X_{[h]}=0$ for $\mathrm{Re}\,h$ sufficiently negative and is \textit{grading restricted} if for any $h\in\CC$, $\dim X_{[h]}<\infty$ and $X_{[h-n]}=0$ for $n\in\ZZ$ sufficiently positive. Lower-bounded and grading-restricted generalized ($\theta$-twisted) $A$-modules are not required to have an $h(0)$-eigenvalue grading.

Highest-weight $A$-modules are also of interest. To define them, let $X$ be a strongly $\cH$-weight-graded ($\theta$-twisted) $A$-module, and let $\mathfrak{g}(A)$ be the Lie (super)algebra spanned by vertex operator modes $a_n$ acting on $X$ for $a\in A$ and $n\in\ZZ$ (if $X$ is local) or $n\in\frac{1}{2}\ZZ$ (if $X$ is $\theta$-twisted). We have a triangular decomposition $\mathfrak{g}(A)=\mathfrak{g}(A)_+\oplus\mathfrak{g}(A)_0\oplus\mathfrak{g}(A)_-$ where
\begin{equation*}
\mathfrak{g}(A)_{\pm} =\mathrm{span}\bigg\lbrace a_n\in\mathfrak{g}(A)\,\,\bigg\vert\,\, n-\mathrm{wt}\,a+1\in\pm\ZZ_{\geq 0}\,\,\text{and}\,\,a\in\bigoplus_{n\in\pm\ZZ_{\geq 1}} J^n\,\,\text{if}\,\,n=\mathrm{wt}\,a-1\bigg\rbrace
\end{equation*}
and 
\begin{equation*}
\mathfrak{g}(A)_0=\mathrm{span}\left\lbrace a_n\in\mathfrak{g}(A)\,\,\vert\,\,a\in J^0,\,\,n=\mathrm{wt}\,a-1\right\rbrace.
\end{equation*}
Thus homogeneous modes in $\mathfrak{g}(A)_+$ (respectively, $\mathfrak{g}(A)_-$) weakly lower conformal weights (respectively, weakly raise conformal weights) and strictly raise $\frac{1}{\lambda_J}h(0)$-eigenvalues (respectively, strictly lower $\frac{1}{\lambda_J}h(0)$-eigenvalues) whenever they preserve conformal weights. Modes in $\mathfrak{g}(A)_0$, on the other hand, preserve both conformal weights and $h(0)$-eigenvalues.

\begin{defi}
A \textit{highest-weight vector} of weight $(h,\lambda)$ in a strongly $\cH$-weight-graded ($\theta$-twisted) $A$-module $X$ is a simultaneous $L(0)$- and $h(0)$-eigenvector $b\in X^{(\lambda)}_{[h]}$ such that $\mathfrak{g}(A)_+\cdot b=0$. A \textit{highest-weight ($\theta$-twisted) $A$-module} of highest weight $(h,\lambda)$ is a strongly $\cH$-weight-graded ($\theta$-twisted) $A$-module which is generated by a highest-weight vector of weight $(h,\lambda)$. 
\end{defi}

\begin{rem}\label{rem:hw_in_OA}
Every highest-weight ($\theta$-twisted) $A$-module is an object of $\cO_A$ since it is a singly-generated strongly $\cH$-weight-graded $A$-module.
\end{rem}

If $X$ is a highest-weight local or $\theta$-twisted $A$-module of highest weight $(h,\lambda)$, then $X=U(\mathfrak{g}(A)_-)\cdot X^{(\lambda)}_{[h]}$. Thus we have:
\begin{prop}\label{prop:grading_of_hw_mods}
If $X$ is a highest-weight ($\theta$-twisted) $A$-module of highest weight $(h,\lambda)$, then $X$ is a lower-bounded ($\theta$-twisted) $A$-module with conformal weights in $h+\frac{1}{2}\ZZ_{\geq 0}$. Moreover, $X_{[h]}=\bigoplus_{n=0}^\infty X_{[h]}^{(\lambda-n\lambda_J)}$.
\end{prop}

\begin{rem}\label{rem:simple_grad-rest_are_hw}
Every simple grading-restricted ($\theta$-twisted) $A$-module $X$ is a highest-weight module, and thus also an object of $\cO_A$: Its conformal weights are contained in $h+\frac{1}{2}\ZZ_{\geq 0}$ for some $h\in\CC$, and $X_{[h]}$ is finite dimensional. Thus $X_{[h]}$ decomposes into (generalized) $h(0)$-eigenspaces, and there is an $h(0)$-eigenvector $b\in X_{[h]}$ of eigenvalue $\lambda$ such that $\lambda+n\lambda_J$ is not an $h(0)$-eigenvalue on $X_{[h]}$ for $n\in\ZZ_{\geq 1}$. Then $b$ is a highest-weight vector which generates $X$ because $X$ is simple. Also, $h(0)$ acts semisimply on $X$ because $X$ is generated by an $h(0)$-eigenvector, and the grading restriction conditions of a strongly $\cH$-weight-graded $A$-module hold trivially because the conformal weight spaces of $X$ are finite dimensional.
\end{rem}

We can also define a version of the BGG category $\cO$ for $A$:
\begin{defi}
Define $\cO_{A}^{\mathrm{h.w.}}$ to be the full subcategory of strongly $\cH$-weight-graded local and $\theta$-twisted $A$-modules $X$ whose weights satisfy the following conditions: There are finitely many $h_1,\ldots h_N\in\CC$ such that the conformal weights of $X$ are contained in $\cup_{n=1}^N (h_n+\frac{1}{2}\ZZ_{\geq 0})$, and for any conformal weight $h$ of $X$, there are finitely many $\lambda_1,\ldots\lambda_M\in\CC$ such that $h(0)$-eigenvalues on $X_{[h]}$ are contained in $\cup_{m=1}^M (\lambda_m-\lambda_J\ZZ_{\geq 0})$.
\end{defi}

It is clear from the definition that $\cO_A^{\mathrm{h.w.}}$ is closed under finite direct sums, submodules, and quotients, and that $\cO_A^{\mathrm{h.w.}}$ contains the category of finitely-generated $h(0)$-semisimple grading-restricted generalized ($\theta$-twisted) $A$-modules. However, it is not immediately clear that every highest-weight $A$-module is an object of $\cO_A^{\mathrm{h.w.}}$. But we shall prove this after classifying simple highest-weight $A$-modules in the next theorem.

\begin{thm}\label{thm:grading-restricted}
If $\lambda_J^2+\frac{p}{2}r_J^2>0$, then every object of $\cO_A$ is a grading-restricted generalized ($\theta$-twisted) $A$-module. In particular, $\cO_A$ is equal to the category of finitely-generated $h(0)$-semisimple grading-restricted generalized ($\theta$-twisted) $A$-modules. If $\lambda_J^2+\frac{p}{2}r_J^2=0$, then:

\begin{enumerate}

\item The module $\cE^{(\ell)}_\lambda$ for $\lambda\in\CC\setminus L^{\circ}$, $\ell\in\frac{1}{2}\ZZ$ is lower bounded if and only if $\ell=0$, and $\cW^{(\ell)}_{r,s}$ for $r\in\ZZ$, $1\leq s\leq p$, $\ell\in\frac{1}{2}\ZZ$ is lower bounded if and only if $\vert\ell\vert\leq r_J\frac{p-s}{2}$.

\item Every simple highest-weight ($\theta$-twisted) $A$-module is isomorphic to exactly one of the modules $\cW_{r,s}^{(\ell)}$ such that $1\leq r\leq r_J$, $1\leq s\leq p-1$, and $-r_J\frac{p-s}{2}<\ell\leq r_J\frac{p-s}{2}$. Thus there are $r_J^2 p(p-1)$ simple highest-weight local and $\theta$-twisted $A$-modules up to isomorphism. All of these highest-weight modules are objects of $\cO_A^{\mathrm{h.w.}}$.

\item Every simple grading-restricted ($\theta$-twisted) $A$-module is isomorphic to exactly one of the modules $\cW_{r,s}^{(\ell)}$ such that $1\leq r\leq r_J$, $1\leq s\leq p-1$, and $\vert\ell\vert<r_J\frac{p-s}{2}$. Thus there are $r_J(p-1)(r_Jp-1)$ simple grading-restricted local and $\theta$-twisted $A$-modules up to isomorphism.
\end{enumerate}

\end{thm}

\begin{proof}
We determine which simple objects of $\cO_A$ are lower bounded, highest weight, or grading restricted by computing conformal weights. For $\lambda\in\CC\setminus L^\circ$ and $\ell\in\frac{1}{2}\ZZ$, we have
\begin{equation*}
\cE^{(\ell)}_{\lambda} \cong\bigoplus_{n\in\ZZ} \cF^\cH_{(r_J(\alpha_+\lambda-p+1)/2+\ell)/\lambda_J+n\lambda_J}\otimes\cF_{\lambda-nr_J\alpha_+/2}
\end{equation*}
as an $\cH\otimes\cM(p)$-module, recalling \eqref{eqn:alpha_Jn}. The lowest conformal weight of the $n$th summand here is thus
\begin{align}\label{eqn:h(lambda_ell_n)}
&\frac{1}{2}\left(\frac{r_J(\alpha_+\lambda-p+1)+2\ell}{2\lambda_J}+n\lambda_J\right)^2+\frac{1}{2}\left(\lambda-\frac{nr_J\alpha_+}{2}\right)\left(\lambda-\frac{nr_J\alpha_+}{2}-\alpha_0\right)\nonumber\\
&\hspace{2em}=\frac{n^2}{2}\left(\lambda_J^2+\frac{p}{2}r_J^2\right)+\frac{n}{2}\left(r_J(\alpha_+\lambda-p+1)+2\ell\right)-\frac{n r_J}{2}\left(\alpha_+\lambda-\frac{\alpha_+\alpha_0}{2}\right)+C,
\end{align}
where $C$ is independent of $n$. Since $\frac{\alpha_+\alpha_0}{2}=p-1$, this simplifies to 
\begin{equation*}
\frac{n^2}{2}\left(\lambda_J^2+\frac{p}{2}r_J^2\right)+n\ell+C.
\end{equation*}
Thus if $\lambda_J^2+\frac{p}{2}r_J^2>0$, the lowest conformal weights of the $\cH\otimes\cM(p)$-module summands of $\cE_{\lambda}^{(\ell)}$ are a quadratic function of $n$ with positive leading coefficient, which implies $\cE_{\lambda}^{(\ell)}$ has finite-dimensional weight spaces and a lower bound on conformal weights. If $\lambda_J^2+\frac{p}{2}r_J^2=0$, the conformal weights of $\cE_\lambda^{(\ell)}$ have a lower bound if and only if $\ell=0$. When $\ell=0$, the lowest conformal weight space of $\cE_\lambda^{(0)}$ is infinite dimensional with no upper bound on $\frac{1}{\lambda_J}h(0)$-eigenvalues, so $\cE_\lambda^{(\ell)}$ is never grading restricted or highest weight.

Now consider $\cW_{r,s}^{(\ell)}$ for $1\leq r\leq r_J$, $1\leq s\leq p$, and $\ell\in\frac{1}{2}\ZZ$: as an $\cH\otimes\cM(p)$-module, 
\begin{equation}\label{eqn:Wrsl_weight_decomp}
\cW_{r,s}^{(\ell)}\cong\bigoplus_{n\in\ZZ} \cF^\cH_{(r_J(1-r)p/2+\ell)/\lambda_J+n\lambda_J}\otimes\cM_{r+nr_J,s},
\end{equation} 
recalling \eqref{eqn:Jn_times_Mrs}. The lowest conformal weight of $\cM_{r',s'}$ is $h_{r',s'}$ for $r'\geq 1$ and $h_{2-r',s'}$ for $r'\leq 1$. Thus for $n\geq 0$, the lowest conformal weight of the $n$th summand of $\cW_{r,s}^{(\ell)}$ is
\begin{align}\label{eqn:h(rsl)_n_positive}
&\frac{1}{2}\left(\frac{r_J(1-r)p+2\ell}{2\lambda_J}+n\lambda_J\right)^2 + \frac{\left((r+nr_J)p-s\right)^2}{4p}-\frac{(p-1)^2}{4p}  \nonumber\\
&\hspace{2em} = \frac{n^2}{2}\left(\lambda_J^2+\frac{p}{2}r_J^2\right)+\frac{n}{2}\left(r_J(1-r)p+2\ell+r_J(rp-s)\right)+\frac{(r_J(1-r)p+2\ell)^2}{8\lambda_J^2}+h_{r,s}\nonumber\\
&\hspace{2em} =\frac{n^2}{2}\left(\lambda_J^2+\frac{p}{2}r_J^2\right)+\frac{n}{2}\left(r_J(p-s)+2\ell\right)+D,
\end{align}
where $D$ is the constant term in $n$. For $n<0$ on the other hand, the lowest conformal weight of the $n$th direct summand is
\begin{align}\label{eqn:h(rsl)_n_negative}
&\frac{1}{2}\left(\frac{r_J(1-r)p+2\ell}{2\lambda_J}+n\lambda_J\right)^2 + \frac{\left((2-r-nr_J)p-s\right)^2}{4p}-\frac{(p-1)^2}{4p}  \nonumber\\
&\hspace{2em} = \frac{n^2}{2}\left(\lambda_J^2+\frac{p}{2}r_J^2\right)+\frac{n}{2}\left(r_J(1-r)p+2\ell+r_J((r-2)p+s)\right)+D+h_{2-r,s}-h_{r,s}\nonumber\\
& \hspace{2em} = \frac{n^2}{2}\left(\lambda_J^2+\frac{p}{2}r_J^2\right)+\frac{n}{2}\left(r_J(s-p)+2\ell\right)+D-(r-1)(p-s).
\end{align}
Again, $\cW_{r,s}^{(\ell)}$ is a grading-restricted ($\theta$-twisted) $A$-module when $\lambda_J^2+\frac{p}{2}r_J^2>0$. When $\lambda_J^2+\frac{p}{2}r_J^2=0$, the lowest conformal weight of the $n$th summand of $\cW_{r,s}^{(\ell)}$ is $n(r_J\frac{p-s}{2}+\ell)+D$ for $n\geq 0$ and $\vert n\vert(r_J\frac{p-s}{2}-\ell)+D-(r-1)(p-s)$ for $n<0$. Thus there is a lower bound on the conformal weight of $\cW_{r,s}^{(\ell)}$ if and only if
\begin{equation*}
-r_J\frac{p-s}{2}\leq \ell\leq r_J\frac{p-s}{2},
\end{equation*}
and $\cW_{r,s}^{(\ell)}$ has finite-dimensional conformal weight spaces if and only if both of the above inequalities are strict. Moreover, if $\ell=-r_J\frac{p-s}{2}$, then $D$ is the lowest conformal weight of $\cW_{r,s}^{(-r_J(p-s)/2)}$ and there is no upper bound on the $\frac{1}{\lambda_J}h(0)$-eigenvalues on $(\cW_{r,s}^{(-r_J(p-s)/2)})_{[D]}$, so $\cW_{r,s}^{(-r_J(p-s)/2)}$ is not highest weight by Proposition \ref{prop:grading_of_hw_mods}. It remains to consider $\ell=r_J\frac{p-s}{2}$ for $1\leq s\leq p-1$: the $h(0)$-eigenspace of $\cW_{r,s}^{(r_J(p-s)/2)}$ with eigenvalue $\frac{r_J((2-r)p-s)}{2\lambda_J}+n\lambda_J$ has lowest conformal weight $D-(r-1)(p-s)$ if $n<0$ and $D+nr_J(p-s)$ if $n\geq 0$. Thus $\cW_{r,s}^{(r_J(p-s)/2)}$ has lower-bounded conformal weights, and for any conformal weight space, there is an upper bound on $\frac{1}{\lambda_J}h(0)$-eigenvalues. This shows that $\cW_{r,s}^{(r_J(p-s)/2)}$ is an object of $\cO_A^{\mathrm{h.w.}}$, and because it is simple, it is generated by any highest-weight vector in the maximal $\frac{1}{\lambda_J}h(0)$-eigenspace of the lowest conformal weight space.

Now for $\lambda_J^2+\frac{p}{2}r_J^2>0$, we have shown that every simple module in $\cO_A$ is grading restricted; thus every module in $\cO_A$ is grading restricted because every object of $\cO_A$ has finite length. We have also proved part (1) of the $\lambda_J^2+\frac{p}{2}r_J^2=0$ case. Part (2) follows after recalling from Remark \ref{rem:hw_in_OA} that every simple highest-weight ($\theta$-twisted) $A$-module is an object of $\cO_A$ and thus isomorphic to one of the highest-weight modules we have identified. For part (3), we just need to note that every simple grading-restricted ($\theta$-twisted) $A$-module is an object of $\cO_A$ (recall Remark \ref{rem:simple_grad-rest_are_hw}).
\end{proof}

\begin{rem}\label{rem:low-bound_conf_wts}
A calculation using \eqref{eqn:h(lambda_ell_n)}, \eqref{eqn:h(rsl)_n_positive}, and \eqref{eqn:h(rsl)_n_negative} shows that when $\lambda_J^2+\frac{p}{2}r_J^2=0$, all non-grading-restricted lower-bounded simple ($\theta$-twisted) $A$-modules have the same lowest conformal weight $-\frac{1}{4p}(p-1)^2$. These are the modules $\cW_{r,s}^{(\pm r_J(p-s)/2)}$ for $r\in\ZZ$, $1\leq s\leq p$, and the modules $\cE^{(0)}_\lambda$ for $\lambda\in\CC\setminus L^\circ$. 
\end{rem}

Next, we describe the structure of the BGG category $\cO_A^{\mathrm{h.w.}}$. For $\lambda_J^2+\frac{p}{2}r_J^2>0$, we just note that Theorem \ref{thm:grading-restricted} shows that $\cO_A$ is a subcategory of $\cO_A^{\mathrm{h.w.}}$, and thus every highest-weight ($\theta$-twisted) $A$-module is an object of $\cO_A^{\mathrm{h.w.}}$ (recall Remark \ref{rem:hw_in_OA}). The inclusion of $\cO_A$ into $\cO_A^{\mathrm{h.w.}}$ can be strict. For example, if $r_J=0$, then $A=V_{\ZZ\lambda_J}\otimes\cM(p)$ where $V_{\ZZ\lambda_J}$ is a lattice vertex operator (super)algebra. Then the $A$-module $V_{\ZZ\lambda_J}\otimes\cW(p)$, where $\cW(p)\cong\bigoplus_{n\in\ZZ} \cM_{2n+1,1}$ is the triplet vertex operator algebra, is grading restricted and thus an object of $\cO_A^{\mathrm{h.w.}}$, but not finitely generated as an $A$-module and thus not an object of $\cO_A$. For $\lambda_J^2+\frac{p}{2}r_J^2=0$, we begin with two lemmas:
\begin{lem}\label{lem:hw_unique}
Assume $\lambda_J^2+\frac{p}{2}r_J^2=0$. If two simple highest-weight ($\theta$-twisted) $A$-modules have the same highest weight, then they are isomorphic.
\end{lem}
\begin{proof}
Suppose that two simple highest-weight modules $\cW_{r,s}^{(\ell)}$ and $\cW_{r',s'}^{(\ell')}$ for $1\leq r,r'\leq r_J$, $1\leq s,s'\leq p-1$, $-r_J\frac{p-s}{2}<\ell\leq r_J\frac{p-s}{2}$, and $-r_J\frac{p-s'}{2}<\ell'\leq r_J\frac{p-s'}{2}$ have the same highest weight $(h,\lambda)$. By \eqref{eqn:h(rsl)_n_positive} and \eqref{eqn:h(rsl)_n_negative} any highest-weight vector in $\cW_{r,s}^{(\ell)}$ appears in either the $n=0$ or the $n=-1$ direct summand of \eqref{eqn:Wrsl_weight_decomp}, and similarly for $\cW_{r',s'}^{(\ell')}$. If the highest-weight vector occurs in the $n=0$ summand for both $\cW_{r,s}^{(\ell)}$ and $\cW_{r',s'}^{(\ell')}$, then lowest conformal weights satisfy $h=\frac{1}{2}\lambda^2+h_{r,s}=\frac{1}{2}\lambda^2+h_{r',s'}$, so that $h_{r,s}=h_{r',s'}$. It follows that $r=r'$ and $s=s'$, since $r,r'\geq 1$ and $1\leq s,s'\leq p-1$. Similarly, if the highest-weight vectors of both modules occur in the $n=-1$ summand, then $h_{2-(r-r_J),s}=h_{2-(r'-r_J),s'}$, and it again follows that $r=r'$ and $s=s'$. In these cases where $r=r'$, we then get
\begin{equation*}
\lambda =\frac{r_J(1-r)p+2\ell}{2\lambda_J}-[\lambda_J] = \frac{r_J(1-r)p+2\ell'}{\lambda_J}-[\lambda_J],
\end{equation*}
where the terms in brackets occur only if both highest-weight vectors appear in the $n=-1$ summands of \eqref{eqn:Wrsl_weight_decomp}. Thus $\ell=\ell'$ as well.

Now suppose that the highest-weight vector of $\cW_{r,s}^{(\ell)}$ occurs in the $n=0$ summand of \eqref{eqn:Wrsl_weight_decomp} and the highest-weight vector of $\cW_{r',s'}^{(\ell')}$ occurs in the $n=-1$ summand. Then
\begin{equation*}
\lambda=\frac{r_J(1-r)p+2\ell}{2\lambda_J}=\frac{r_J(1-r')p+2\ell'}{2\lambda_J}-\lambda_J,
\end{equation*}
or
\begin{align*}
\ell'-\ell = \frac{r_J p}{2}(r'-r)+\lambda_J^2 =\frac{r_J p}{2}(r'-r) -\frac{r_J^2 p}{2}=-\frac{r_J p}{2}(r_J-r'+r)
\end{align*}
Since
\begin{equation*}
\vert \ell'-\ell\vert < r_J\frac{p-s'}{2}+r_J\frac{p-s}{2} = r_J \left(p-\frac{s+s'}{2}\right)\leq r_J(p-1),
\end{equation*}
and since $1\leq r,r'\leq r_J$, the only possibility is $r_J-r'+r=1$, which forces $r=1$ and $r'=r_J$. But then lowest conformal weights would satisfy $h=\frac{1}{2}\lambda^2+h_{1,s} =\frac{1}{2}\lambda^2 + h_{2-(r_J-r_J),s'}$, or $h_{1,s}=h_{2,s'}$, which is impossible since $1\leq s,s'\leq p-1$.
\end{proof}

\begin{lem}\label{lem:simple_proj_in_hw_category}
Assume $\lambda_J^2+\frac{p}{2}r_J^2=0$, and let $P: X\twoheadrightarrow\cW_{r,s}^{(\ell)}$ be a surjection where $X$ is either a highest-weight ($\theta$-twisted) $A$-module or an object of $\cO_A^{\mathrm{h.w.}}\cap\cO_A$, and $\cW_{r,s}^{(\ell)}$ is a simple highest-weight ($\theta$-twisted) $A$-module. Then there is an injection $I: \cW_{r,s}^{(\ell)}\hookrightarrow X$ such that $P\circ I=\Id_{\cW_{r,s}^{(\ell)}}$.
\end{lem}
\begin{proof}
 If $X$ is highest weight, then $X$ is singly generated and thus an object of $\cO_A$. So in either case for $X$, the surjection $P$ is a morphism in $\cO_A$ and thus there is a map $q: \cQ_{r,s}^{(\ell)}\rightarrow X$ such that the diagram
\begin{equation*}
\xymatrixcolsep{3pc}
\xymatrix{
& \cQ_{r,s}^{(\ell)} \ar[ld]_q \ar[d]^(.45){p^{(\ell)}_{r,s}} \\
X \ar[r]_(.45){p} & \cW_{r,s}^{(\ell)}
}
\end{equation*}
commutes, where $p^{(\ell)}_{r,s}: \cQ_{r,s}^{(\ell)}\twoheadrightarrow\cW_{r,s}^{(\ell)}$ is the natural surjection. We need to determine $\ker q$.

First note that $\im q\subseteq X$ is lower bounded and there is an upper bound on the $\frac{1}{\lambda_J}h(0)$-eigenvalues of any minimal conformal weight space of $\im q$ (by Proposition \ref{prop:grading_of_hw_mods} in case $X$ is highest weight). Next, note that  $1\leq s\leq p-1$ and $-r_J\frac{p-s}{2}<\ell\leq r_J\frac{p-s}{2}$ since $\cW_{r,s}^{(\ell)}$ is highest weight. If $\vert\ell\vert <r_J\frac{p-s}{2}$, then neither of $\cW_{r\pm 1,p-s}^{(\ell\pm r_J p/2)}$ is lower bounded, since 
\begin{equation*}
\ell+\frac{r_Jp}{2}>\frac{r_Js}{2}=r_J\frac{p-(p-s)}{2},\qquad \ell-\frac{r_Jp}{2}\leq-\frac{r_J s}{2}=-r_J\frac{p-(p-s)}{2}.
\end{equation*}
Thus from the Loewy diagram \eqref{eqn:Q_rsl_Loewy_diag} of $\cQ_{r,s}^{(\ell)}$, we get $\im q\cong\cW_{r,s}^{(\ell)}$ or equivalently $\ker q=\ker p_{r,s}^{(\ell)}$ when $\vert\ell\vert<r_J\frac{p-s}{2}$. In case $\ell=r_J\frac{p-s}{2}$, the composition factor $\cW_{r+1,p-s}^{(r_J(p-s/2))}$ of $\cQ_{r,s}^{(r_J(p-s)/2)}$ is not lower bounded, but $\cW_{r-1,p-s}^{(-r_Js/2)}$ is. We want to rule out the possibility that $\cW_{r-1,p-s}^{(-r_Js/2)}$ is a submodule of $\im q$. This is clear if $X$ is an object of $\cO_A^{\mathrm{h.w.}}$, since there is no upper bound on the $\frac{1}{\lambda_J}h(0)$-eigenvalues of the conformal weight spaces of $\cW_{r-1,p-s}^{(-r_Js/2)}$. If $X$ is highest weight, then we only know there is an upper bound for the $\frac{1}{\lambda_J}h(0)$-eigenvalues of the lowest conformal weight space of $X$. Thus to show that $\cW_{r-1,p-s}^{(-r_Js/2)}$ is not a composition factor of $X$ when $X$ is highest weight, we need to show that $X$ and $\cW_{r-1,p-s}^{(-r_Js/2)}$ have the same lowest conformal weight. Indeed, since $P$ maps the generating highest-weight vector of $X$ to a highest-weight vector of $\cW_{r,s}^{(r_J(p-s)/2)}$, $X$ has the same lowest conformal weight $-\frac{1}{4p}(p-1)^2$ as $\cW_{r,s}^{(r_J(p-s)/2)}$, which is the same as the lowest conformal weight of $\cW_{r-1,p-s}^{(-r_Js/2)}$ (recall Remark \ref{rem:low-bound_conf_wts}). We have now shown that $\ker q=\ker p_{r,s}^{(\ell)}$ in the case $\ell=r_J\frac{p-s}{2}$ also.

We now have a well-defined injection $I:\cW_{r,s}^{(\ell)}\hookrightarrow X$ such that $I\circ p_{r,s}^{(\ell)}=q$. Then $P\circ I=\Id_{\cW_{r,s}^{(\ell)}}$ since
\begin{equation*}
P\circ I\circ p_{r,s}^{(\ell)} =P\circ q= p_{r,s}^{(\ell)}
\end{equation*}
and $p_{r,s}^{(\ell)}$ is surjective.
\end{proof}

We can now prove:
\begin{thm}\label{thm:hw_cat_ss}
Assume $\lambda_J^2+\frac{p}{2}r_J^2=0$. Then:
\begin{enumerate}
\item Every highest-weight ($\theta$-twisted) $A$-module is simple and thus an object of $\cO_A^{\mathrm{h.w.}}$.
\item The highest-weight category $\cO_A^{\mathrm{h.w.}}$ equals the semisimple subcategory of $\cO_A$ with simple objects $\cW_{r,s}^{(\ell)}$ for $1\leq r \leq r_J$, $1\leq s\leq p-1$, and $-r_J\frac{p-s}{2}<\ell\leq r_J\frac{p-s}{2}$.
\end{enumerate}
\end{thm}
\begin{proof}
To prove part (1), let $X$ be a highest-weight ($\theta$-twisted) $A$-module. Since $X$ is singly generated, it is a (finite-length) object of $\cO_A$ and thus has a simple quotient. Moreover, the generating highest-weight vector in $X$ maps to a highest-weight vector in the simple quotient, so there is a surjection $X\twoheadrightarrow\cW_{r,s}^{(\ell)}$ for some highest-weight module $\cW_{r,s}^{(\ell)}$. By Lemma \ref{lem:simple_proj_in_hw_category}, $\cW_{r,s}^{(\ell)}$ is a direct summand of $X$. Then either $X\cong\cW_{r,s}^{(\ell)}$ or $X/\cW_{r,s}^{(\ell)}$ is a highest-weight module with a lower-dimensional highest-weight space than $X$. By induction on the dimension of its highest-weight space, $X$ is a finite direct sum of simple highest-weight modules that have the same highest weight as $X$. But by Lemma \ref{lem:hw_unique}, only one $\cW_{r,s}^{(\ell)}$ has the same highest weight as $X$. So because $X$ is singly generated, we must have $X\cong\cW_{r,s}^{(\ell)}$, that is, $X$ is simple.

For part (2), first we show that every object $X$ of $\cO_A^{\mathrm{h.w.}}$ has finite length and thus is an object of $\cO_A$. Let $(h_1,\lambda_1),\ldots, (h_N,\lambda_N)$ be the highest weights of the finitely many simple highest-weight ($\theta$-twisted) $A$-modules. We prove that $X$ has finite length by induction on $d=\sum_{n=1}^N \dim X_{[h_n]}^{(\lambda_n)}$, which is finite because $X$ is strongly $\cH$-weight graded. First, if $X\neq 0$, then the definition of $\cO_A^{\mathrm{h.w.}}$ implies that $X$ has a highest-weight vector which generates a highest-weight submodule $\til{X}$. This highest-weight submodule is simple by part (1), and thus its highest weight is one of the $(h_n,\lambda_n)$. Consequently, if $d=0$, then $X=0$, and if $d>0$, then $\sum_{n=1}^N \dim  (X/\til{X})^{(\lambda_n)}_{[h_n]} < d$. Thus by induction on $d$, $X/\til{X}$ has finite length, and then $X$ has finite length as well since $\til{X}$ is simple.

Now that we have shown $\cO_A^{\mathrm{h.w.}}$ is a subcategory of $\cO_A$, Lemma \ref{lem:simple_proj_in_hw_category} shows that the simple highest-weight ($\theta$-twisted) $A$-modules are projective in $\cO_A^{\mathrm{h.w.}}$. Thus $\cO_A^{\mathrm{h.w.}}$ is semisimple since its objects have finite length.
\end{proof}

As the category of $h(0)$-semisimple grading-restricted generalized ($\theta$-twisted) $A$-modules is a subcategory of $\cO_A^{\mathrm{h.w.}}$, we also have:
\begin{cor}\label{cor:h0_ss}
If $\lambda_J^2+\frac{p}{2}r_J^2=0$, then the category of $h(0)$-semisimple grading-restricted generalized ($\theta$-twisted) $A$-modules is semisimple with finitely many simple objects.
\end{cor}

We can actually drop the assumption of $h(0)$-semisimplicity in the above corollary, although the proof is quite technical:
\begin{thm}\label{thm:grad_rest_ss}
If $\lambda_J^2+\frac{p}{2}r_J^2=0$, then the category $\cC_A$ of grading-restricted generalized local and $\theta$-twisted $A$-modules is semisimple with finitely many simple objects $\cW_{r,s}^{(\ell)}$ for $1\leq r\leq r_J$, $1\leq s\leq p-1$, and $\vert\ell\vert <r_J\frac{p-s}{2}$.
\end{thm}
\begin{proof}
As there are finitely many simple grading-restricted ($\theta$-twisted) $A$-modules, an argument similar to the second paragraph in the proof of Theorem \ref{thm:hw_cat_ss} (see also \cite[Proposition 3.15]{Hu-C2}) shows that any grading-restricted generalized local or twisted $A$-module (whether $h(0)$-semisimple or not) has finite length. Thus to show that $\cC_A$ is semisimple, it is enough to show that its simple objects are projective in $\cC_A$, and for this it is enough to show that all exact sequences
\begin{equation*}
0\longrightarrow\cW_{r', s'}^{(\ell')}\longrightarrow X\longrightarrow\cW_{r,s}^{(\ell)}\longrightarrow 0
\end{equation*}
split, where $1\leq r,r'\leq r_J$, $1\leq s,s'\leq p-1$, $\vert\ell\vert<r_J\frac{p-s}{2}$, and $\vert \ell'\vert < r_J\frac{p-s'}{2}$. In fact, such an exact sequence splits for $(r',s',\ell')\neq (r, s, \ell)$ by Corollary \ref{cor:h0_ss}, because then $X$ must be $h(0)$-semisimple (since the locally nilpotent part of $h(0)$ is an endomorphism of $X$). 

It remains to rule out non-split self-extensions. In fact, we claim that all exact sequences
\begin{equation*}
0\longrightarrow\cW_{r,s}^{(\ell)}\longrightarrow X\longrightarrow\cW_{r,s}^{(\ell)}\longrightarrow 0
\end{equation*}
split, where $1\leq r\leq r_J$, $1\leq s\leq p-1$, $\ell\in\frac{1}{2}\ZZ$, and $X$ is a (not necessarily grading-restricted) generalized local or $\theta$-twisted $A$-module. To prove this, first note that by \eqref{eqn:Wrsl_weight_decomp}, $X$ decomposes into generalized $h(0)$-eigenspaces, $X\cong\bigoplus_{n\in\ZZ} X^{[\gamma+n\lambda_J]}$, where $\gamma=\frac{r_J(1-r)p+2\ell}{2\lambda_J}$ and there is an exact $\cH\otimes\cM(p)$-module sequence
\begin{equation*}
0\longrightarrow\cF_{\gamma+n\lambda_J}\otimes\cM_{r+nr_J,s}\longrightarrow X^{[\gamma+n\lambda_J]}\longrightarrow\cF_{\gamma+n\lambda_J}\otimes\cM_{r+nr_J,s}\longrightarrow 0
\end{equation*}
for each $n\in\ZZ$. Then each Heisenberg vacuum space
\begin{equation*}
\Omega^n =\lbrace b\in X^{[\gamma+n\lambda_J]}\,\,\vert\,\,h(m)b=0\,\,\text{for all}\,\,m>0\rbrace
\end{equation*}
is an $\cM(p)$-module fitting into an exact sequence
\begin{equation*}
0\longrightarrow\cM_{r+nr_J,s}\longrightarrow \Omega^n\longrightarrow\cM_{r+nr_J,s}\longrightarrow 0.
\end{equation*}
We would like to show that this $\cM(p)$-module exact sequence splits, at least for $n\geq 0$.

To prove that $\Omega^n\cong\cM_{r+nr_J,s}\oplus\cM_{r+nr_J,s}$ when $n\geq 0$, we will show that its lowest conformal weight space $T(\Omega^n)=\Omega^n_{[h_{r+nr_J,s}]}$ is a semisimple module for the Zhu algebra $A(\cM(p))$. From the description of $A(\cM(p))$ in \cite{Ad}, it is sufficient to show that both $L(0)$ and the weight-preserving mode of the second strong generator $H\in\cM(p)_{(2p-1)}$ of $\cM(p)$ act by scalars on $T(\Omega^n)$. For $L(0)$, we consider the Virasoro submodule $U(\cV ir)\cdot T(\Omega^n)$ generated by $T(\Omega^n)$. As a $\cV ir$-module, $\cM_{r+nr_J,s}\cong
\bigoplus_{m=0}^\infty \cL_{r+nr_J+2m,s}$
 where $\cL_{r,s}$ for $r\in\ZZ_{\geq 1}$, $1\leq s\leq p$ is the irreducible highest-weight $\cV ir$-module generated by a singular vector of conformal weight $h_{r,s}$. Thus we have an exact sequence
\begin{equation*}
0\longrightarrow \cL_{r+nr_J,s}\longrightarrow U(\cV ir)\cdot T(\Omega^n)\longrightarrow \cV_{r+nr_J,s}/\cJ \longrightarrow 0
\end{equation*}
where $\cV_{r+nr_J,s}$ is the Virasoro Verma module of lowest conformal weight $h_{r+nr_J,s}$ and $\cJ$ is some submodule such that all simple subquotients of $\cV_{r+nr_J,s}/\cJ$ have the form $\cL_{r+nr_J+2m,s}$ for some $m\in\ZZ_{\geq 0}$. The structure of $\cV_{r+nr_J,s}$ (which follows from the Verma module embedding diagrams in, for example, \cite[Section 5.3]{IK}; see also \cite[Section 2.1]{MY}) shows that when $1\leq s\leq p-1$, all quotients of $\cV_{r+nr_J,s}$ which contain $\cL_{r+nr_J+2m,s}$, $m>0$, as a composition factor also contain $\cL_{r+nr_J+1,p-s}$. Thus $\cV_{r+nr_J,s}/\cJ$ must be the simple quotient $\cL_{r+nr_J,s}$, so $U(\cV ir)\cdot T(\Omega^n)$ is a self-extension of $\cL_{r+nr_J,s}$. By \cite[Section 5.4]{GK}, $\cL_{r,s}$ admits non-split self-extensions only when $s=p$ and $r\geq 2$, so $U(\cV ir)\cdot T(\Omega^n)\cong\cL_{r+nr_J,s}\oplus \cL_{r+nr_J,s}$, that is, $L(0)$ acts semisimply on $T(\Omega^n)$.

We now consider the action of $H_{2p-2}$ on $T(\Omega^n)$. For the case $r+nr_J>1$, the description of the Zhu algebra $A(\cM(p))$ in \cite{Ad} shows that the action of $H_{2p-2}$ on $T(\Omega^n)$ squares to a polynomial in $L(0)$ whose only roots are $h_{1,s}$, $1\leq s\leq p$. Thus when $r+nr_J>1$, $H_{2p-2}$ on $T(\Omega^n)$ squares to a non-zero scalar, and thus $H_{2p-2}$ is also semisimple on $T(\Omega^n)$. For the case $r+nr_J=1$, recall from \cite{Ad} that $H$ is a Virasoro singular vector which generates a $\cV ir$-submodule of $\cM(p)$ isomorphic to $\cL_{3,1}$. Thus $H_{2p-2}\vert_{U(\cV ir)\cdot T(\Omega^n)}$ is a coefficient of an intertwining operator of type $\binom{\Omega^n}{\cL_{3,1}\,\cL_{1,s}\oplus\cL_{1,s}}$. The image of this intertwining operator is $C_1$-cofinite (see \cite[Key Theorem]{Miy}) and thus is a homomorphic image of $\cL_{3,1}\tens(\cL_{1,s}\oplus\cL_{1,s})\cong\cL_{3,s}\oplus\cL_{3,s}$ (see \cite[Theorem 4.6]{MY}). In particular, for $v\in U(\cV ir)\cdot T(\Omega^n)$ homogeneous, $H_{2p-2} v=0$ unless $\mathrm{wt}\,v\geq h_{3,s}> h_{1,s}$. Thus $H_{2p-2}\vert_{T(\Omega^n)}=0$ when $r+nr_J=1$. We have now shown that $T(\Omega^n)$ is a semisimple $A(\cM(p))$-module and thus that $\Omega^0\cong\cM_{r+nr_J,s}\oplus\cM_{r+nr_J,s}$, at least when $n\geq 0$.

Now since the $\cH\otimes\cM(p)$-module $X^{[\gamma+n\lambda_J]}$ is generated as an $\cH$-module by its vacuum space $\Omega^n$, we have now shown that $X^{[\gamma+n\lambda_J]}$ for $n\geq 0$ is, as an $\cM(p)$-module, a sum of copies of $\cM_{r+nr_J,s}$. It follows (see for example \cite[Theorem 2.7]{McR-cosets}) that $X^{[\gamma+n\lambda_J]}\cong \cG_{\gamma+n\lambda_J}\otimes\cM_{r+nr_J,s}$ for some grading-restricted generalized $\cH$-module $\cG_{\gamma+n\lambda_J}$. This module has a two-dimensional lowest conformal weight space, and $(h(0)-\gamma-n\lambda_J)^2=0$ on $\cG_{\gamma+n\lambda_J}$; it follows that we have an exact sequence
\begin{equation*}
0\longrightarrow \cF^\cH_{\gamma+n\lambda_J}\longrightarrow\cG_{\gamma+n\lambda_J}\longrightarrow\cF^\cH_{\gamma+n\lambda_J}\longrightarrow 0.
\end{equation*}
Note that $X^{[\gamma+n\lambda_J]}$ generates $X$ as a generalized $A$-module since it intersects both composition factors of $X$ non-trivially. Thus if $h(0)$ acts semisimply on $\cG_{\gamma+n\lambda_J}$, then it acts semisimply on all of $X$ and we may conclude $X\cong\cW_{r,s}^{(\ell)}\oplus\cW_{r,s}^{(\ell)}$ (since the structure \eqref{eqn:Q_rsl_Loewy_diag} of the projective cover $\cQ_{r,s}^{(\ell)}$ shows that $\cW_{r,s}^{(\ell)}$ does not have non-split self-extensions in $\cO_A$). So taking $n=0,1$, it remains to rule out the possibility that $h(0)$ acts non-semisimply on both $\cG_\gamma$ and $\cG_{\gamma+\lambda_J}$.

By \eqref{eqn:Y_compat_with_h(0)}, the restriction $Y_X\vert_{J\otimes X^{[\gamma]}}$ of the $A$-module vertex operator $Y_X: A\otimes X\rightarrow X((x))$ is an $\cH\otimes\cM(p)$-module intertwining operator of type $\binom{X^{[\gamma+\lambda_J]}}{J\,\,X^{[\gamma]}}$. More generally,
\begin{equation*}
\im Y_X\vert_{J^n\otimes X^{[\gamma]}}\subseteq X^{[\gamma+n\lambda_J]},
\end{equation*}
for any $n\in\ZZ$, so because $X^{[\gamma]}$ generates $X$ as an $A$-module, $\im Y_X\vert_{J\otimes X^{[\gamma]}}=X^{[\gamma+\lambda_J]}$. Then by \cite[Theorem 2.10]{ADL}, $Y_X\vert_{J\otimes X^{[\gamma]}}=\cY_1\otimes\cY_2$ where $\cY_1$ is a surjective $\cH$-module intertwining operator of type $\binom{\cG_{\gamma+\lambda_J}}{\cF^\cH_{\lambda_J}\,\cG_{\gamma}}$ and $\cY_2$ is the unique (up to scale) non-zero $\cM(p)$-module intertwining operator of type $\binom{\cM_{r+r_J,s}}{\cM_{r_J+1,1}\,\cM_{r,s}}$. To complete the proof of the theorem, it is enough to show that if $h(0)$ acts non-semisimply on $\cG_\gamma$ and $\cG_{\gamma+\lambda_J}$, then $\cY_1$ must be a logarithmic intertwining operator, because then $Y_X$ would involve powers of $\log x$, which it cannot since it is the vertex operator for a generalized ($\theta$-twisted) $A$-module.

To prove that $\cY_1$ is logarithmic, let $\cY_\tens:\cF^\cH_{\lambda_J}\otimes\cG_\gamma\rightarrow(\cF^\cH_{\lambda_J}\tens\cG_\gamma)[\log x]\lbrace x\rbrace$ be the tensor product intertwining operator in the tensor category of $C_1$-cofinite grading-restricted generalized $\cH$-modules. By the definition of vertex algebraic tensor products, there is a unique surjection 
\begin{equation*}
f: \cF^\cH_{\lambda_J}\tens\cG_{\gamma}\longrightarrow\cG_{\gamma+\lambda_J}
\end{equation*}
such that $f\circ\cY_\tens=\cY_1$. Then $f$ is an isomorphism since tensoring with the simple current $\cF^\cH_{\lambda_J}$ preserves lengths, and thus the domain of $f$ is an $\cH$-module of length $2$. Consequently, $\cY_1$ involves powers of $\log x$ if and only if $\cY_\tens$ does, and indeed $\cY_\tens$ does since there are logarithmic $\cH$-module intertwining operators of type $\binom{W}{\cF^\cH_\lambda\,\cG_\gamma}$ for some $\cH$-module $W$ (see \cite[Theorem 5.5]{M}).
\end{proof}

\begin{rem}\label{rem:hw_ss}
More generally, the proof of the above theorem shows that if $\lambda_J^2+\frac{p}{2}r_J^2=0$, then the category of finite-length generalized $A$-modules whose composition factors are simple highest-weight $A$-modules is semisimple (with finitely many simple objects up to isomorphism).
\end{rem}

 \subsection{The category of \texorpdfstring{$C_1$}{C1}-cofinite \texorpdfstring{$A$}{A}-modules}\label{subsec:C1-cofinite}
 
 In this subsection, we classify $C_1$-cofinite (local) modules for the vertex operator (super)algebra $A$. Recall that a grading-restricted generalized $A$-module $X$ is $C_1$-cofinite if $\dim X/C_1(X)<\infty$ where
 \begin{equation*}
C_1(X)=\mathrm{span}\left\lbrace a_{-1} b\,\,\vert\,\,b\in X, a\in A\,\,\text{homogeneous such that}\,\,\mathrm{wt}\,a>0\right\rbrace. 
 \end{equation*}
 Let $\cC^1_A$ be the category of $C_1$-cofinite grading-restricted generalized $A$-modules. Due to results of Huang \cite{Hu-diff-eqs} and Miyamoto \cite{Miy, Miy2}, one expects $\cC^1_A$ to admit the vertex algebraic tensor category structure of \cite{HLZ1}-\cite{HLZ8}, and here we show that this is the case when $A$ is a simple current extension of $\cH\otimes\cM(p)$. However, when $\lambda_J^2+\frac{p}{2}r_J^2=0$, we will see that $\cC^1_A$ is a very small subcategory of the larger vertex algebraic tensor category $\cO_A$. We begin with two lemmas on properties of simple $\cH\otimes\cM(p)$-modules:
 \begin{lem}\label{lem:H_times_M(p)_C1_cofinite}
 If $W$ is a simple grading-restricted $\cH\otimes\cM(p)$-module,  then
 \begin{equation*}
 W = C_1(W)+\bigoplus_{i=0}^I \CC L(-1)^i w,
 \end{equation*}
 for some $I\in\ZZ_{\geq 0}$, where $w\in W$ is a non-zero lowest-conformal-weight vector.
 \end{lem}
 
 \begin{proof}
 We take $W=\cF^\cH_\gamma\otimes\cM_\lambda$, where $\gamma,\lambda\in\CC$ and $\cM_\lambda$ is the $\cM(p)$-socle of $\cF_\lambda$. Since
 \begin{equation*}
 \cF^\cH_\gamma=\mathrm{span}\lbrace h(-n_1)\cdots h(-n_k) v^\cH_\gamma\,\,\vert\,\,n_1,\ldots,n_k\in\ZZ_{\geq 1}\rbrace = C_1(\cF^\cH_\gamma)+\CC v_\gamma^\cH
 \end{equation*}
 for some lowest-conformal-weight vector $v^\cH_\gamma$, we have
\begin{equation*}
W = C_1(W) + v^\cH_\gamma\otimes T
\end{equation*} 
 where $T\subseteq\cM_\lambda$ is a graded subspace such that $\cM_\lambda= C_1(\cM_\lambda)+T$. If $\lambda=\alpha_{r,s}$ for some $r\in\ZZ$, $1\leq s\leq p$, then equation (7) in the proof of \cite[Theorem 13]{CMR} shows that we may take $T=\bigoplus_{i=0}^I \CC L(-1)_{\cM(p)}^i v_\lambda$ for some $I$, where $v_\lambda\in\cM_\lambda$ is a lowest-conformal-weight vector and $L(-1)_{\cM(p)}$ is the $L(-1)$-mode for the conformal vector of $\cM(p)$. If $\lambda\in\CC\setminus L^\circ$, in which case $\cM_\lambda=\cF_\lambda$, we can also take $T$ of this form. Indeed, $\cF_\lambda$ is then simple as a Virasoro module, so $\cF_\lambda=C_1(\cF_\lambda)+\bigoplus_{i=0}^\infty \CC L(-1)^i_{\cM(p)} v_\lambda$. But because $\cF_\lambda$ is $C_1$-cofinite as an $\cM(p)$-module \cite[Theorem 13]{CMR}, $L(-1)^i v_\lambda\in C_1(\cF_\lambda)$ for $i$ greater than some $I$.
 
 We now have $W=C_1(W)+\bigoplus_{i=0}^I \CC L(-1)^i_{\cM(p)} w$ for some $I$, where $w=v^\cH_\gamma\otimes v_\lambda$. Then
 \begin{align*}
 L(-1)_{\cM(p)}^i w & =(L(-1)-L(-1)_\cH)^i w\nonumber\\
 & =\sum_{i'=0}^i \binom{i}{i'} L(-1)_\cH^{i-i'} L(-1)^{i'}w \equiv L(-1)^i w\mod C_1(W),
 \end{align*}
 for any $i\in\ZZ_{\geq 0}$, so $W=C_1(W)+\bigoplus_{i=0}^I \CC L(-1)^i w$ as desired.
 \end{proof}

 \begin{lem}\label{lem:Mp_intw_op_low_wt_vectors}
 Let $\cY$ be a non-zero $\cH\otimes\cM(p)$-module intertwining operator of type $\binom{W_3}{W_1\,W_2}$ where $W_1$, $W_2$, and $W_3$ are simple. If $w_1\in W_1$ and $w_2\in W_2$ are vectors of lowest conformal weight, then the coefficient of the lowest power of $x$ in $\cY(w_1,x)w_2$ is a non-zero vector of lowest conformal weight in $W_3$.
 \end{lem}
 \begin{proof}
 By Lemma \ref{lem:H_times_M(p)_C1_cofinite}, $W_1=C_1(W_1)+T$ where $T=\bigoplus_{i=0}^I \CC L(-1)^i w_1$ for some finite $I$. Now let $\mathfrak{g}_{\leq 0}$ be the Lie (super)algebra spanned by vertex operator modes $v_n$ acting on $W_3$ where $v\in\cH\otimes\cM(p)$ is homogeneous and $n-\mathrm{wt}\,v+1\leq 0$. The proof of \cite[Lemma 2.11]{CMY-completions} shows that the image of $\cY$ is generated as a $\mathfrak{g}_{\leq 0}$-module by the coefficients of powers of $x$ in $\cY(t,x)w_2$ for $t\in T$. By the $L(-1)$-derivative property for $\cY$, this means the coefficients of $\cY(w_1,x)w_2$ generate the image of $\cY$ as a $\mathfrak{g}_{\leq 0}$-module. Thus since vertex operator modes in $\mathfrak{g}_{\leq 0}$ weakly raise conformal weight, and since a non-zero intertwining operator to a simple $\cH\otimes\cM(p)$-module is surjective, the coefficient of the lowest power of $x$ in $\cY(w_1,x)w_2$ is a non-zero lowest-conformal-weight vector in $W_3$.
 \end{proof}
 
 In proving the following theorem, the module $W_1$ in the preceding lemma will be the simple current $\cH\otimes\cM(p)$-modules $J^{\pm 1}$. 
 \begin{thm}\label{thm:simple_C1_cofinite}
 If $\lambda_J^2+\frac{p}{2}r_J^2>0$, then every simple module in $\Oloc_A$ is $C_1$-cofinite. If $\lambda_J^2+\frac{p}{2}r_J^2=0$, then a simple grading-restricted $A$-module is $C_1$-cofinite if and only if it is isomorphic to one of the $\cW_{r,1}^{(0)}$ for $1\leq r\leq r_J$.
 \end{thm}
 \begin{proof}
Let $X\cong\cF(W)$ be a simple module in $\Oloc_A$, where $W$ is a simple $\cH\otimes\cM(p)$-module. Thus $X\cong\bigoplus_{n\in\ZZ} J^n\tens W$ as an $\cH\otimes\cM(p)$-module. By Lemma \ref{lem:H_times_M(p)_C1_cofinite}, there are $I_n\in\ZZ_{\geq 0}$ and lowest-conformal-weight vectors $w^{(n)}\in J^n\tens W$ such that
\begin{equation*}
X = C_1(X)+\bigoplus_{n\in\ZZ}\bigoplus_{i=0}^{I_n} \CC L(-1)^i w^{(n)}.
\end{equation*}
 We claim that if $w^{(n)}\in C_1(X)$, then $J^n\tens W\subseteq C_1(X)$. Indeed, if $w^{(n)}$ is equal to (a finite sum of) $a_{-1} b$ for $a\in A$ and $b\in X$, then  the $L(-1)$-derivative property
 \begin{equation*}
 L(-1)Y_X(a,x)=\frac{d}{dx}Y_X(a,x)+Y_X(a,x)L(-1),
 \end{equation*}
 implies that
 \begin{equation*}
 L(-1)^i a_{-1} b=\sum_{i'=0}^i \frac{i!}{(i-i')!} a_{-i'-1} L(-1)^{i-i'}b\in C_1(X).
 \end{equation*}
 Thus to show $X$ is $C_1$-cofinite, it is sufficient to show that $w^{(n)}\in C_1(X)$ for all but finitely many $n\in\ZZ$.

To determine when $w^{(n)}\in C_1(X)$, fix non-zero vectors $j\in J$, $j'\in J^{-1}$, and $w\in W$ of minimal conformal weight. Then from Lemma \ref{lem:Mp_intw_op_low_wt_vectors}, there exist $m_n, m_n'\in\ZZ$ such that for all $n\in\ZZ_{\geq 1}$,
 \begin{equation*}
w^{(n)} = j_{m_n}j_{m_{n-1}}\cdots j_{m_1} w,\qquad\quad w^{(-n)} = j'_{m_n'} j'_{m_{n-1}'}\cdots j'_{m_1'} w
 \end{equation*}
 are non-zero vectors in the lowest conformal weight spaces of $J^n\tens W$ and $J^{-n}\tens W$, respectively. Specifically, by conformal weight considerations,
 \begin{equation*}
 h_J-m_n-1 = h_{J^n\tens W}-h_{J^{n-1}\tens W},\qquad h_{J^{-1}} -m_n'-1=h_{J^{-n}\tens W}-h_{J^{-n+1}\tens W}, 
 \end{equation*}
 where $h_M$ denotes the lowest conformal weight of a simple module $M$.
 
 In the proof of Theorem \ref{thm:grading-restricted}, we saw that $h_{J^n\tens W}$ is quadratic in $n$ with leading coefficient $\frac{1}{2}(\lambda_J^2+\frac{p}{2}r_J^2)$. Thus if $\lambda_J^2+\frac{p}{2}r_J^2>0$, then
 \begin{equation*}
 m_n= h_{J^{n-1}\tens W}-h_{J^n\tens W}+h_J-1
 \end{equation*}
 is linear in $n$ with negative slope, and thus $m_n>-1$ for only finitely many $n\in\ZZ_{\geq 1}$. Similarly, $m_n'>-1$ for only finitely many $n\in\ZZ_{\geq 1}$. Thus $w^{(n)}\in C_1(X)$ for all but finitely many $n$, which then implies that $X\cong\cF(W)$ is $C_1$-cofinite for any simple $\cH\otimes\cM(p)$-module $W$ such that $\cF(W)$ is local. If $\lambda_J^2+\frac{p}{2}r_J^2=0$  and $\cF(W)=\cW_{r,1}^{(0)}$, that is, $W=\cF^\cH_{r_J(1-r)p/(2\lambda_J)}\otimes\cM_{r,1}$, then \eqref{eqn:h(rsl)_n_positive} and \eqref{eqn:lowest_wt_of_Jn} show that for $n\geq 1$,
 \begin{equation*}
 h_{J^n\tens W}-h_{J^{n-1}\tens W} = r_J\frac{p-1}{2} =h_J, 
 \end{equation*}
 so that $m_n=-1$ for all $n\geq 1$. Similarly, \eqref{eqn:h(rsl)_n_negative} and \eqref{eqn:lowest_wt_of_Jn} show that $m_n'=-1$ for $n\geq 2$. Again, $w^{(n)}\in C_1(X)$ for all but at most two $n$, implying that $\cW_{r,1}^{(0)}$ is $C_1$-cofinite.

It remains to show that if $\lambda_J^2+\frac{p}{2}r_J^2=0$, then $\cW_{r,s}^{(\ell)}$ is not $C_1$-cofinite when $s>1$ or $\ell\neq 0$. We shall prove this using the tensor product formulas of Theorem \ref{thm:general_fusion_rules}(1) and the fact that if $\cY$ is a surjective intertwining operator of type $\binom{W_3}{W_1\,W_2}$ such that $W_1$, $W_2$ are $C_1$-cofinite grading-restricted generalized $A$-modules and $W_3$ is any generalized $A$-module, then $W_3$ is $C_1$-cofinite and grading restricted (see \cite[Key Theorem]{Miy} and \cite[Corollary 2.12]{CMY-completions}). In particular, if $W_1\tens W_2$ is the tensor product of $W_1$ and $W_2$ in $\cO_A$, then any homomorphic image of $W_1\tens W_2$ is grading restricted.

Now suppose $\cW_{r,s}^{(\ell)}$ is a $C_1$-cofinite grading-restricted $A$-module (so $1\leq s\leq p-1$). By tensoring with the $C_1$-cofinite module $\cW_{r',1}^{(0)}$ such that $r+r'-1\equiv 1\,(\mathrm{mod}\,r_J)$, we see that $\cW_{1,s}^{(\ell)}$ is also $C_1$-cofinite. Then so is $\cW_{1,s}^{(\ell)}\tens\cW_{1,s}^{(\ell)}$ and its quotient $\cW_{1,1}^{(2\ell)}$. It then follows that $(\cW_{1,1}^{(2\ell)})^{\tens n} \cong \cW_{1,1}^{(2n\ell)}$ is $C_1$-cofinite for all $n\geq 1$. Such modules are grading restricted only when $\vert 2n\ell\vert < \frac{p-1}{2}$, so $\ell=0$. Next, all quotients of direct summands in
\begin{equation*}
\cW_{1,s}^{(0)}\tens\cW_{1,s}^{(0)}\cong \bigoplus_{n=1}^{\min(s, p-s)} \cW^{(0)}_{1,2n-1}\oplus\bigoplus_{n=p-s+1}^{\lfloor (p+1)/2\rfloor} \cQ^{(0)}_{1,2n-1}.
\end{equation*}
are $C_1$-cofinite. In particular, $\cW_{1,3}^{(0)}$ would be $C_1$-cofinite were $s>1$ (actually $s\geq 3$ since $\cW_{1,s}^{(0)}$ is local and thus $s$ is odd). But then the fusion rules
\begin{equation*}
\cW_{1,3}^{(0)}\tens\cW_{1,s}^{(0)}\cong\left\lbrace\begin{array}{lll}
\cW_{1,s-2}^{(0)}\oplus\cW_{1,s}^{(0)}\oplus\cW_{1,s+2}^{(0)} & \text{if} & 3\leq s\leq p-2\\
\cW_{1,p-3}^{(0)}\oplus\cQ_{1,p-1}^{(0)} & \text{if} & s=p-1\\ 
\end{array}\right.
\end{equation*}
would imply that either $\cQ_{1,p-1}^{(0)}$ (if $p$ is even) or $\cW_{1,p}^{(0)}=\cQ_{1,p}^{(0)}$ (if $p$ is odd) is $C_1$-cofinite. But neither of these projective modules is grading restricted, so $s>1$ is impossible. In conclusion, we have shown that if $\cW_{r,s}^{(\ell)}$ is a $C_1$-cofinite grading-restricted (local) $A$-module, then $\ell=0$ and $s=1$, completing the proof of the theorem.
 \end{proof}

Now that we have characterized the simple $C_1$-cofinite $A$-modules, we characterize the entire category $\cC_A^1$:
\begin{thm}\label{thm:C_1-cofin}
If $\lambda_J^2+\frac{p}{2}r_J^2>0$, then $\cC_A^1$ is equal to the category of finite-length grading-restricted generalized $A$-modules. If $\lambda_J^2+\frac{p}{2}r_J^2=0$, then $\cC_A^1$ is the semisimple subcategory of $\Oloc_A$ with simple objects $\cW_{r,1}^{(0)}$, $1\leq r\leq r_J$.
\end{thm}
\begin{proof}
If $\lambda_J^2+\frac{p}{2}r_J^2>0$, then Theorem \ref{thm:simple_C1_cofinite} shows that all simple grading-restricted $A$-modules are $C_1$-cofinite, so all finite-length grading-restricted generalized $A$-modules are $C_1$-cofinite. Conversely, we need to show that if $X$ is a $C_1$-cofinite grading-restricted generalized $A$-module, then $X$ has finite length. Since $X$ is $C_1$-cofinite, the proof of \cite[Theorem 3.3.5]{CY} shows that there is a finite filtration
\begin{equation*}
0=X_0\hookrightarrow X_1\hookrightarrow\cdots\hookrightarrow X_{n-1}\hookrightarrow X_n =X
\end{equation*}
such that each $X_i/X_{i-1}$ for $1\leq i\leq n$ is generated by a lowest conformal weight space which is a finite-dimensional irreducible module for the Lie (super)algebra of conformal-weight-preserving vertex operator modes. Thus each $X_i/X_{i-1}$ is a highest-weight $A$-module and thus a (finite-length) object of $\Oloc_A$. Consequently, $X$ also has finite length.

If $\lambda_J^2+\frac{p}{2}r_J^2=0$, then every $C_1$-cofinite grading-restricted generalized $A$-module $X$ is semisimple by Theorem \ref{thm:grad_rest_ss}. Thus every composition factor of $X$ is $C_1$-cofinite because $\cC_A^1$ is closed under quotients. It follows from Theorem \ref{thm:simple_C1_cofinite}  that every $C_1$-cofinite grading-restricted generalized $A$-module is a finite direct sum of modules $\cW_{r,1}^{(0)}$ for $1\leq r\leq r_J$.
\end{proof}

From the results of \cite[Section 4]{CJORY} and \cite[Section 3]{CY} and their proofs (see also \cite[Theorem 2.3]{McR-cosets}), $\cC_A^1$ admits the vertex algebraic tensor category structure of \cite{HLZ1}-\cite{HLZ8} provided it is closed under contragredients. This condition is evident from the characterization of $\cC_A^1$ in the preceding theorem, so we conclude:
\begin{cor}\label{cor:C1_cofin}
The category $\cC_A^1$ of $C_1$-cofinite grading-restricted generalized $A$-modules admits the vertex algebraic tensor category structure of \cite{HLZ1}-\cite{HLZ8}. Moreover, $\cC_A^1$ is rigid with duals given by contragredient modules, and ribbon in case $A$ is $\ZZ$-graded.
\end{cor}
\begin{proof}
For rigidity in the case $\lambda_J^2+\frac{p}{2}r_J^2=0$, \cite[Key Theorem]{Miy} and \cite[Corollary 2.12]{CMY-completions} show that the tensor product in $\cO_A$ of two $C_1$-cofinite modules is $C_1$-cofinite, and thus $\cC_A^1$ is a tensor subcategory of the rigid tensor category $\cO_A$. In case $\lambda_J^2+\frac{p}{2}r_J^2>0$, \eqref{eqn:Y_compat_with_h(0)} shows that the tensor product in $\cC_A^1$ of two $h(0)$-semisimple $C_1$-cofinite modules is $h(0)$-semisimple, and thus the rigid tensor category $\Oloc_A$ is a tensor subcategory of $\cC_A^1$. In particular, all simple objects of $\cC_A^1$ are rigid, so $\cC_A^1$ is rigid by \cite[Theorem 4.4.1]{CMY-singlet}.
\end{proof}

\begin{rem}
When $\lambda_J^2+\frac{p}{2}r_J^2=0$, the category of (not necessarily $C_1$-cofinite) grading-restricted generalized $A$-modules does not usually admit natural tensor category structure since the argument concluding the proof of Theorem \ref{thm:simple_C1_cofinite} shows that this category is not usually closed under tensor products in $\Oloc_A$.
\end{rem}

\section{Rigid tensor categories for some subregular \texorpdfstring{$W$}{W}-algebras}\label{subregular}

In this section, we use the general results developed in the previous section to describe the categories of finitely-generated strongly $\cH$-weight-graded modules for the vertex operator algebras $\cB_p$, $p\in\ZZ_{\geq 2}$, of \cite{CRW}, as well as for the finite cyclic orbifolds of $\cB_2$. These vertex operator algebras are examples of simple affine $W$-algebras associated to subregular nilpotent elements in $\mathfrak{sl}_n$.

\subsection{The vertex operator algebra \texorpdfstring{$\mathcal B_p$}{Bp}}

The vertex operator algebra $\mathcal{B}_p$ for integers $p\geq 2$ is a simple current extension of $\cH \otimes \cM(p)$ introduced in \cite{CRW}; see also  \cite{Cr}. The algebra $\cB_2$ is the $\beta\gamma$-ghost vertex algebra, while for $p\geq 3$, $\cB_p$ is the simple subregular $W$-algebra of $\mathfrak{sl}_{p-1}$ at level $-(p-1)+\frac{p-1}{p}$ \cite{ACGY}. In particular, $\cB_3$ is the simple affine vertex algebra $L_{-4/3}(\mathfrak{sl}_2)$ (it was first shown that $L_{-4/3}(\mathfrak{sl}_2)$ is an extension of $\cH\otimes\cM(3)$ in \cite{Ad_sl2_-4/3}). The tensor structure on the category $\Oloc_{\cB_p}$ of strongly $\cH$-weight graded local $\cB_p$-modules was predicted in \cite{ACKR} based on conjectural relationships between the singlet algebra $\cM(p)$ and the unrolled restricted quantum group of $\mathfrak{sl}_2$ at the root of unity $e^{\pi i/p}$. Here, we confirm the predictions of \cite{ACKR}.

The decomposition of $\cB_p$ as an $\cH\otimes\cM(p)$-module is as follows \cite{CRW, ACKR}:
\begin{equation}\label{eqn:Bp_decomp}
\mathcal{B}_p = \bigoplus_{k \in \ZZ} \cF_{k\lambda_p}^\cH \otimes \cM_{1-k,1},
\end{equation}
where $\lambda_p \in i\mathbb{R}$ is such that $\lambda_p^2 = -\frac{p}{2}$. To be concrete, we fix a choice $i$ of square root of $-1$ and set $\lambda_p=i\sqrt{p/2}$. Then in the notation of the previous section, 
\begin{equation*}
J=\cF^\cH_{-i\sqrt{p/2}}\otimes\cM_{2,1},\qquad \lambda_J=-i\sqrt{p/2},\qquad r_J=1.
\end{equation*}
For $k\in\ZZ$, the lowest conformal weight of $\cF_{k\lambda_p}^\cH\otimes\cM_{1-k,1}$ is
\begin{equation*}
\frac{1}{2}(k\lambda_p)^2+h_{1+\vert k\vert,1} =\frac{1}{2}\vert k\vert(p-1),
\end{equation*}
so $\cB_p$ is a $\frac{1}{2}\ZZ$-graded vertex operator algebra if $p$ is even and a $\ZZ$-graded vertex operator algebra if $p$ is odd.

The algebraic torus $\CC/\ZZ$ acts on $\cB_p$ by automorphisms:
\begin{align*}
\CC/\ZZ & \rightarrow \mathrm{Aut}(\cB_p)\\
a+\ZZ & \mapsto e^{2\pi i ah(0)/\lambda_p},
\end{align*}
with $\cH\otimes\cM(p)$ the fixed-point subalgebra. In particular, the finite cyclic group $\ZZ/m\ZZ$ for $m\in\ZZ_{\geq 1}$ acts by automorphisms on $\cB_p$. For $m=2$, the automorphism group $\ZZ/2\ZZ$ of $\cB_p$ is generated by the involution $\theta$ of the previous section. For $p$ even (in which case $\cB_p$ is $\frac{1}{2}\ZZ$-graded) the fixed-point subalgebra $\cB_p^{\ZZ/2\ZZ}$ is precisely the $\ZZ$-graded subalgebra, and $\theta$-twisted $\cB_p$-modules form the Ramond sector of the tensor category $\cO_{\cB_p}$.

By Theorem \ref{thm:main_thm}, the category $\cO_{\cB_p}=\Oloc_{\cB_p}\oplus\Otw_{\cB_p}$ of finitely-generated strongly $\cH$-weight-graded local and $\theta$-twisted $\cB_p$-modules is a rigid braided $\ZZ/2\ZZ$-crossed tensor category, Moreover, $\Oloc_{\cB_p}$ is a rigid braided tensor subcategory of $\cO_{\cB_p}$ which is ribbon when $p$ is odd. Theorem \ref{thm:simple_A-module_classification} classifies the simple objects of $\cO_{\cB_p}$ and $\Oloc_{\cB_p}$; since $\lambda_J^2+r_J^2\frac{p}{2}=0$ for $\cB_p$, Remark \ref{rem:mod_class_for_Bp_and_orbs} shows that the simple objects of $\cO_{\cB_p}$ are naturally parametrized by $(\CC/\frac{1}{2}L)\times\frac{1}{2}\ZZ$. Thus in the next theorem, we slightly modify the notation of Theorem \ref{thm:simple_A-module_classification} for simple objects:

\begin{thm}\label{thm:Bp_mod_class}
 Every simple object in $\cO_{\mathcal{B}_p}$ is isomorphic to exactly one of the following:
 \begin{align*}
&  \cW_s^{(\ell)} :=\cF\left(\cF^\cH_{-i\ell\alpha_-}\otimes\cM_{1,s}\right)
\end{align*}
for $1\leq s\leq p$ and $\ell\in\frac{1}{2}\ZZ$, or
\begin{align*}
& \cE_{\overline{\lambda}}^{(\ell)}:=\cF\left(\cF^{\cH}_{i(\lambda-\alpha_0/2-\ell\alpha_-)}\otimes\cF_{\lambda}\right)
 \end{align*}
 for $\overline{\lambda}=\lambda+\frac{1}{2}L\in(\CC\setminus L^\circ)\big/\frac{1}{2}L$, $\ell\in\frac{1}{2}\ZZ$. Moreover, $\cW_s^{(\ell)}$ is an object of $\Oloc_{\cB_p}$ if and only if $\ell\in\frac{s-1}{2}+\ZZ$, and $\cE^{(\ell)}_{\overline{\lambda}}$ is an object of $\Oloc_{\cB_p}$ if and only if $\ell\in\frac{p-1}{2}+\ZZ$.
 \end{thm}
 
Theorem \ref{thm:gen_proj_covers} classifies the indecomposable projective objects of $\cO_{\cB_p}$:
 
\begin{thm}\label{thm:Bp_proj_modules}
The simple $\mathcal{B}_p$-modules $\cW_p^{(\ell)}$ and $\cE_{\overline{\lambda}}^{(\ell)}$ for $\overline{\lambda}\in(\CC\setminus L^\circ)\big/\frac{1}{2}L$, $\ell\in\frac{1}{2}\ZZ$ are projective in $\cO_{\mathcal{B}_p}$. For $1\leq s\leq p-1$ and $\ell\in\frac{1}{2}\ZZ$, $\cW_{s}^{(\ell)}$ has a projective cover
 \begin{equation*}
  \mathcal{Q}_{s}^{(\ell)}:=\cF\left(\cF^\cH_{i\ell\alpha_-}\otimes\cP_{1,s}\right)
 \end{equation*}
in $\cO_{\mathcal{B}_p}$, which has Loewy diagram
\begin{equation}\label{eqn:Qsl_Loewy_diag}
 \begin{matrix}
  \begin{tikzpicture}[->,>=latex,scale=1.5]
\node (b1) at (1,0) {$\cW_{s}^{(\ell)}$};
\node (c1) at (-1, 1){$\mathcal{Q}_{s}^{(\ell)}$:};
   \node (a1) at (0,1) {$\cW_{p-s}^{(\ell-p/2)}$};
   \node (b2) at (2,1) {$\cW_{p-s}^{(\ell+p/2)}$};
    \node (a2) at (1,2) {$\cW_{s}^{(\ell)}$};
\draw[] (b1) -- node[left] {} (a1);
   \draw[] (b1) -- node[left] {} (b2);
    \draw[] (a1) -- node[left] {} (a2);
    \draw[] (b2) -- node[left] {} (a2);
\end{tikzpicture}
\end{matrix} .
 \end{equation}
 \end{thm}
 
 The theorems in Sections \ref{subsec:lb_hw_gr} and \ref{subsec:C1-cofinite} describe the categories of highest-weight, grading-restricted, and $C_1$-cofinite $\cB_p$-modules, and we can also use the second assertion in Theorem \ref{thm:Bp_mod_class} to determine how many simple objects in these categories are local:
\begin{thm}\label{thm:Bp_grad-rest}
For the vertex operator algebra $\cB_p$:
\begin{enumerate}
\item The highest-weight category $\cO^{\mathrm{h.w.}}_{\cB_p}$ is semisimple with $p(p-1)$ simple highest-weight local and $\theta$-twisted modules up to isomorphism, namely the modules $\cW_s^{(\ell)}$ for $1\leq s\leq p-1$ and $-\frac{p-s}{2}<\ell\leq\frac{p-s}{2}$. Of these simple modules, $\frac{1}{2}p(p-1)$ are local.

\item The category of grading-restricted generalized  ($\theta$-twisted) $\cB_p$-modules is semisimple with $(p-1)^2$ simple objects up to isomorphism, namely the modules $\cW_s^{(\ell)}$ for $1\leq s\leq p-1$ and $\vert\ell\vert<\frac{p-s}{2}$. Of these simple modules, $\frac{1}{2} p(p-1)$ are local when $p$ is even and $\frac{1}{2}(p-1)(p-2)$ are local when $p$ is odd.

\item The category of $C_1$-cofinite grading-restricted generalized $\cB_p$-modules is semisimple with one simple object $\cW_1^{(0)}\,\,(\cong\cB_p)$ up to isomorphism.
\end{enumerate}

\end{thm}

 Theorem \ref{thm:general_fusion_rules} gives the tensor products of simple objects in $\cO_{\cB_p}$:
  \begin{thm}\label{thm:Bp_tensor_products}
 Tensor products of simple modules in $\cO_{\mathcal{B}_p}$ are as follows:
 \begin{enumerate}
 
\item For $1\leq s,s'\leq p$ and $\ell,\ell'\in\frac{1}{2}\ZZ$, 
\begin{equation*}
  \cW_{s}^{(\ell)}\tens \cW_{s'}^{(\ell')}\cong\bigoplus_{\substack{t=\vert s-s'\vert+1\\ t+s+s'\equiv 1\,(\mathrm{mod}\,2)}}^{\mathrm{min}(s+s'-1, 2p-1-s-s')} \cW_{t}^{(\ell+\ell')}\oplus\bigoplus_{\substack{t=2p+1-s-s'\\ t+s+s'\equiv 1\,(\mathrm{mod}\,2)}}^p \mathcal{Q}_{t}^{(\ell+\ell')}.
 \end{equation*}
 
 \item For $1\leq s\leq p$, $\overline{\lambda}\in(\CC\setminus L^\circ)\big/\frac{1}{2}L$, and $\ell,\ell'\in\frac{1}{2}\ZZ$,
 \begin{equation*}
 \cW_s^{(\ell)}\tens\cE_{\overline{\lambda}}^{(\ell')} \cong\bigoplus_{k=0}^{s-1} \cE_{\overline{\lambda+\alpha_{1,s}+k\alpha_-}}^{(\ell+\ell'+k-(s-1)/2)}.
\end{equation*}
 
 \item For $\ell,\ell'\in\frac{1}{2}\ZZ$ and $\overline{\lambda},\overline{\mu}\in(\CC\setminus L^\circ)\big/\frac{1}{2}L$ such that $\overline{\lambda+\mu}=\alpha_0+\alpha_{1,s}+\frac{1}{2}L$,
 \begin{equation*}
\cE_{\overline{\lambda}}^{(\ell)}\tens\cE_{\overline{\mu}}^{(\ell')}\cong\bigoplus_{\substack{t=s\\ t\equiv s\,(\mathrm{mod}\,2)}}^p \mathcal{Q}_{t}^{(\ell+\ell'+(s-1)/2)} \oplus \bigoplus_{\substack{t=p-s+2\\ t\equiv p-s\,(\mathrm{mod}\,2)}}^p \mathcal{Q}_{t}^{(\ell+\ell'-(p-s+1)/2)}.
\end{equation*}

\item For $\ell,\ell'\in\frac{1}{2}\ZZ$ and $\overline{\lambda},\overline{\mu},\overline{\lambda+\mu}\in(\CC\setminus L^\circ)\big/ \frac{1}{2}L$,
\begin{equation*}
 \cE_{\overline{\lambda}}^{(\ell)}\tens\cE_{\overline{\mu}}^{(\ell')}\cong\bigoplus_{k=0}^{p-1} \cE_{\overline{\lambda+\mu+k\alpha_-}}^{(\ell+\ell'+k-(p-1)/2)} .
 \end{equation*}
 
 \end{enumerate}

\end{thm}

\subsection{Finite cyclic orbifolds of \texorpdfstring{$\cB_2$}{B2}}

For all $m\in\ZZ_{>0}$, the fixed-point subalgebra $\cB_2^m:=\cB_2^{\ZZ/m\ZZ}$ of $\cB_2$ is the simple subregular $W$-algebra of $\mathfrak{sl}_m$ at level $-m+\frac{m+1}{m}$ \cite[Theorem 9.4]{CL-trialities}. In particular, $\cB_2^{2}$ is the simple affine vertex algebra $L_{-1/2}(\mathfrak{sl}_2)$, which is well known. It is clear from \eqref{eqn:Bp_decomp} that $\cB_2^{m}$ has the decomposition
\begin{equation*}
\cB_2^{m} =\bigoplus_{k\in\ZZ} \cF^\cH_{ikm}\otimes\cM_{1-km,1}
\end{equation*}
as an $\cH\otimes\cM(2)$-module, where $i$ is a square root of $-1$. Thus for $\cB_2^m$, 
\begin{equation*}
J=\cF^\cH_{-im}\otimes\cM_{m+1,1},\qquad\lambda_J=-im,\qquad r_J=m.
\end{equation*}
For $k\in\ZZ$, the lowest conformal weight of $\cF^\cH_{ikm}\otimes\cM_{1-km,1}$ is $\frac{m}{2}\vert k\vert$, so that $\cB_2^m$ is $\frac{1}{2}\ZZ$-graded if $m$ is odd and $\ZZ$-graded if $m$ is even.

By Theorem \ref{thm:main_thm}, the category $\cO_{\cB_2^m}$ of finitely-generated strongly $\cH$-weight-graded local and $\theta$-twisted $\cB_2^m$-modules is a rigid braided $\ZZ/2\ZZ$-crossed tensor category, and $\Oloc_{\cB_2^m}$ is a rigid braided tensor subcategory which is ribbon in case $m$ is even. By Remark \ref{rem:mod_class_for_Bp_and_orbs}, simple objects of $\cO_{\cB_2^m}$ are naturally indexed by $(\CC/m\ZZ)\times\frac{1}{2}\ZZ$ since $L=\ZZ\alpha_+=2\ZZ$ when $p=2$. To avoid confusion with the notation for $\cB_p$-modules in the previous subsection, we will use the alternate notation $\cX_{\overline{r}}^{(\ell)}:=\cW_{r,1}^{(\ell)}$ for $\overline{r}=r+m\ZZ\in\ZZ/m\ZZ$ and $\cG_{\overline{\lambda}}^{(\ell)}:=\cE_{\lambda}^{(\ell)}$ for $\overline{\lambda}=\lambda+m\ZZ\in(\CC\setminus\frac{1}{2}\ZZ)/m\ZZ$. The modules $\cW_{r,2}^{(\ell)}$ (in the notation of Theorem \ref{thm:simple_A-module_classification}) are more naturally grouped with the typical modules $\cG^{(\ell)}_{\overline{\lambda}}$, so we denote them by $\cG^{(\ell)}_{\overline{\alpha_{r,2}}}$ in the following theorem (and recall that $\alpha_{r,2}=\frac{3}{2}-r$ when $p=2$):

\begin{thm}\label{thm:B2m_simple_objects}
Every simple object in $\cO_{\cB_2^m}$ is isomorphic to exactly one of the following:
 \begin{align*}
&  \cX_{\overline{r}}^{(\ell)} :=\cF\left(\cF^\cH_{i(1-r +\ell/m)}\otimes\cM_{r,1}\right)
\end{align*}
for $\overline{r} = r + m\ZZ \in \ZZ/m\ZZ$ and $\ell\in\frac{1}{2}\ZZ$, or
\begin{align*}
& \cG_{\overline{\lambda}}^{(\ell)}:=\cF\left(\cF^{\cH}_{i(\lambda-1/2+\ell/m)}\otimes\cF_{\lambda}\right)
 \end{align*}
 for $\overline{\lambda}=\lambda+m\ZZ\in(\CC\setminus \ZZ)/m\ZZ$ and $\ell\in\frac{1}{2}\ZZ$. Moreover, $\cX_{\overline{r}}^{(\ell)}$ is an object of $\Oloc_{\cB_2^m}$ if and only if $\ell\in\ZZ$, and $\cG_{\overline{\lambda}}^{(\ell)}$ is an object of $\Oloc_{\cB_2^m}$ if and only if $\ell\in\frac{m}{2}+\ZZ$.
\end{thm}

Indecomposable projective objects in $\cO_{\cB_2^m}$ are given by Theorem \ref{thm:gen_proj_covers}:
\begin{thm}\label{thm:B2m_proj_modules}
The simple $\cB_2^m$-modules $\cG^{(\ell)}_{\overline{\lambda}}$ for $\overline{\lambda}\in(\CC\setminus\ZZ)/m\ZZ$, $\ell\in\frac{1}{2}\ZZ$ are projective in $\cO_{\cB_2^m}$. For $\overline{r}=r+m\ZZ\in\ZZ/m\ZZ$ and $\ell\in\frac{1}{2}\ZZ$, $\cX_{\overline{r}}^{(\ell)}$ has a projective cover
\begin{equation*}
\mathcal{R}_{\overline{r}}^{(\ell)}:=\cF\left(\cF^\cH_{i(r-1-\ell/m)}\otimes\cP_{r,1}\right)
\end{equation*}
in $\cO_{\cB_2^m}$, which has Loewy diagram
\begin{equation}\label{eqn:Qsl_Loewy_diag:orbifold2}
 \begin{matrix}
  \begin{tikzpicture}[->,>=latex,scale=1.5]
\node (b1) at (1,0) {$\cX_{\overline{r}}^{(\ell)}$};
\node (c1) at (-1, 1){$\mathcal{R}_{\overline{r}}^{(\ell)}:$};
   \node (a1) at (0,1) {$\cX_{\overline{r-1}}^{(\ell-m)}$};
   \node (b2) at (2,1) {$\cX_{\overline{r+1}}^{(\ell+m)}$};
    \node (a2) at (1,2) {$\cX_{\overline{r}}^{(\ell)}$};
\draw[] (b1) -- node[left] {} (a1);
   \draw[] (b1) -- node[left] {} (b2);
    \draw[] (a1) -- node[left] {} (a2);
    \draw[] (b2) -- node[left] {} (a2);
\end{tikzpicture}
\end{matrix} .
 \end{equation}
\end{thm}

Using Sections \ref{subsec:lb_hw_gr} and \ref{subsec:C1-cofinite}, we can describe the categories of highest-weight, grading-restricted, and $C_1$-cofinite $\cB_2^m$-modules:
\begin{thm}\label{thm:B2m_grad-rest}
For the vertex operator algebra $\cB_2^m$,
\begin{enumerate}

\item The highest-weight category $\cO_{\cB_2^m}^{\mathrm{h.w.}}$ is semisimple with $2m^2$ simple highest-weight local and $\theta$-twisted modules up to isomorphism, namely the modules $\cX_{\overline{r}}^{(\ell)}$ for $\overline{r}\in\ZZ/m\ZZ$ and $-\frac{m}{2}<\ell\leq\frac{m}{2}$. Of these simple modules, $m^2$ are local.

\item The category of grading-restricted generalized ($\theta$-twisted) $\cB_2^m$-modules is semisimple with $m(2m-1)$ simple objects up to isomorphism, namely the modules $\cX_{\overline{r}}^{(\ell)}$ for $\overline{r}\in\ZZ/m\ZZ$ and $\vert\ell\vert<\frac{m}{2}$. Of these simple modules, $m^2$ are local when $m$ is odd and $m(m-1)$ are local when $m$ is even.

\item The category of $C_1$-cofinite grading-restricted generalized $\cB_2^m$-modules is semisimple with $m$ simple objects $\cX_{\overline{r}}^{(0)}$, $\overline{r}\in\ZZ/m\ZZ$, up to isomorphism.

\end{enumerate}
\end{thm}

We also get tensor products of simple objects in $\cO_{\cB_p^m}$ from Theorem \ref{thm:general_fusion_rules}:

\begin{thm}\label{thm:B2m_tens_prods}
Tensor products of simple modules in $\cO_{\cB_2^{m}}$ are as follows:
\begin{enumerate}

\item  For $\overline{r},\overline{r'}\in\ZZ/m\ZZ$, $\overline{\lambda}\in(\CC\setminus\ZZ)/m\ZZ$, and $\ell,\ell'\in\frac{1}{2}\ZZ$,
\begin{equation*}
\cX_{\overline{r}}^{(\ell)}\tens \cX_{\overline{r'}}^{(\ell')}\cong \cX_{\overline{r+r'-1}}^{(\ell+\ell')},\qquad\qquad\cX_{\overline{r}}^{(\ell)}\tens \cG_{\overline{\lambda}}^{(\ell')}\cong\cG_{\overline{\lambda-r+1}}^{(\ell+\ell')}.
\end{equation*}


\item For $\overline{\lambda},\overline{\mu}\in(\CC\setminus\ZZ)/m\ZZ$ and $\ell,\ell'\in\frac{1}{2}\ZZ$,
 \begin{equation*}
 \cG_{\overline{\lambda}}^{(\ell)}\tens\cG_{\overline{\mu}}^{(\ell')}\cong\left\lbrace\begin{array}{lll}
 \cR_{\overline{2-\lambda-\mu}}^{(\ell+\ell')} & \text{if} & \overline{\lambda+\mu}\in\ZZ/m\ZZ\\
\cG^{(\ell+\ell'-m/2)}_{\overline{\lambda+\mu}}\oplus\cG^{(\ell+\ell'+m/2)}_{\overline{\lambda+\mu-1}} & \text{if} & \overline{\lambda+\mu}\in(\CC\setminus\ZZ)/m\ZZ \\
 \end{array}\right. .
%
%
 \end{equation*}
\end{enumerate}
\end{thm}

In the previous subsection, we saw that $\cB_p$ has a single simple $C_1$-cofinite module, $\cB_p$ itself. For $\cB_2^m$, Theorem \ref{thm:B2m_grad-rest}(3) shows that the category $\cC^1_{\cB_2^m}$ of $C_1$-cofinite $\cB_2^m$ modules has $m$ inequivalent simple modules, which are precisely the simple currents appearing in the decomposition of $\cB_2$ as a $\cB_2^m$-module. It then follows from \cite[Corollary 4.8]{McR-orb1} that $\cC_{\cB_2^m}^1$ is a symmetric tensor category equivalent to $\rep\,\ZZ /m\ZZ$. Recalling that $\cB_2^m$ is the subregular quantum Hamiltonian reduction of the simple affine vertex operator algebra $L_k(\mathfrak{sl}_m)$ at level $k = -m + \frac{m+1}{m}$, we can show that the category of $C_1$-cofinite $L_k(\mathfrak{sl}_m)$-modules is also equivalent to $\rep\,\ZZ/m\ZZ$, suggesting that subregular quantum Hamiltonian reduction is a braided tensor equivalence in these examples:
\begin{prop}\label{prop:qHr_braided_tensor?}
There are equivalences of symmetric tensor categories 
\begin{equation*}
\cC^1_{L_k(\mathfrak{sl}_m)}\cong\rep\,\ZZ/m\ZZ\cong\cC_{\cB_2^m}^1.
\end{equation*}
\end{prop}
\begin{proof}
It remains to prove the first equivalence. The category $\cC^1_{L_k(\mathfrak{sl}_m)}$ is a braided fusion category \cite{Cr2} with exactly $m$ inequivalent simple objects, namely the modules  $L_{k}( \omega_i)$ 
 whose lowest conformal weight spaces are the irreducible $\mathfrak{sl}_m$-modules of highest weights $\omega_i$, where $\omega_0 =0$ and $\omega_i$ for $1\leq i\leq m-1$ is the $i$th fundamental weight. Let $H^0_{DS,\text{prin}}$ be the principal quantum Hamiltonian reduction functor from $L_k(\mathfrak{sl}_m)$-modules to modules for $H^0_{DS,\text{prin}}(L_k(\mathfrak{sl}_m))\cong\CC$. Since a module for $L_k(\mathfrak{sl}_m)$ or one of its quantum Hamiltonian reductions is $C_1$-cofinite if and only if it is finitely strongly generated \cite[Lemma 3.1.6]{Ar-assoc-var}, principal quantum Hamiltonian reduction restricts to a functor $H^0_{DS,\text{prin}}: \cC^1_{L_k(\mathfrak{sl}_m)}\rightarrow\cV ec$ by \cite[Corollary 4.5.10]{Ar-C2}. By \cite[Theorem 10.4]{ACF}, $H^0_{DS,\text{prin}}$ is also a braided tensor functor, which implies in particular that $\cC^1_{L_k(\mathfrak{sl}_m)}$ is symmetric.
 
 Moreover, $H^0_{DS,\text{prin}}$ is exact (see the statement in \cite[Theorem 3.1]{ACF}) and faithful (since it is an exact tensor functor between rigid abelian tensor categories; see \cite[Proposition 1.19]{DM}). Thus $H^0_{DS,\text{prin}}$ is a fiber functor, and Tannakian reconstruction (as in \cite[Theorem 2.11]{DM}) implies that $\cC^1_{L_k(\mathfrak{sl}_m)}$ is equivalent to the category $\rep\,G$ of finite-dimensional representations of a (finite) group $G$. Then $G$ must be $\ZZ/m\ZZ$ because \cite[Corollary 7.4]{Cr2} shows that $C_1$-cofinite $L_k(\mathfrak{sl}_m)$-modules have the same $\ZZ/m\ZZ$ fusion rules as modules for the $\mathfrak{sl}_m$-root lattice vertex operator algebra $L_1(\mathfrak{sl}_m)$. 
\end{proof}

\begin{rem}\label{rem:qHr_braided_tensor}
The preceding proposition suggests that subregular quantum Hamiltonian reduction defines a braided tensor equivalence 
\begin{equation*}
H^0_{DS,\text{subreg}}: \cC^1_{L_k(\mathfrak{sl}_m)}\longrightarrow\cC^1_{\cB_2^m}.
\end{equation*}
But although we do get such a functor by \cite[Lemma 3.1.6]{Ar-assoc-var} and \cite[Corollary 4.5.10]{Ar-C2}, we cannot use \cite[Theorem 10.4]{ACF} to conclude it is a braided tensor functor because we do not yet have braided tensor category structure on a suitable module category for the simple subregular $W$-algebra of $\mathfrak{sl}_m$ at level $k+1$. We can at least see that $H^0_{DS,\text{subreg}}$ is an equivalence of categories: Let $J$ be the Heisenberg field of $\cB_2^m$ normalized so that the strong generators of $\cB_2^m$ of conformal weight $\frac{m}{2}$ have $J(0)$-eigenvalues $\pm 1$. 
 Then for $0\leq i\leq m-1$, $H^0_{DS,\text{subreg}}(L_{k}( \omega_i))$ is a highest-weight $\cB_2^m$-module whose highest-weight vector has $J(0)$-eigenvalue $\frac{i}{m}$. Comparing these $J(0)$-weights and using Theorem \ref{thm:B2m_simple_objects}, we then observe that 
 $H^0_{DS,\text{subreg}}(L_{k}( \omega_i)) \cong \cX_{\overline{i+1}}^{(0)}$, so that $H^0_{DS,\text{subreg}}$ takes the $m$ distinct simple $C_1$-cofinite $L_k(\mathfrak{sl}_m)$-modules to the $m$ distinct simple $C_1$-cofinite $\cB_2^m$-modules.
\end{rem}

\subsection{Some special cases}\label{subsec:special_cases}

In this subsection, we compare our results for $\cB_2^m=\cB_2^{\ZZ/m\ZZ}$, $m=1,2,3$ and $\cB_p$, $p=2,3,4$, with previous conjectures and results from the mathematical physics literature. These particular algebras have received substantial attention: $\cB_2$ is the $\beta\gamma$-vertex algebra, $\cB_3$ and $\cB_2^{2}$ are vertex algebras associated to affine $\mathfrak{sl}_2$ at admissible levels, and $\cB_4$ and $\cB_2^{3}$ are Bershadsky-Polyakov algebras (that is, subregular affine $W$-algebras associated to $\mathfrak{sl}_3$) at admissible levels.

\subsubsection{The \texorpdfstring{$\beta\gamma$}{beta-gamma}-vertex algebra}

The vertex operator algebra $\cB_2$, equivalently $\cB_2^1$, is the $\beta\gamma$-vertex algebra, and our results in this section recover many obtained previously. The tensor products of Theorem \ref{thm:B2m_tens_prods} were predicted in \cite{RW} using a conjectural Verlinde formula, and the corresponding intertwining operators were constructed in \cite{AP}. The full rigid braided tensor category structure on $\Oloc_{\cB_2}$ was constructed in \cite{AW}; we have here somewhat enhanced the results of \cite{AW} to include $\theta$-twisted modules. Recently, a larger tensor category of (local) $\cB_2$-modules incorporating modules with non-semisimple $h(0)$-actions has been constructed in \cite{BN} using relations between the $\beta\gamma$-vertex algebra and affine $\mathfrak{gl}_{1\vert 1}$, as well as a tensor category for affine $\mathfrak{gl}_{1\vert 1}$ constructed in \cite{CMY3}.

To compare our results in the previous subsections with those of \cite{AW}, recall from Theorem \ref{thm:B2m_simple_objects} that the simple objects of $\Oloc_{\cB_2}$ are the modules $\cX_{\overline{1}}^{(\ell)}$ for $\ell\in\ZZ$ and $\cG_{\overline{\lambda}}^{(\ell)}$ for $\overline{\lambda}\in (\CC\setminus\ZZ)/\ZZ$, $\ell\in\frac{1}{2}+\ZZ$; of these modules, only $\cX_{\overline{1}}^{(0)}$ is lower bounded or grading restricted. However, \cite{AW} uses a different conformal vector in $\cB_2$, and with the corresponding different conformal weight gradings, $\cX_{\overline{1}}^{(\ell)}$ for $\ell=0,-1$ and $\cG_{\overline{\lambda}}^{(-1/2)}$ for $\overline{\lambda}\in(\CC\setminus\ZZ)/\ZZ$ (in our notation) are lower bounded, but none are grading restricted. To compare our notation for simple modules with that of \cite{AW}, the simple modules in \cite{AW} are labeled by eigenvalues for a multiple $J(0)$ of our Heisenberg zero-mode $h(0)$, and according to \cite[Equation 2.15]{AW}, one unit of spectral flow changes $J(0)$-eigenvalues by $-1$. On the other hand, one unit of spectral flow changes $h(0)$-eigenvalues by $i$ (recall Remark \ref{rem:spectral_flow}), so $J(0)=i h(0)$. Consequently, the lower-bounded module $\cW_\lambda$ (in the notation of \cite{AW}) with $J(0)$-eigenvalues in $\lambda+\ZZ$ is identified with the module $\cG_{\overline{-\lambda}}^{(-1/2)}$ (in our notation).

More generally, we have the following comparisons of our notation with that of \cite{AW}:
\begin{equation*}
\cX_{\overline{1}}^{(\ell)} \longleftrightarrow \sigma^\ell(\cV),\qquad \cR_{\overline{1}}^{(\ell)} \longleftrightarrow \sigma^\ell(\cP),\qquad\cG_{\overline{-\lambda}}^{(\ell-1/2)}\longleftrightarrow\sigma^\ell(\cW_\lambda)
\end{equation*}
for $\ell\in\ZZ$, $\lambda\in\CC\setminus\ZZ$. Thus Theorem \ref{thm:B2m_tens_prods}(2), for example, implies
\begin{align*}
\cW_\lambda\tens\cW_\mu & =\cG^{(-1/2)}_{\overline{-\lambda}}\tens\cG^{(-1/2)}_{\overline{-\mu}}\cong \left\lbrace\begin{array}{lll}
\cR_{\overline{1}}^{(-1)} & \text{if} & \lambda+\mu\in\ZZ\\
\cG_{\overline{-\lambda-\mu}}^{(-3/2)} \oplus \cG_{\overline{-\lambda-\mu-1}}^{(-1/2)} & \text{if} & \lambda+\mu\notin\ZZ\\
\end{array} \right.\nonumber\\
& = \left\lbrace\begin{array}{lll}
\sigma^{-1}(\cP) & \text{if} & \lambda+\mu\in\ZZ\\
\cW_{\lambda+\mu} \oplus \sigma^{-1}(\cW_{\lambda+\mu}) & \text{if} & \lambda+\mu\notin\ZZ\\
\end{array} \right. 
\end{align*}
for $\lambda,\mu\in\CC$, in agreement with \cite[Equation 6.5]{AW}.

Theorem \ref{thm:grad_rest_ss} implies that the category of grading-restricted generalized $\cB_2$-modules is semisimple with only one simple object, $\cB_2$ itself. This agrees with the recent result \cite[Theorem 55]{BBOPY} that $\cB_2$ is rational as a $\CC$-graded vertex algebra with respect to certain conformal vectors.

\subsubsection{Affine \texorpdfstring{$\mathfrak{sl}_2$}{sl(2)} at levels \texorpdfstring{$-\frac{1}{2}$}{-1/2} and \texorpdfstring{$-\frac{4}{3}$}{-4/3}}

The $\ZZ/2\ZZ$-orbifold $\cB_2^2$ is the simple affine vertex operator algebra $L_{-1/2}(\mathfrak{sl}_2)$ associated to $\mathfrak{sl}_2$ at level $-\frac{1}{2}$, and $\cB_3$ is isomorphic to $L_{-4/3}(\mathfrak{sl}_2)$. As the simplest examples of affine vertex operator algebras at admissible levels, they have been heavily studied in the mathematical physics literature. Strongly $\cH$-weight-graded $L_{-1/2}(\mathfrak{sl}_2)$-modules and their fusion rules have been studied in \cite{LMRS, Ri1, Ri2, CR1}, while tensor products of strongly $\cH$-weight-graded $L_{-4/3}(\mathfrak{sl}_2)$ have been predicted in \cite{Ga, CR1, ACKR}. Going beyond levels $-\frac{1}{2}$ and $-\frac{4}{3}$, fusion rules for $L_k(\mathfrak{sl}_2)$ at general admissible $k$ have been predicted in \cite{CR2} using a conjectural Verlinde formula, and logarithmic modules and some intertwining operators have been constructed in \cite{Ad-sl2_osp1|2}. For general affine vertex operator algebras at admissible levels, tensor category structure has been obtained on the category of grading-restricted modules in \cite{CHY}, and this category is rigid at least when the finite-dimensional simple Lie algebra is simply laced \cite{Cr2}. For $L_{-1/2}(\mathfrak{sl}_2)$ and $L_{-4/3}(\mathfrak{sl}_2)$, however, the grading-restricted module category is almost trivial. Here, we have obtained for the first time rigid braided tensor category structure on the full categories of finitely-generated strongly $\cH$-weight-graded modules for affine $\mathfrak{sl}_2$ at levels $-\frac{1}{2}$ and $-\frac{4}{3}$.

Let us compare our results in Theorem \ref{thm:B2m_tens_prods} on tensor products of simple modules in $\Oloc_{L_{-1/2}(\mathfrak{sl}_2)}$ with the fusion rules summarized in \cite[Section 2.3]{CR1}. First, simple lower-bounded $L_{-1/2}(\mathfrak{sl}_2)$-modules are labeled in \cite{CR1} by eigenvalues for a certain multiple of our Heisenberg zero-mode $h(0)$. Denoting this multiple by $H(0)$, \cite[Figure 1]{CR1} indicates that one unit of spectral flow changes $H(0)$-eigenvalues by $-\frac{1}{2}$, while Remark \ref{rem:spectral_flow} implies that one unit of spectral flow changes $h(0)$-eigenvalues by $\frac{i}{2}$. Thus $H(0)=ih(0)$. Then comparing the $m=2$ case of Theorem \ref{thm:B2m_simple_objects} with the modules in \cite{CR1}, we obtain the following dictionary:
\begin{equation*}
\cX_{\overline{r+1}}^{(\ell)}\longleftrightarrow\sigma^\ell(\cL_{r}),\qquad\cR_{\overline{r+1}}^{(\ell)}\longleftrightarrow\sigma^\ell(\cS_{r}),\qquad \cG^{(\ell)}_{\overline{1/2-\lambda}}\longleftrightarrow \sigma^\ell(\cE_\lambda)
\end{equation*}
for $\ell\in\ZZ$, $r=0,1$, and $\overline{\lambda}\in(\CC\setminus\ZZ)/2\ZZ$. We can now compare fusion rules. First, the compatibility of spectral flow with fusion in \cite[Equation 2.10]{CR1} is clear in all cases of Theorem \ref{thm:B2m_tens_prods}. Then the $\ell=\ell'=0$ case of Theorem \ref{thm:B2m_tens_prods} yields
\begin{equation*}
\cL_r\tens\cL_{r'} =\cX_{\overline{r+1}}^{(0)}\tens\cX_{\overline{r'+1}}^{(0)}\cong\cX_{\overline{r+r'+1}}^{(0)} =\cL_{r+r'}
\end{equation*}
for $\overline{r},\overline{r'}\in\ZZ/2\ZZ$,
\begin{equation*}
\cL_r\tens\cE_\lambda = \cX_{\overline{r+1}}^{(0)}\tens\cG_{\overline{1/2-\lambda}}^{(0)}\cong\cG^{(0)}_{\overline{1/2-\lambda-r}} = \cE_{r+\lambda}
\end{equation*}
for $\overline{r}\in\ZZ/2\ZZ$, $\overline{\lambda}\in(\CC\setminus\ZZ)/2\ZZ$, and
\begin{align*}
\cE_\lambda\tens\cE_\mu & =\cG^{(0)}_{\overline{1/2-\lambda}}\tens\cG^{(0)}_{\overline{1/2-\mu}}\cong\left\lbrace\begin{array}{lll}
\cR_{\overline{\lambda+\mu-1}}^{(0)} & \text{if} & \lambda+\mu\in\ZZ\\
 \cG_{\overline{1-\lambda-\mu}}^{(-1)}\oplus\cG_{\overline{-\lambda-\mu}}^{(1)} & \text{if} & \lambda+\mu\notin\ZZ\\
\end{array}\right.\nonumber\\
& = \left\lbrace\begin{array}{lll}
\cS_{\lambda+\mu} & \text{if} & \lambda+\mu\in\ZZ\\
\sigma^{-1}(\cE_{\lambda+\mu-1/2})\oplus \sigma(\cE_{\lambda+\mu+1/2}) & \text{if} & \lambda+\mu\notin\ZZ
\end{array}\right. 
\end{align*}
for $\overline{\lambda},\overline{\mu}\in(\CC\setminus\ZZ)/2\ZZ$, in agreement with \cite[Equation 2.11]{CR1}. For $\overline{\lambda},\overline{\mu}\in(\CC\setminus\ZZ)/2\ZZ$, the intertwining operator associated to the surjection $\cE_\lambda\tens\cE_\mu\rightarrow\sigma^{-1}(\cE_{\lambda+\mu-1/2})$ has also been constructed in \cite{Ad-sl2_osp1|2}.

We also compare our results for $\cB_3=L_{-4/3}(\mathfrak{sl}_2)$ with the fusion rules derived in \cite{Ga, CR1}. Comparing spectral flow shifts as before shows that irreducible $L_{-4/3}(\mathfrak{sl}_2)$-modules in \cite[Section 5]{CR1} are labeled by $2i\sqrt{2/3}h(0)$-eigenvalues (modulo $2\ZZ$). Specifically, we have the following correspondences between our notation in Theorems \ref{thm:Bp_mod_class} and \ref{thm:Bp_proj_modules} and the notation in \cite{CR1}:
\begin{equation*}
\cW_1^{(\ell)}\longleftrightarrow\sigma^\ell(\cL_0),\qquad\cW_2^{(\ell+1/2)}\longleftrightarrow\sigma^\ell(\cD^+_{-2/3}),\qquad\cW_3^{(\ell)}\longleftrightarrow\sigma^\ell(\cE_0)
\end{equation*}
for $\ell\in\ZZ$, and for the projective covers,
\begin{equation*}
\cQ_1^{(\ell)}\longleftrightarrow\sigma^\ell(\cS_0),\qquad \cQ_2^{(\ell+1/2)}\longleftrightarrow\sigma^\ell(\cS_{-2/3}^+).
\end{equation*}
Also, for $\lambda+2\ZZ\in(\CC\setminus\frac{2}{3}\ZZ)/2\ZZ$ and $\ell\in\ZZ$, we have the correspondence
\begin{equation*}
\cE^{(\ell)}_{\overline{\sqrt{2/3}(1-3\lambda/4)}}\longleftrightarrow\sigma^\ell(\cE_\lambda).
\end{equation*}
We also note that \cite{CR1} uses alternate notation for some of these modules:
\begin{equation*}
\cW_2^{(-1/2)}\longleftrightarrow\cD^-_{2/3}=\sigma^{-1}(\cD^+_{-2/3}),\qquad\cQ^{(-1/2)}_2\longleftrightarrow\cS^-_{2/3} =\sigma^{-1}(\cS^+_{-2/3}).
\end{equation*}

We can now verify the fusion rule conjectures of \cite{Ga, CR1}. First, as for $L_{-1/2}(\mathfrak{sl}_2)$, the compatibility of spectral flow with fusion in \cite[Equation 2.10]{CR1} is evident from the formulas in Theorem \ref{thm:Bp_tensor_products}. Then the fusion computations of \cite{Ga}, summarized in \cite[Equation 5.2]{CR1}, are verified using Theorem \ref{thm:Bp_tensor_products}:
\begin{equation*}
\cD_{2/3}^-\tens\cD_{-2/3}^+ = \cW_2^{(-1/2)}\tens\cW_2^{(1/2)} \cong \cW_1^{(0)}\oplus\cW_3^{(0)} = \cL_0\oplus\cE_0,
\end{equation*}

\begin{equation*}
\cD_{-2/3}^+\tens\cE_0 = \cW_2^{(1/2)}\tens\cW_3^{(0)} \cong \cQ_2^{(1/2)} = \cS^+_{-2/3},
\end{equation*}
and
\begin{equation*}
\cE_0\tens\cE_0 = \cW_3^{(0)}\tens\cW_3^{(0)} \cong \cQ_1^{(0)}\oplus\cQ_3^{(0)} =\cS_0\oplus\cE_0.
\end{equation*}
We also confirm that the fusion rule conjectured in \cite[Equation 8.2]{Ga} is incomplete, and that the revised conjecture in \cite[Equation 5.26]{CR1} is correct:
\begin{prop}
In the notation of Theorems \ref{thm:Bp_mod_class} and \ref{thm:Bp_proj_modules}, equivalently of \cite{CR1},
\begin{equation*}
\cW_2^{(1/2)}\tens\cQ_1^{(0)}\cong\cQ_2^{(1/2)}\oplus\cW_3^{(-1)}\oplus\cW_3^{(2)}\longleftrightarrow \cD_{-2/3}^+\tens\cS_0\cong\cS_{-2/3}^+\oplus\sigma^{-1}(\cE_0)\oplus\sigma^2(\cE_0).
\end{equation*}
\end{prop}
\begin{proof}
Since $\cW_2^{(1/2)}$ is rigid, $\cW_2^{(1/2)}\tens\bullet$ is exact. Thus there is a surjection
\begin{equation*}
\cW_2^{(1/2)}\tens\cQ_1^{(0)}\twoheadrightarrow\cW_2^{(1/2)}\tens\cW_1^{(0)}\xrightarrow{\cong}\cW_2^{(1/2)}.
\end{equation*}
Moreover, because $\cW_2^{(1/2)}$ is rigid and $\cQ_1^{(0)}$ is projective, $\cW_2^{(1/2)}\tens\cQ_1^{(0)}$ is also projective and thus properties of projective covers imply $\cQ_2^{(1/2)}$ is a direct summand of $\cW_2^{(1/2)}\tens\cQ_1^{(0)}$. On the other hand, the Loewy diagram \eqref{eqn:Qsl_Loewy_diag} of $\cQ_1^{(0)}$ and the exactness of $\cW_2^{(1/2)}\tens\bullet$ imply that there are exact sequences
\begin{equation*}
0\longrightarrow\cW_2^{(1/2)}\tens\cZ_1^{(0)}\longrightarrow\cW_2^{(1/2)}\tens\cQ_1^{(0)}\longrightarrow\cW_2^{(1/2)}\tens\cW_1^{(0)}\longrightarrow 0,
\end{equation*}
where $\cZ_1^{(0)}$ is the kernel of the surjection $\cQ_1^{(0)}\twoheadrightarrow\cW_1^{(0)}$, and 
\begin{equation*}
0\longrightarrow\cW_2^{(1/2)}\tens\cW_1^{(0)}\longrightarrow\cW_2^{(1/2)}\tens\cZ_1^{(0)}\longrightarrow(\cW_2^{(1/2)}\tens\cW_2^{(-3/2)})\oplus(\cW_2^{(1/2)}\tens\cW_2^{(3/2)})\longrightarrow 0.
\end{equation*}
Thus using Theorem \ref{thm:Bp_tensor_products}(1), $\cW_2^{(1/2)}\tens\cQ_1^{(0)}$ has six composition factors:
\begin{equation*}
\cW_2^{(1/2)}, \cW_2^{(1/2)}, \cW_1^{(-1)}, \cW_1^{(2)},\cW_3^{(-1)},\cW_3^{(2)}.
\end{equation*}
The first four are the composition factors of $\cQ_2^{(1/2)}$,  and the last two are projective. So the complement of the direct summand $\cQ_2^{(1/2)}\subseteq\cW_2^{(1/2)}\tens\cQ_1^{(0)}$ is a direct sum $\cW_3^{(-1)}\oplus\cW_3^{(2)}$.
\end{proof}

Via associativity, the above proposition leads to formulas for further tensor products involving projective modules, which are stated in \cite[Equation 5.27]{CR1}. We now turn to tensor products of the irreducible projective $L_{-4/3}(\mathfrak{sl}_2)$-modules. For $\lambda,\mu\in\CC$ such that $\lambda+\mu\in\frac{2}{3}\ZZ$, Theorem \ref{thm:Bp_tensor_products}(3) implies
\begin{align*}
\cE_\lambda\tens\cE_\mu & = \cE^{(0)}_{\overline{\sqrt{2/3}(1-3\lambda/4)}}\tens\cE^{(0)}_{\overline{\sqrt{2/3}(1-3\mu/4)}} \cong \left\lbrace\begin{array}{lll}
\cQ_1^{(0)}\oplus\cQ_3^{(0)} & \text{if} & \lambda+\mu\in 2\ZZ\\
\cQ_2^{(1/2)}\oplus\cQ_3^{(-1)} & \text{if} & \lambda+\mu\in -\frac{2}{3}+2\ZZ\\
\cQ_3^{(1)}\oplus\cQ_2^{(-1/2)} & \text{if} & \lambda+\mu\in\frac{2}{3}+2\ZZ\\
\end{array}\right. \nonumber\\
& =\left\lbrace\begin{array}{lll}
\cE_0\oplus\cS_0 & \text{if} & \lambda+\mu\in 2\ZZ\\
\sigma^{-1}(\cE_0) \oplus S_{-2/3}^+& \text{if} & \lambda+\mu\in -\frac{2}{3}+2\ZZ\\
\sigma(\cE_0)\oplus \cS_{2/3}^- & \text{if} & \lambda+\mu\in\frac{2}{3}+2\ZZ\\
\end{array} \right. ,
\end{align*}
and for $\lambda,\mu\in\CC$ such that $\lambda+\mu\notin\frac{2}{3}\ZZ$, Theorem \ref{thm:Bp_tensor_products}(4) implies
\begin{align*}
\cE_\lambda\tens\cE_\mu & = \cE^{(0)}_{\overline{\sqrt{2/3}(1-3\lambda/4)}}\tens\cE^{(0)}_{\overline{\sqrt{2/3}(1-3\mu/4)}}\nonumber\\
& \cong \cE^{(-1)}_{\overline{\sqrt{2/3}(2-3(\lambda+\mu)/4)}}\oplus\cE^{(0)}_{\overline{\sqrt{2/3}(1-3(\lambda+\mu)/4)}}\oplus\cE^{(1)}_{\overline{\sqrt{2/3}(-3(\lambda+\mu)/4)}}\nonumber\\
&=\sigma^{-1}(\cE_{\lambda+\mu-4/3})\oplus\cE_{\lambda+\mu}\oplus\sigma(\cE_{\lambda+\mu+4/3}),
\end{align*}
in agreement with \cite[Equation 5.22]{CR1}.

\subsubsection{Bershadsky-Polyakov algebras at levels \texorpdfstring{$-\frac{5}{3}$}{-5/3} and \texorpdfstring{$-\frac{9}{4}$}{-9/4}}

The Bershadsky-Polyakov algebras \cite{Po, Be} are the subregular affine $W$-algebras associated to $\mathfrak{sl}_3$. Thus the $\frac{1}{2}\ZZ$-graded vertex operator algebras $\cB_2^3$ and $\cB_4$ are Bershadsky-Polyaokov algebras at the admissible levels $-3+\frac{4}{3}=-\frac{5}{3}$ and $-3+\frac{3}{4}=-\frac{9}{4}$, respectively. As some of the simplest $W$-algebras, Bershadsky-Polyakov algebras at admissible levels have received a lot of attention recently. They are rational when the denominator of the admissible level is $2$ \cite{Ar-BP}. When the denominator is greater than $2$, they are still quasi-lisse (in the sense of \cite{AK}) by the discussion following \cite[Lemma 6.3]{AK} and thus have finitely many simple grading-restricted modules, consistent with Theorem \ref{thm:grading-restricted}(3).

Simple strongly $\cH$-weight-graded modules for Bershadsky-Polyakov algebras at admissible levels have been constructed and classified in \cite{AK1, AK2, FKR, AKR}. It is also shown in \cite{FKR} that the category of highest-weight modules is semisimple (with finitely many objects), consistent with Theorem \ref{thm:hw_cat_ss}. Thus for the levels $-\frac{5}{3}$ and $-\frac{9}{4}$ covered in this work, Theorem \ref{thm:grad_rest_ss} and Remark \ref{rem:hw_ss} strengthen this result of \cite{FKR}, showing that the category of finite-length generalized modules with highest-weight composition factors is semisimple, even without assuming $h(0)$-semisimplicity. In \cite{FR}, fusion rules of simple strongly $\cH$-weight-graded ($\theta$-twisted) modules for Bershadsky-Polyakov algebras at admissible levels were predicted using a conjectural Verlinde formula. Using our rigid tensor categories of such modules, we can now confirm these conjectures at levels $-\frac{5}{3}$ and $-\frac{9}{4}$. For further work on subregular $W$-algebras in type $A$ beyond Bershadsky-Polyakov algebras, see \cite{Fe}.

Consider the Bershadsky-Polyakov algebra $\cB_2^3$ at level $-\frac{5}{3}$. By Theorem \ref{thm:B2m_grad-rest}, $\cB_2^3$ has nine simple grading-restricted local modules, namely the $\cX_{\overline{r}}^{(\ell)}$ for $r=1,2,3$ and $\ell=-1,0,1$, in agreement with \cite[Section 5.3.1]{FR} (see also \cite[Section 5]{AK1}, where a different conformal vector for $\cB_2^3$ is used, leading to a different enumeration of grading-restricted modules). By Theorem \ref{thm:B2m_grad-rest}(3), only the three simple modules $\cX_{\overline{r}}^{(0)}$ with $r=1,2,3$ are $C_1$-cofinite, and Theorem \ref{thm:B2m_tens_prods}(1) shows that these modules are simple currents with $\ZZ/3\ZZ$ fusion rules. In agreement with \cite[Proposition 5.9]{FR}, these are the fusion rules of the  simple modules for the $\mathfrak{sl}_3$-root lattice vertex operator algebra $L_1(\mathfrak{sl}_3)$ (but recall from Proposition \ref{prop:qHr_braided_tensor?} and Remark \ref{rem:qHr_braided_tensor} that it is more natural to compare $C_1$-cofinite $\cB_2^3$-modules with those of $L_{-5/3}(\mathfrak{sl}_3)$).

Perhaps the most interesting remaining fusion rules for $\cB_2^3$ to consider are those of the lower-bounded $\theta$-twisted modules $\cG_{\overline{\lambda}}^{(0)}$ for $\overline{\lambda}\in(\CC\setminus\ZZ)/3\ZZ$. To make our labeling of these modules more similar to that of \cite{FR}, let us introduce alternate notation
\begin{equation*}
\cG^{(\ell)}_{\overline{1/2-3\lambda}}\longleftrightarrow\sigma^\ell(\cE_\lambda),\qquad\cR_{\overline{r+1}}^{(\ell)} \longleftrightarrow\sigma^\ell(\cS_r)
\end{equation*}
for $\ell\in\frac{1}{2}\ZZ$, $\lambda\notin\frac{1}{6}+\frac{1}{3}\ZZ$, and $\overline{r}\in\ZZ/3\ZZ$. Then Theorem \ref{thm:B2m_tens_prods}(2) becomes
\begin{align*}
\cE_\lambda\tens\cE_\mu & = \cG^{(0)}_{1/2-3\lambda}\tens\cG^{(0)}_{1/2-3\mu}\cong\left\lbrace\begin{array}{lll}
\cR_{\overline{3(\lambda+\mu)+1}}^{(0)} & \text{if} & \lambda+\mu\in\frac{1}{3}\ZZ\\
\cG_{\overline{1-3(\lambda-\mu)}}^{(-3/2)}\oplus\cG_{\overline{-3(\lambda+\mu)}}^{(3/2)} & \text{if} & \lambda+\mu\notin\frac{1}{3}\ZZ\\
\end{array}\right.\nonumber\\
& = \left\lbrace\begin{array}{lll}
\cS_{3(\lambda+\mu)} & \text{if} & \lambda+\mu\in\frac{1}{3}\ZZ\\
\sigma^{-3/2}(\cE_{\lambda+\mu-1/6})\oplus\sigma^{3/2}(\cE_{\lambda+\mu+1/6}) & \text{if} & \lambda+\mu\notin\frac{1}{3}\ZZ\\
\end{array}\right.
\end{align*}
for $\lambda,\mu\in\CC\setminus(\frac{1}{6}+\frac{1}{3}\ZZ)$, which is consistent with \cite[Equation 5.40]{FR}. We have also proved the conjecture in \cite{FR} that when $\lambda+\mu\in\frac{1}{3}\ZZ$, the tensor product $\cE_\lambda\tens\cE_\mu$ is the indecomposable projective cover of one of the simple currents in $\cC^1_{\cB_2^3}$.

Finally, we briefly discuss the Bershadsky-Polyakov algebra $\cB_4$ at level $-\frac{9}{4}$; its modules have also been studied in \cite[Section 6]{AK1}, and its fusion rules have been conjectured in \cite[Section 6.3]{FR}. It has six grading-restricted local modules, $\cW_1^{(0)}$, $\cW_1^{(\pm 1)}$, $\cW_2^{(\pm 1/2)}$, and $\cW_3^{(0)}$, which are depicted in \cite[Equation 6.43]{FR}. By Theorem \ref{thm:Bp_grad-rest}(3), only the first of these modules is $C_1$-cofinite, and thus the semisimple tensor category $\cC^1_{\cB_4}$ is trivial (equivalent to the tensor category of finite-dimensional vector spaces). We discuss how the conjectural fusion rules in \cite{FR} correspond to our results in Theorem \ref{thm:Bp_tensor_products}: \cite[Equation 6.48]{FR} corresponds to Theorem \ref{thm:Bp_tensor_products}(4),
\begin{equation*}
\cE_{\overline{\lambda}}^{(\ell)}\tens\cE_{\overline{\mu}}^{(\ell')}\cong \cE_{\overline{\lambda+\mu}}^{(\ell+\ell'-3/2)}\oplus\cE_{\overline{\lambda+\mu-1/\sqrt{2}}}^{(\ell+\ell'-1/2)}\oplus\cE_{\overline{\lambda+\mu}}^{(\ell+\ell'+1/2)}\oplus\cE_{\overline{\lambda+\mu-1/\sqrt{2}}}^{(\ell+\ell'+3/2)}
\end{equation*}
for $\overline{\lambda},\overline{\mu},\overline{\lambda+\mu}\in(\CC\setminus\frac{1}{2\sqrt{2}}\ZZ)/\sqrt{2}\ZZ$; \cite[Equation 6.49]{FR} corresponds to
\begin{equation*}
\cW_3^{(\ell)}\tens\cE^{(\ell')}_{\overline{\lambda}} \cong\cE_{\overline{\lambda+1/\sqrt{2}}}^{(\ell+\ell'-1)}\oplus\cE_{\overline{\lambda}}^{(\ell+\ell')}\oplus\cE_{\overline{\lambda+1/\sqrt{2}}}^{(\ell+\ell'+1)}
\end{equation*}
for $\overline{\lambda}\in(\CC\setminus\frac{1}{2\sqrt{2}}\ZZ)/\sqrt{2}\ZZ$; \cite[Equation 6.50]{FR} corresponds to
\begin{equation*}
\cW_3^{(\ell)}\tens\cW^{(\ell')}_3\cong \cW_1^{(\ell+\ell')}\oplus\cQ_3^{(\ell+\ell')},
\end{equation*}
though for this last equation, \cite{FR} only gives Grothendieck fusion rules, that is, the composition factors of the tensor product module. Next, \cite[Equation 6.52]{FR} corresponds to
\begin{equation*}
\cW_2^{(\ell-1/2)}\tens\cE_{\overline{\lambda}}^{(\ell')}\cong\cE_{\overline{\lambda+1/2\sqrt{2}}}^{(\ell+\ell'-1)}\oplus\cE_{\overline{\lambda-1/2\sqrt{2}}}^{(\ell+\ell')}
\end{equation*}
and \cite[Equation 6.53]{FR} corresponds to
\begin{equation*}
\cW_2^{(\ell+1/2)}\tens\cW_2^{(\ell'-1/2)}\cong \cW_1^{(\ell+\ell')}\oplus \cW_3^{(\ell+\ell')}.
\end{equation*}
Finally, \cite[Equation 6.54]{FR} corresponds to
\begin{equation*}
\cW_3^{(\ell)}\tens\cW_2^{(\ell'-1/2)}\cong\cW_2^{(\ell+\ell'-1/2)}\oplus\cQ_4^{(\ell+\ell'-1/2)}.
\end{equation*}
Thus the Verlinde formula of \cite{FR} for the Bershadsky-Polyakov algebra at level $-\frac{9}{4}$ yields the correct Grothendieck fusion rules.

\section{Rigid tensor supercategories for some principal \texorpdfstring{$W$}{W}-superalgebras}\label{sec:super}

In this section, we describe the categories of finitely-generated strongly $\cH$-weight-graded modules for the Feigin-Semikhatov duals of $\cB_p$ and $\cB_2^{\ZZ/m\ZZ}$. These are affine $W$-superalgebras associated to principal nilpotent elements in $\mathfrak{sl}_{n\vert 1}$ \cite{CGN}: The physicists Kazama and Suzuki originally related the WZW theory of $\mathfrak{sl}_2$ to a conformal field theory with $N=2$ superconformal symmetry \cite{KS}, and then Feigin and Semikhatov realized that a similar idea applies to subregular $W$-algebra of type $A$ \cite{FS}.

\subsection{The Feigin-Semikhatov dual of \texorpdfstring{$\cB_p$}{Bp}}

Let $\mathcal{S}_p$ for $p\in\ZZ_{\geq 3}$ be the Feigin-Semikhatov dual of $\mathcal{B}_p$, that is, it is the coset of a diagonal Heisenberg subalgebra in the tensor product of $\mathcal{B}_p$ with the latttice vertex operator superalgebra $V_\ZZ$. From \cite[Theorem~4.4(2)]{CGN},  $\cS_p$ is the simple principal $W$-superalgebra of $\mathfrak{sl}_{p-1\vert 1}$ at level $-(p-2)+\frac{p}{p-1}$. Specifically,
\begin{equation}\label{eqn:Sp_def}
\cS_p =\mathrm{Com}(\til{\cH}_1,\cB_p\otimes V_\ZZ)
\end{equation}
where $V_\ZZ$ is the lattice vertex operator superalgebra associated to the lattice $\ZZ\phi$ with bilinear form such that $(\phi,\phi)=1$, and $\til{\cH}_1$ is the Heisenberg vertex operator algebra associated to the abelian Lie algebra $\CC\left(-h/\lambda_p + \phi\right)$; recall that $\lambda_p^2=-\frac{p}{2}$.

 The Heisenberg  subalgebra $\til{\cH}_2\subseteq\cB_p\otimes V_\ZZ$ associated to $\CC(\lambda_p h+\phi)$ commutes with $\til{\cH}_1$, so the rank-$2$ Heisenberg subalgebra of $\cB_p\otimes V_\ZZ$ is isomorphic to $\til{\cH}_1\otimes\til{\cH}_2$, and $\cS_p$ is an extension of $\til{\cH}_2\otimes\cM(p)$. We then identify $\til{\cH}_2\xrightarrow{\cong}\cH$ via $\frac{1}{\sqrt{\lambda_p^2+1}}(\lambda_p h+\phi)\mapsto h$, so that $\cS_p$ is an extension of $\cH\otimes\cM(p)$.
\begin{prop}\label{prop:Sp_decomp}
As an $\cH\otimes\cM(p)$-module
\begin{equation}\label{eqn:Sp_decomp}
\mathcal{S}_p \cong \bigoplus_{k \in \ZZ} \cF^\cH_{k\mu_p} \otimes \cM_{1-k,1},
\end{equation}
where $\mu_p \in i\mathbb{R}$ satisfies $\mu_p^2 = \frac{2-p}{2}$. 
\end{prop}
\begin{proof}
In general, for a Heisenberg vertex algebra associated with an abelian Lie algebra $\CC\psi$, we use $\cF^\psi_\lambda$ for $\lambda\in\CC$ to denote the Heisenberg Fock module whose lowest-conformal-weight vector $v_\lambda$ satisfies
\begin{equation*}
\psi(n)v_\lambda = \delta_{n,0}\lambda v_\lambda
\end{equation*}
for $n\in\ZZ_{\geq 0}$. Thus from \eqref{eqn:Sp_def} and \eqref{eqn:Bp_decomp},
\begin{align*}
\cS_p &= {\rm Com}(\widetilde{\cH}_1, \cB_p \otimes V_{\ZZ})= {\rm Com}\bigg(\widetilde{\cH}_1,\bigoplus_{k,n \in \ZZ} \cF_{k\lambda_p}^h \otimes \cM_{1-k,1} \otimes \cF^{\phi}_n\bigg)\\
&= {\rm Com}\bigg(\widetilde{\cH}_1,\bigoplus_{k,n \in \ZZ} \cF_{-k+n}^{-h/\lambda_p+\phi} \otimes \cF^{\lambda_p h+\phi}_{\lambda_p^2 k+n}\otimes \cM_{1-k,1}\bigg)= \bigoplus_{k \in \ZZ} \cF^{\lambda_p h+\phi}_{k(1-p/2)} \otimes \cM_{1-k,1} \\
&=\bigoplus_{k \in \ZZ} \cF^\cH_{k\mu_p} \otimes \cM_{1-k,1},
\end{align*}
where the last step uses the isomorphism $\til{\cH}_2\rightarrow\cH$ in the discussion above.
\end{proof}

We set $\mu_p=i\sqrt{\frac{p-2}{2}}$ for some fixed choice $i$ of square root of $-1$. Then in the notation of Sections \ref{sec:main_results} and \ref{sec:detailed_structure}, 
\begin{equation*}
J=\cF^\cH_{-i\sqrt{(p-2)/2}}\otimes\cM_{2,1},\qquad\lambda_J=-i\sqrt{(p-2)/2},\qquad r_J=1.
\end{equation*}
Thus $\lambda_J^2+\frac{p}{2} r_J^2 =1>0$. The lowest conformal weight of $\cF^\cH_{k\mu_p} \otimes \cM_{1-k,1}$ is
\begin{equation*}
\frac{1}{2}(k\mu_p)^2 + h_{1+\vert k\vert,1}=\frac{1}{2}\vert k\vert(\vert k\vert+p-1),
\end{equation*}
so $\mathcal{S}_p$ is $\ZZ$-graded by conformal weights if $p$ is even and $\frac{1}{2}\ZZ$-graded if $p$ is odd.

\begin{rem}
We can also define the simple current extension \eqref{eqn:Sp_decomp} for $p=2$, in which case we get the affine vertex superalgebra associated to $\mathfrak{sl}_{1\vert 1}$. But the Heisenberg vertex algebra $\cH$ is then commutative since $\mu_2=0$, and we cannot apply the results of Section \ref{sec:main_results} since they require the individual summands of \eqref{eqn:Sp_decomp} to have distinct $h(0)$-eigenvalues. From the tensor category point of view, it is better to define $\mathcal{S}_2$ to be the affine vertex operator superalgebra associated to $\mathfrak{gl}_{1\vert 1}$. But since the tensor structure on the representation category of this superalgebra has already been described in detail \cite{CMY3}, we do not consider the $p=2$ case of \eqref{eqn:Sp_decomp} further in this paper. 
\end{rem}

Similar to $\cB_p$, \eqref{eqn:Sp_decomp} gives the decomposition of $\cS_p$ into $h(0)$-eigenspaces, and $\cS_p=\bigoplus_{\lambda\in\CC} \cS_p^{(\lambda)}$ is a strongly $\CC$-graded vertex operator superalgebra such that $\cS_p^{(\lambda)}=\cF^\cH_{k\mu_p}\otimes\cM_{1-k,1}$ if $\lambda=k\mu_p$ for some $k\in\ZZ$, and $\cS_p^{(\lambda)}=0$ otherwise. Moreover, $\CC/\ZZ$ acts on $\cS_p$ by automorphisms with fixed-point subalgebra $\cH\otimes\cM(p)$, so in particular the finite cyclic orbifold $\cS_p^{\ZZ/m\ZZ}$ is a simple current extension of $\cH\otimes\cM(p)$ for any $m\in\ZZ_+$. For $m=2$, the automorphism group $\ZZ/2\ZZ$ of $\cS_p$ is generated by the involution $\theta$, and $\cS_p^{\ZZ/2\ZZ}$ is the even vertex operator subalgebra of $\cS_p$ (it is also the $\ZZ$-graded subalgebra when $p$ is odd). Thus $\theta$-twisted $\cS_p$-modules form the Ramond sector of the tensor category $\cO_{\cS_p}$.

By Theorem \ref{thm:main_thm}, the category $\cO_{\cS_p}=\Oloc_{\cS_p}\oplus\Otw_{\cS_p}$ is a rigid braided $\ZZ/2\ZZ$-crossed tensor supercategory, and $\Oloc_{\cS_p}$ is a rigid braided tensor subcategory of $\cO_{\cS_p}$ which is ribbon when $p$ is even. The classification of simple objects in $\cO_{\cS_p}$ comes from Theorem \ref{thm:simple_A-module_classification}:

\begin{thm}\label{thm:Sp_mod_class}
Simple objects in $\cO_{\cS_p}$ and $\Oloc_{\cS_p}$ are as follows:
\begin{enumerate}
\item Every simple object in $\cO_{\cS_p}$ is isomorphic to one of the following induced modules:
\begin{equation*}
\cS\cW_s^{(\ell)} :=\cF\bigg(\cF^\cH_{i\ell\sqrt{\frac{2}{p-2}}}\otimes\cM_{1,s}\bigg)
\end{equation*}
for $1\leq s\leq p$, $\ell\in\frac{1}{2}\ZZ$, or
\begin{equation*}
\cS\cE_{\lambda}^{(\ell)}=\cF\bigg(\cF^{\cH}_{i\big(\lambda\sqrt{\frac{p}{p-2}}+\big(\ell-\frac{p-1}{2}\big)\sqrt{\frac{2}{p-2}}\big)}\otimes\cF_{\lambda}\bigg)
\end{equation*}
for $\lambda\in\CC\setminus L^\circ$, $\ell\in\frac{1}{2}\ZZ$.

\item The following are  all isomorphisms between simple modules from part (1): $\cS\cE_\lambda^{(\ell)}\cong\cS\cE_{\lambda-n\sqrt{p}/2}^{(\ell+n)}$
for $\lambda\in\CC\setminus L^\circ$, $\ell\in\frac{1}{2}\ZZ$, and $n\in\ZZ$.

\item The module $\cS\cW_s^{(\ell)}$ is an object of $\Oloc_{\cS_p}$ if and only if $\ell\in\frac{s-1}{2}+\ZZ$, and the module $\cS\cE_\lambda^{(\ell)}$ is an object of $\Oloc_{\cS_p}$ if and only if $\ell\in\frac{p-1}{2}+\ZZ$.
\end{enumerate}

\end{thm}

By Theorem \ref{thm:gen_proj_covers}, indecomposable projective objects of $\cO_{\cS_p}$ are as follows:
\begin{thm}
The modules $\cS\cW_p^{(\ell)}$ and $\cS\cE_{\lambda}^{(\ell)}$ for $\lambda \in \CC\setminus L^{\circ}$ and $\ell\in\frac{1}{2}\ZZ$ are projective in $\cO_{\mathcal{S}_p}$. For $1\leq s \leq p-1$ and $\ell\in\frac{1}{2}\ZZ$, $\cS\cW_{s}^{(\ell)}$ has a projective cover
 \begin{equation*}
  \cS\mathcal{Q}_{s}^{(\ell)}=\cF\bigg(\cF^\cH_{i\ell\sqrt{\frac{2}{p-2}}}\otimes\cP_{1,s}\bigg)
 \end{equation*}
in $\cO_{\mathcal{S}_p}$, which has Loewy diagram
\begin{equation*}\label{eqn:Qsl_Loewy_diag_Sp}
 \begin{matrix}
  \begin{tikzpicture}[->,>=latex,scale=1.5]
\node (b1) at (1,0) {$\cS\cW_{s}^{(\ell)}$};
\node (c1) at (-1.5, 1){$\cS\mathcal{Q}_{s}^{(\ell)}$:};
   \node (a1) at (0,1) {$\cS\cW_{p-s}^{(\ell-(p-2)/2)}$};
   \node (b2) at (2,1) {$\cS\cW_{p-s}^{(\ell+(p-2)/2)}$};
    \node (a2) at (1,2) {$\cS\cW_{s}^{(\ell)}$};
\draw[] (b1) -- node[left] {} (a1);
   \draw[] (b1) -- node[left] {} (b2);
    \draw[] (a1) -- node[left] {} (a2);
    \draw[] (b2) -- node[left] {} (a2);
\end{tikzpicture}
\end{matrix} .
 \end{equation*}
\end{thm}

By Theorem \ref{thm:grading-restricted}, all objects of $\cO_{\cS_p}$ are grading restricted. Moreover, by Theorem \ref{thm:C_1-cofin} and Corollary \ref{cor:C1_cofin}, the category $\cC_{\cS_p}^1$ of $C_1$-cofinite grading-restricted generalized $\cS_p$-modules equals the category of finite-length grading-restricted generalized $\cS_p$-modules, and $\cC^1_{\cS_p}$ is a rigid braided tensor supercategory. Tensor products of simple modules in $\cO_{\cS_p}$ (and in $\cC^1_{\cS_p}$) are given by Theorem \ref{thm:general_fusion_rules} and are similar to those for $\cB_p$ in Theorem \ref{thm:Bp_tensor_products}. In fact, the formulas in Theorem \ref{thm:Bp_tensor_products}(1), (2), and (4) hold for $\cS_p$ as well once we make the changes $\cW_s^{(\ell)}\mapsto\cS\cW_s^{(\ell)}$ and $\cE^{(\ell)}_{\overline{\lambda}}\mapsto\cS\cE^{(\ell)}_{\lambda}$. The formula in Theorem \ref{thm:Bp_tensor_products}(3) should be modified in the spectral flow indices to account for the isomorphisms in Theorem \ref{thm:Sp_mod_class}(2):
\begin{align*}
\cS\cE_{\lambda}^{(\ell)}\tens\cS\cE_{\mu}^{(\ell')}&\cong\bigoplus_{\substack{t=s\\ t\equiv s\,(\mathrm{mod}\,2)}}^p \cS\mathcal{Q}_{t}^{(\ell+\ell'-r+(s+1)/2)}\oplus \bigoplus_{\substack{t=p-s+2\\ t\equiv p-s\,(\mathrm{mod}\,2)}}^p \cS\mathcal{Q}_{t}^{(\ell+\ell'-r-(p-s-3)/2)}
\end{align*}
for $\ell,\ell'\in\frac{1}{2}\ZZ$ and $\lambda,\mu\in\CC\setminus L^\circ$ such that $\lambda+\mu=\alpha_0+\alpha_{r,s}$.

\subsection{The Feigin-Semikhatov dual of \texorpdfstring{$\cB_2^{\ZZ/m\ZZ}$}{B2m}}

Similar to the definition of $\cS_p$, we can define a Feigin-Semikhatov dual of $\cB_2^m=\cB_2^{\ZZ/m\ZZ}$ for $m\geq 2$:
\begin{equation*}
S_2^m = {\rm Com}(\widetilde{\cH}_1, \cB_2^m\otimes V_{\ZZ}) 
\end{equation*} 
where $V_{\ZZ}$ is the lattice vertex operator superalgebra associated to $\ZZ\phi$ with $(\phi, \phi) = 1$, and $\widetilde{\cH}_1$ is the Heisenberg vertex operator algebra associated to the abelian Lie algebra $\CC\big(\frac{1}{m}ih + \phi\big)$ (here $i$ is a square root of $-1$). By \cite[Theorem~4.4(2)]{CGN}, $\cS_2^m$ is the principal $W$-superalgebra associated to $\mathfrak{sl}_{m\vert 1}$ at level $-(m-1)+\frac{m}{m+1}$.

 The Heisenberg subalgebra $\widetilde{\cH}_2\subseteq\cB_2^m\otimes V_\ZZ$ associated to $\CC(imh+\phi)$ commutes with $\til{\cH}_1$, so $\cB_2^m\otimes V_\ZZ$ contains a rank-$2$ Heisenberg subalgebra isomorphic to $\til{\cH}_1\otimes\til{\cH}_2$, and $\cS_2^m$ is an extension of $\til{\cH}_2\otimes\cM(2)$. We then identify $\til{\cH}_2\xrightarrow{\cong}\cH$ via $\frac{1}{i\sqrt{m^2-1}}(imh+\phi)\mapsto h$, so that $\cS_2^m$ is an extension of $\cH\otimes\cM(2)$.
 
\begin{prop}\label{prop:sl1m}
As an $\cH\otimes\cM(2)$-module,
\begin{equation}
\cS_2^m \cong \bigoplus_{k \in \ZZ} \cF^{\cH}_{ik\sqrt{m^2-1}} \otimes \cM_{1-km,1}.
\end{equation}
\end{prop}
\begin{proof}
Using notation for Heisenberg Fock modules as in the proof of Proposition \ref{prop:Sp_decomp},
\begin{align*}
\cS_2^m & = {\rm Com}(\widetilde{\cH}_1, \cB_2^m\otimes V_{\ZZ}) = {\rm Com}\bigg(\widetilde{\cH}_1, \bigoplus_{k,n\in\ZZ} \cF^{h}_{ikm}\otimes\cM_{1-km,1}\otimes \cF^{\phi}_n\bigg) \\
& = {\rm Com}\bigg(\widetilde{\cH}_1, \bigoplus_{k,n\in\ZZ} \cF^{\frac{1}{m}ih+\phi}_{-k+n}\otimes \cF^{imh+\phi}_{-km^2+n}\otimes\cM_{1-km,1}\bigg)  = \bigoplus_{k\in\ZZ} \cF^{imh+\phi}_{-km^2+k} \otimes  \cM_{1-km,1}\\
& = \bigoplus_{k \in \ZZ} \cF^{\cH}_{ik\sqrt{m^2-1}} \otimes \cM_{1-km,1},
\end{align*}
where the last step uses the isomorphism of $\til{\cH}_2$ and $\cH$ discussed above.
\end{proof}

In the notation of Sections \ref{sec:main_results} and \ref{sec:detailed_structure}, for $\cS_2^m$ we have
\begin{equation*}
J=\cF^\cH_{-i\sqrt{m^2-1}}\otimes\cM_{m+1,1},\qquad \lambda_J=-i\sqrt{m^2-1},\qquad r_J=m.
\end{equation*}
Thus $\lambda_J^2+\frac{p}{2} r_J^2= 1 >0$. The lowest conformal weight of $\cF^\cH_{ik\sqrt{m^2-1}}\otimes\cM_{1-km}$ is
\begin{equation*}
\frac{1}{2}(ik\sqrt{m^2-1})^2+h_{1+\vert k\vert m,1} =\frac{1}{2}\vert k\vert(\vert k\vert +m),
\end{equation*}
so $\cS_2^m$ is $\ZZ$-graded by conformal weights if $m$ is odd and $\frac{1}{2}\ZZ$-graded if $m$ is even. The automorphism $\theta$ of $\cS_2^m$ which acts by $(-1)^n$ on the simple current $J^n$ is the parity involution of $\cS_2^m$, and thus $\theta$-twisted $\cS_2^m$-modules form the Ramond sector of the tensor category $\cO_{\cS_2^m}$.

By Theorem \ref{thm:main_thm}, $\cO_{\cS_2^m}=\Oloc_{\cS_2^m}\oplus\Otw_{\cS_2^m}$ is a rigid braided $\ZZ/2\ZZ$-crossed tensor supercategory, and $\Oloc_{\cS_2^m}$ is a rigid braided tensor subcategory of $\cO_{\cS_2^m}$ which is ribbon when $m$ is odd. Theorem \ref{thm:simple_A-module_classification} yields the classification of simple objects in $\cO_{\cS_2^m}$, similar to the classification for $\cB_2^m$:

\begin{thm}\label{thm:S2m_mod_class}
Simple objects in $\cO_{\cS_2^m}$ and $\Oloc_{\cS_2^m}$ are as follows:
\begin{enumerate}
\item Every simple object in $\cO_{\cS_2^m}$ is isomorphic to one of the following induced modules:
 \begin{align*}
&  \cS\cX_{r}^{(\ell)} =\cF\left(\cF^\cH_{i(m(1-r)+\ell)/\sqrt{m^2-1}}\otimes\cM_{r,1}\right)
\end{align*}
for $r\in\ZZ$, $\ell\in\frac{1}{2}\ZZ$, or
\begin{align*}
& \cS\cG_{\lambda}^{(\ell)}=\cF\left(\cF^{\cH}_{i(m(\lambda-1/2)+\ell)/\sqrt{m^2-1}}\otimes\cF_{\lambda}\right)
 \end{align*}
for $\lambda\in\CC\setminus\ZZ$, $\ell\in\frac{1}{2}\ZZ$.

\item The following are  all isomorphisms between simple modules from part (1): 
\begin{align*}
\cS\cX_r^{(\ell)} & \cong \cS\cX_{r+mn}^{(\ell+n)},\qquad\quad r,n\in\ZZ, \,\,\, \ell\in\frac{1}{2}\ZZ,\\
\cS\cG^{(\ell)}_\lambda & \cong\cS\cG_{\lambda-mn}^{(\ell+n)}, \qquad\quad \lambda\in\CC\setminus \ZZ,\,\,\, \ell\in\frac{1}{2}\ZZ,\,\,\, n\in\ZZ.
\end{align*}

\item The module $\cS\cX_r^{(\ell)}$ is an object of $\Oloc_{\cS_2^m}$ if and only if $\ell\in\ZZ$, and the module $\cS\cG_\lambda^{(\ell)}$ is an object of $\Oloc_{\cS_2^m}$ if and only if $\ell\in\frac{m}{2}+\ZZ$.
\end{enumerate}

\end{thm}

\begin{rem}
Any simple object of $\cO_{\cS_2^m}$ is isomorphic to precisely one the simple modules $\cS\cX_r^{(\ell)}$ with $1\leq r\leq m$ or $\cS\cG_\lambda^{(\ell)}$ with $0\leq\mathrm{Re}(\lambda)<m$. But it is useful label the simple objects of $\cO_{\cS_2^m}$ redundantly (keeping in mind the isomorphisms of Theorem \ref{thm:S2m_mod_class}(2)) in order to obtain uniform descriptions of projective modules and uniform tensor product formulas (as in Theorems \ref{thm:gen_proj_covers} and \ref{thm:general_fusion_rules}).
\end{rem}

As with previous examples, Theorem \ref{thm:gen_proj_covers} gives the indecomposable projective objects of $\cO_{\cS_2^m}$. In particular, $\cS\cG^{(\ell)}_{\lambda}$ is projective for $\lambda\in\CC\setminus\ZZ$, $\ell\in\frac{1}{2}\ZZ$, while 
\begin{equation*}
\cS\mathcal{R}_{r}^{(\ell)}:=\cF\left(\cF^\cH_{i(m(1-r)+\ell)/\sqrt{m^2-1}}\otimes\cP_{r,1}\right)
\end{equation*}
for $r\in\ZZ$, $\ell\in\frac{1}{2}\ZZ$ is a length-$4$ projective cover of $\cS\cX_r^{(\ell)}$. The Loewy diagram of $\cS\cR_r^{(\ell)}$ is the same as the Loewy diagram \eqref{eqn:Qsl_Loewy_diag:orbifold2} for $\cB_2^m$-modules with the substitutions $\cX_{\overline{r}}^{(\ell)}\mapsto\cS\cX_r^{(\ell)}$ and $\cX_{\overline{r\pm 1}}^{(\ell\pm m)}\mapsto\cS\cX_{r\pm 1}^{(\ell\pm m)}$. By Theorem \ref{thm:general_fusion_rules},  tensor products of simple modules in $\cO_{\cS_2^m}$ are also similar to those for $\cB_2^m$. Specifically,
\begin{equation*}
\cS\cX_r^{(\ell)}\tens\cS\cX_{r'}^{(\ell')}\cong\cS\cX_{r+r'-1}^{(\ell+\ell')},\qquad\qquad\cS\cX_r^{(\ell)}\tens\cS\cG_\lambda^{(\ell')}\cong \cS\cG_{\lambda-r+1}^{(\ell+\ell')}
\end{equation*}
for $r,r'\in\ZZ$, $\lambda\in\CC\setminus\ZZ$, $\ell,\ell'\in\frac{1}{2}\ZZ$, and
\begin{equation*}
\cS\cG_\lambda^{(\ell)}\tens\cS\cG_\mu^{(\ell')}\cong\left\lbrace\begin{array}{lll}
\cS\cR_{2-\lambda-\mu}^{(\ell+\ell')} & \text{if} & \lambda+\mu\in\ZZ\\
\cS\cG_{\lambda+\mu}^{(\ell+\ell'-m/2)}\oplus\cS\cG_{\lambda+\mu-1}^{(\ell+\ell'+m/2)} & \text{if} & \lambda+\mu\notin\ZZ\\
\end{array}\right.
\end{equation*}
for $\lambda,\mu\in\CC\setminus\ZZ$, $\ell,\ell'\in\frac{1}{2}\ZZ$. One can use the isomorphisms of Theorem \ref{thm:S2m_mod_class}(2) to rewrite these formulas in terms of the modules $\cS\cX_r^{(\ell)}$, $\cS\cR_r^{(\ell)}$ with $1\leq r\leq m$ and $\cS\cG_\lambda^{(\ell)}$ with $0\leq\mathrm{Re}(\lambda)<m$, if desired.

Finally, Theorem \ref{thm:C_1-cofin} and Corollary \ref{cor:C1_cofin} imply that the supercategory $\cC_{\cS_2^m}^1$ of $C_1$-cofinite grading-restricted generalized $\cS_2^m$-modules equals the supercategory of finite-length grading-restricted generalized $\cS_2^m$-modules. Moreover, $\cC^1_{\cS_2^m}$ is a rigid braided tensor supercategory that contains $\Oloc_{\cS_2^m}$.

\end{document}